\documentclass [11pt]{amsart}
\usepackage{xcolor,bm}
\usepackage{amssymb,amsmath,amsthm,amsfonts}
\usepackage{enumerate}
\usepackage{setspace}
\usepackage{empheq}
\usepackage[all]{xy}
\usepackage{verbatim}
\usepackage[normalem]{ulem}
\usepackage{todonotes}

\usepackage{hyperref}  
\hypersetup{linkcolor=blue,citecolor=blue,filecolor=blue,urlcolor=blue} 

\newcommand{\blue}{}

\numberwithin{equation}{section}
\numberwithin{figure}{section}

\newtheorem{theorem}{Theorem}[section]
\newtheorem{proposition}[theorem]{Proposition}

\newtheorem{lemma}[theorem]{Lemma}

\newtheorem{corollary}[theorem]{Corollary}

\newtheorem{lem}{Lemma}[section]

\theoremstyle{definition}
\newtheorem{definition}[theorem]{Definition}

\newtheorem{remark}[theorem]{Remark}

\usepackage[
backend=biber,
style=alphabetic,
]{biblatex}

\definecolor{myblue}{rgb}{0.6, 0.9, 1}

\makeatletter

\newcommand{\Rmnum}[1]{\expandafter\@slowromancap\romannumeral #1@}
\makeatother

\definecolor{myblue}{rgb}{0.6, 0.9, 1}
\definecolor{mygreen}{rgb}{0,0,1}
\definecolor{purple}{rgb}{0.6,0.2,1}
\definecolor{orange}{rgb}{0.8,0,0.2}

\newcommand{\eps}{\varepsilon}

\newcommand{\la}{\lambda}

\newcommand{\Res}{\operatorname{Res}}

\newcommand{\Pic}{\operatorname{Pic}}

\newcommand{\Z}{\mathbb{Z}}

\newcommand{\bb}{\mathbb}

\newcommand{\Ph}{\operatorname{Ph}}
\newcommand{\cal}{\mathcal}
\newcommand{\ovl}{\overline}
\newcommand{\Fal}{\mathrm{Fal}}
\newcommand{\h}{\widehat{h}}
\newcommand{\Spec}{\operatorname{Spec}}
\newcommand{\Rat}{\operatorname{Rat}}

\newcommand{\trdeg}{\operatorname{trdeg}}
\newcommand{\Supp}{\operatorname{Supp}}
\newcommand{\vphi}{\varphi}

\newcommand{\GL}{\operatorname{GL}}
\newcommand{\berkP}{\mathbb{P}^{1,\operatorname{an}}}

\newcommand{\rat}{\operatorname{rat}}

\newcommand{\wh}{\widehat}

\newcommand{\triv}{\operatorname{triv}}
\newcommand{\End}{\operatorname{End}}
\newcommand{\gal}{\operatorname{gal}}
\newcommand{\Div}{\operatorname{Div}}
\newcommand{\crit}{\operatorname{crit}}

\newcommand{\Prep}{\operatorname{Prep}}

\newcommand{\diam}{\operatorname{diam}}
\newcommand{\Zar}{\operatorname{Zar}}

\newcommand{\divv}{\operatorname{div}}

\newcommand{\Lyap}{\mathrm{Lyap}}
\newcommand{\id}{\operatorname{id}}
\newcommand{\an}{\operatorname{an}}

\theoremstyle{definition} 

\begin{document}
\title{Non-Archimedean Rigidity and the Uniform Dynamical Manin--Mumford for $\left(\bb{P}^1\right)^n$}

\author{Zhuchao Ji}
\address{Institute for Theoretical Sciences, Westlake University, 600 Dunyu Road, Hangzhou, Zhejiang 310030, China}
\email{jizhuchao@westlake.edu.cn}

\author{Jit Wu Yap}
\address{Department of Mathematics, National University of Singapore, 10 Lower Kent Ridge Road, Singapore 119076}
\email{yapjitwu@nus.edu.sg}

\date{}

\begin{abstract}
We prove a uniform form of the dynamical Manin–Mumford conjecture for split endomorphisms $\Phi = \left(f_1,\ldots,f_n\right): \left(\mathbb{P}^1\right)^n \to \left(\mathbb{P}^1\right)^n$ over $\mathbb{C}$, where each $f_i$ has degree $d \ge 2$ and at most one $f_i$ is exceptional. Specifically, we bound the number of irreducible components of $\overline{\Prep\left(\Phi\right)\cap V\left(\bb{C}\right)}^{\mathrm{Zar}}$ 
uniformly, depending only on $\left(n,d,\deg V\right)$. We also obtain uniform statements for the dynamical Bogomolov conjecture. Our approach combines a degeneration argument via ultrafilters with a new non-archimedean rigidity theorem, along with the quantitative dynamical Zhang's inequality by Mavraki and the second author.
\end{abstract}

	\maketitle

\section{Introduction}
\subsection{The main results}
Let $\Phi = \left(f_1,\ldots,f_n\right): \left(\mathbb{P}^1\right)^n \to \left(\mathbb{P}^1\right)^n$ be a split endomorphism of $\left(\bb{P}^1\right)^n$ over $\mathbb{C}$ with $\deg f_i = d \geq 2$. {\blue Assume that at most one of the maps $f_i$ is exceptional.} {\blue Here non-exceptional means} not conjugate to a monomial, Chebyshev or Latt{\`e}s map. Then the dynamical Manin--Mumford  conjecture \cite{Zha06, GTZ11}, proven in this setting by Ghioca--Nguyen--Ye \cite{GNY18, GNY19}, with a gap fixed by  Mavraki--Schmidt--Wilms \cite{MSW23}, states that any subvariety $V \subseteq \left(\bb{P}^1\right)^n$ with a Zariski dense set of preperiodic points for $\Phi$ must itself be preperiodic for $\Phi$. The reader may also see \cite{DF17, DFR23} for some cases of the dynamical Manin--Mumford conjecture for maps over $\bb{P}^2$. 
\par 
Ghioca--Nguyen--Ye's proof follows the strategy introduced by Ullmo \cite{Ulm98} and Zhang \cite{Zha98} in their proofs of the Bogomolov conjecture. Applying Yuan's equidistribution theorem for small points \cite{Yua08}, they show  that if $C \subseteq \left(\bb{P}^1\right)^2$ is a curve with infinitely many preperiodic points, then 
\[\left(\deg \pi_2\right) \pi_1^*\mu_{f_1} = \left(\deg \pi_1\right) \pi_2^* \mu_{f_2}\] 
as measures on $C\left(\bb{C}\right)$, where $\mu_{f_i}$ is the equilibrium measure of $f_i$. The proof is then completed by establishing that equality of measures forces $C$ to be preperiodic under $\left(f_1,f_2\right)$. This is a rigidity statement for the equilibrium measures $\mu_{f_1},\mu_{f_2}$, belonging to a wider class of rigidity phenomena studied by various authors \cite{Ere89,Lev90,LP97,Din00,Ye15, DFG23,JX25, MS25}.
\par 
Our first main result is a non-archimedean analogue of Ghioca--Nguyen--Ye's rigidity theorem. 

\begin{theorem} \label{IntroRigidity1}
Let $f,g: \bb{P}^1 \to \bb{P}^1$ be two rational maps of degree $d \ge 2$ defined over $K$ where $K$ is an algebraically closed complete valued non-archimedean field of residue characteristic $0$. Let $C \subseteq \bb{P}^1 \times \bb{P}^1$ be an irreducible curve over $K$ that projects dominantly to both coordinates. Assume that $\Lyap\left(f\right) > 0$. Then if $\left(\deg \pi_2\right) \pi_1^*\mu_f = \left(\deg \pi_1\right) \pi_2^*\mu_g$, the curve $C$ must be preperiodic for $\left(f,g\right)$.  
\end{theorem}

{\blue In fact, we prove a stronger statement in Theorem \ref{Rigidity1}: the assumption $\deg f=\deg g$ is unnecessary, and it suffices that $\pi_1^*\mu_f$ and $\pi_2^*\mu_g$ are proportional on an open subset $U\subseteq C^{\an}$ with $\left(\pi_1^*\mu_f\right)\left(U\right)>0$. The conclusion is that $C$ is preperiodic under $\left(f^a,g^b\right)$ for some positive integers $a,b$ satisfying $\left(\deg f\right)^a=\left(\deg g\right)^b$.}
\par
The condition that $\Lyap\left(f\right) > 0$ is necessary for Theorem \ref{IntroRigidity1}. If $\Lyap\left(f\right) = 0$, and $C = \Delta$ is the diagonal, the classification in \cite{FRL25} imply that any $g$ which is small enough perturbation of $f$ will satisfy $\mu_f = \mu_g$.  
\par 
When $f$ and $g$ are polynomials and $C = \Delta$, Theorem \ref{IntroRigidity1} was recently proved by Chen Gong and the second-named author \cite{GY26} using a relatively elementary argument based on a non-archimedean analogue of the polynomial trees constructed by DeMarco and McMullen \cite{DM08}. Our approach is different: it relies on ergodic properties of the equilibrium measures $\mu_f,\mu_g$ which were established recently only by Favre and Rivera-Letelier \cite{FRL25}. 

\par  A non-archimedean dynamical rigidity theorem for a place $v$ may be viewed as a rigidity theorem for families of dynamical systems, when viewed through the local geometry determined at $v$. Using ultrafilter degeneration techniques by Luo \cite{Luo21, Luo22} and Favre--Gong \cite{FG24}, we show that Theorem \ref{IntroRigidity1} implies the uniform dynamical Manin-Mumford conjecture for $\left(\bb{P}^1\right)^2$.  
\begin{theorem} \label{IntroUniformCurve1}
Let $f,g: \bb{P}^1 \to \bb{P}^1$ be two rational maps of degree $d \ge 2$ over $\mathbb{C}$ and let $C \subseteq \left(\bb{P}^1\right)^2$ be an irreducible curve of degree $D$. Then there exists a constant $B = B\left(d,D\right) > 0$ so that either 
\[C \text{ is } \left(f,g\right)\text{-special or } \left|\Prep\left(f,g\right) \cap C\right| \leq B.\]
\end{theorem}

When $f,g$ are non-exceptional, special curves are simply the preperiodic curves. {\blue The definition of an $\left(f,g\right)$-special curve is given in Definition~\ref{def:SpecialCurves}. } Under the same philosophy, in the polynomial case, Gong--Yap  \cite{GY26} showed that their non-archimedean rigidity theorem implies a uniform bound on the number of common preperiodic points.

{\blue Taking $C$ to be the diagonal in Theorem~\ref{IntroUniformCurve1}, we resolve the conjecture of DeMarco--Krieger--Ye \cite[Conjecture 1.4]{DKY21} on a uniform bound for common preperiodic points.
\begin{theorem}\label{IntroCommonPreperiodic1}
For every integer $d\geq2$, there exists a constant $B=B(d)>0$ such that, for any two rational maps $f,g:\bb{P}^1\to\bb{P}^1$ of degree $d$ over $\bb{C}$, either
\[
\Prep(f)=\Prep(g)
\quad\text{or}\quad
\left|\Prep(f)\cap\Prep(g)\right|\leq B.
\]
\end{theorem}}

Using ideas of David--Philippon \cite{DP99} in the setting of $\bb{G}_m^n$ and specialization arguments by Mavraki--Yap \cite{MY25}, we can extend Theorem \ref{IntroUniformCurve1} to the case of $\left(\bb{P}^1\right)^n$ for general $n \geq 2$. {\blue This gives us the following theorem, which can be viewed as a uniform dynamical Manin--Mumford on $\left(\bb{P}^1\right)^n$ for general $n\geq2$, allowing at most one exceptional coordinate map.}  

\begin{theorem} \label{IntroTheorem1}
Let $\Phi = \left(f_1,\ldots,f_n\right)$ be a split endomorphism of $\left(\bb{P}^1\right)^n$ defined over $\mathbb{C}$ such that each $f_i$ is of the same degree $d \ge 2$. {\blue Assume that at most one of the maps $f_i$ is exceptional.} Let $V \subseteq \left(\bb{P}^1\right)^n$ be an irreducible subvariety, defined over $\mathbb{C}$, of degree $D$. Then there exists $N = N\left(n,d,D\right) > 0$ such that 
\[\ovl{\Prep\left(\Phi\right) \cap V\left(\bb{C}\right)}^{\Zar} \text{ is a union of at most } N \text{ irreducible preperiodic subvarieties.}\]
\end{theorem}

As is the usual case with Arakelov-theoretic methods, we obtain analogous uniform Bogomolov statements. In particular, we obtain a uniform height gap for hypersurfaces $H \subseteq \left(\bb{P}^1\right)^n$, analogous to the new gap principle by Gao--Ge--K{\"u}hne \cite{GGK21}. Given a split endomorphism $\Phi: \left(\bb{P}^1\right)^n \to \left(\bb{P}^1\right)^n$ of rational maps of degree $d \ge 2$, we say that $V$ is weakly $\Phi$-special if any of the following holds:

\begin{enumerate}
\item $V$ is preperiodic for $\Phi$;
\item $V$ does not project dominantly to one of the coordinates;
\item $V$ is of the form $\pi^{-1}\left(W\right)$ where $\pi: \left(\bb{P}^1\right)^n \to \left(\bb{P}^1\right)^{n-1}$ is the projection away from a coordinate and $W$ a closed subvariety of $\left(\bb{P}^1\right)^{n-1}$.
\end{enumerate}

\begin{theorem} \label{IntroTheorem2}
Let $\Phi = \left(f_1,\ldots,f_n\right)$ be a split endomorphism of $\left(\bb{P}^1\right)^n$ defined over $\ovl{\bb{Q}}$ such that {\blue each $f_i$ has degree $d \geq 2$ and at least one $f_i$ is non-exceptional}. Let $H \subseteq \left(\bb{P}^1\right)^n$ be {\blue an irreducible hypersurface of degree $D$} which is not weakly $\Phi$-special that is defined over $\ovl{\bb{Q}}$. Then there exists $c = c\left(n,d,D\right) > 0$ such that 
\[\h_{\Phi}\left(H\right) \geq c \max\left\{h_{\rat_d}\left(\Phi\right),1\right\}.\]    
\end{theorem}

Here, $\h_{\Phi}$ denotes the canonical height of $\Phi$ with respect to the line bundle $L$ and $h_{\rat_d}\left(\Phi\right) = \sum_{i=1}^{n} h_{\rat_d}\left(f_i\right)$, where by abuse of notation $h_{\rat_d}$ is a Weil height for the moduli space $\rat_d$ of degree $d$ rational maps up to equivalence \cite{Sil98}. 
\par 
\begin{remark}
Shortly after we proved Theorem \ref{IntroRigidity1}, DeMarco and Mavraki \cite{DM26b} informed us that they have a strategy to deduce non-degeneracy results from our non-archimedean rigidity theorem, and in particular obtain different proofs of Theorems \ref{IntroUniformCurve1} and \ref{IntroTheorem1}, as well as relative versions. As our strategies for obtaining uniformity is different from theirs, we decided to write our results separately. 
\end{remark} 

\begin{remark}
One major technical difficulty in implementing the strategy of \cite{GY26} to the setting of Theorem \ref{IntroUniformCurve1} is the lack of understanding of the loci of triplets $\left(f,g,C\right)$ where $C$ is preperiodic for $\left(f,g\right)$. In \cite{GY26}, the authors only deal with the case where $(f,g)$ are polynomials and $C = \Delta$ is the diagonal, and in particular this loci is known to be almost always a Zariski closed set. However in our setting, we a priori only know that it is a countable union of Zariski closed sets. This makes the ultrafilter degeneration method much more technical. In \cite{DM26b}, DeMarco and Mavraki are able to use their non-degeneracy results to characterize the loci \cite[Theorem 1.1]{DM26}. As an example, \cite[Theorem 1.2]{DM26b} states that the loci of pairs of non-exceptional rational maps $(f,g)$ in $\Rat_d \times \Rat_d$ such that $\Prep(f) = \Prep(g)$ is Zariski closed. 
\end{remark}

\subsection{Proof Strategy for Theorem \ref{IntroCommonPreperiodic1}} 
There have been many partial results towards Theorem \ref{IntroTheorem1}. The case of $n = 2$ and $V$ being the diagonal $\Delta \subseteq \left(\bb{P}^1\right)^2$ was first studied by DeMarco--Krieger--Ye \cite{DKY20, DKY21} for the quadratic unicritical family and the Legendre family of Latt{\`e}s maps. Poineau \cite{Poi24a, Poi24b} subsequently extended their results to cover all pairs of Latt{\`e}s maps. In this case, one is looking for a uniform bound on the size of $\Prep\left(f,g\right) \cap \Delta = \Prep\left(f\right) \cap \Prep\left(g\right)$. For general families of rational maps, DeMarco--Mavraki \cite{DM24} showed the existence of a Zariski open $U \subseteq \Rat_d \times \Rat_d$ and $N = N\left(d\right) > 0$ depending only on $d$ so that $\left|\Prep\left(\Phi\right) \cap \Delta \right| \leq N$ for all $\Phi \in U$. Gauthier \cite{Gau24} extended DeMarco--Mavraki results to the case of different degrees. For one-parameter families and $n = 2$, Theorem \ref{IntroTheorem1} was proven by Mavraki--Schmidt \cite{MS25}. 
\par 
Among these previous results, one can summarize two different approaches to proving cases of Theorem \ref{IntroTheorem1} when $n = 2$. By a specialization argument following Mavraki and the second author \cite{MY25} (see Section \ref{sec: Specialization}), it suffices to handle the case where $\Phi$ and $C$ are defined over $\ovl{\bb{Q}}$. The first is to study an adelic energy pairing $\left(f,g\right)$ between two rational maps of degree $d \ge 2$ over $\ovl{\bb{Q}}$, first defined by Favre--Rivera-Letelier \cite{FRL06}, see also Fili \cite{Fil17}. The adelic energy pairing $\left(f,g\right)$ can be written as a sum of local energy pairings 
\[\left(f,g\right) = \sum_{v \in M_K} n_v \left(f,g\right)_v\]
where $K$ is a number field where $f,g$ are defined over and $n_v = \left[K_v:\bb{Q}_p\right]/\left[K:\bb{Q}\right]$. The local energy pairing $\left(f,g\right)_v$ is non-negative and is zero if and only if $\mu_{f,v} = \mu_{g,v}$, where $\mu_{f,v},\mu_{g,v}$ are the equilibrium measure of $f$ and $g$, supported on $\bb{P}_v^{1,\an}$, see \cite{BR10} and \cite{Ben19} for more details of dynamics on $\berkP_v$ and this measure.  
\par 
By the work of DeMarco--Krieger--Ye \cite{DKY20,DKY21}, generalized by Gauthier \cite{Gau24} and Ang--Yap \cite{AY23}, it is known that to obtain uniform bounds on $\left|\Prep\left(f\right) \cap \Prep\left(g\right)\right|$, it suffices to prove that if $\left(f,g\right) \not = 0$, then 
\begin{equation} \label{eq: IntroIneq1}
\left(f,g\right) \geq c \max\left\{h\left(f,g\right),1\right\}
\end{equation}
for some $c = c\left(d\right) > 0$ depending only on the degree $d$. Here, $h\left(f,g\right)$ can be taken as $h\left(f\right) + h\left(g\right)$ for an ample Weil height $h$ on $\Rat_d$. DeMarco--Krieger--Ye shows that \eqref{eq: IntroIneq1} holds if $\left(f,g\right)$ belong to the quadratic unicritical family \cite{DKY21}, or to the family of Latt{\`e}s map associated to the Legendre family \cite{DKY20}. This is also the approach taken by Poineau \cite{Poi24a, Poi24b} for the case of general Latt{\`e}s maps $\left(f,g\right)$.  
\par 
The second approach is to study the non-degeneracy of subvarieties with respect to a bifurcation current. Given a family of polarized dynamical systems $\left(X \to S, \Phi, \cal{L}\right)$ and a closed subvariety $Y \subseteq X$ of codimension $r$ that defines a flat family over $S$, we say that the triple $\left(X,Y,\Phi\right)$ is non-degenerate if we have the non-vanishing
\[\hat{T}_{\Phi}^{k-r+ \dim S} \wedge \left[Y\right] \not = 0\]
where $\hat{T}_{\Phi}$ is a dynamical $\left(1,1\right)$-current associated to our family. One may consult \cite[Section 4]{DM24} for more details. Studying and proving instances of non-degeneracy has been very successful in resolving the uniform Mordell--Lang conjecture \cite{GH19, DGH21, Kuh21, GGK21} and also in dynamical settings \cite{GTV23, DM24, DMY24, MS25, AZ26}.   
\par 
We follow the first approach in directly establishing uniform lower bounds for the global energy pairing $\left\langle f ,g \right\rangle$. In the more general setting of bounding $\left|\Prep\left(f,g\right) \cap C\right|$, we may define a generalization of the adelic energy pairing using the notion of a flat metric \cite{YZ17}. We write it as $\left(f,g\right)_C$ and there is also a local decomposition 
\[\left(f,g\right)_C = \sum_{v \in M_K} n_v \left\langle f , g \right\rangle_{C,v}.\]
Then using the quantitative dynamical Zhang's inequality by Mavraki--Yap \cite{MY25}, if $C$ is of bidegree $\left(d_1,d_2\right)$, it suffices to prove the existence of $c = c\left(d_1,d_2,d\right) > 0$ depending only on $\left(d_1,d_2,d\right)$ such that 
\[\left(f,g\right)_C \geq c \max\left\{ h\left(f,g\right),1\right\}.\]
Observe that the quantitative Zhang's inequality is essential in our approach as we do not establish any form of bigness for the associated line bundles. This is because we are unable to understand when does the exceptional locus, where $C$ is preperiodic for $\left(f,g\right)$, lie in a proper Zariski closed set. In particular, we are unable to apply the height inequalities or equidistribution results from Yuan--Zhang \cite{YZ24} to deduce uniformity. Instead, we have to rely on a pointwise approach which can be handled by the quantitative Zhang's inequality. This is also why we only obtain uniform bounds and not any sort of relative results as in \cite{MS25} and \cite{DM26b}.  

\subsection{Degenerations and Non-Archimedean Rigidity} To prove uniform lower bounds on $\left(f,g\right)_C$, we will prove uniform lower bounds on the local energy pairing $\left\langle f , g \right\rangle_{C,v}$. Let $S$ be a quasi-projective variety over $\ovl{\bb{Q}}$ parametrizing a family of triples $\left(f,g,C\right)$ where $f,g$ are rational maps of degree $d \ge 2$ and $C$ is a degree $\left(d_1,d_2\right)$ curve. In Section \ref{sec: Flat}, we prove that the local energy pairing $\left\langle f , g \right\rangle_{C,s}$ defines a function 
\[\left\langle f , g \right\rangle_{C,s} : S^{\an} \setminus S^{\triv} \to \bb{R}_{\geq 0}\]
where $S^{\an}$ is the Berkovich analytification of $S$ and $S^{\triv}$ is the set of points $s \in S^{\an}$ whose residue field $\cal{H}\left(s\right)$ is trivially valued. Furthermore using properties of the Deligne pairing \cite[Section 4]{YZ24}, the function $\left\langle f , g \right\rangle_{C,s}$ is continuous under some technical hypothesis which we are able to assume in our argument. Now, the locus of points $\left(f,g,C\right) \in S$ for which $C$ is preperiodic for $\left(f,g\right)$ can be given by a countable union of closed subvarieties $V_1,V_2,\ldots$. Ideally, this should be contained in a proper Zariski closed subvariety but it is not clear a priori. This makes the argument more technical than in \cite{GY26}. 
\par 
We now let $\lambda_{V_i}$ be local height functions for $V_i$ (see Section \ref{sec: LocalSubSec}). Let $\lambda_{\partial S}$ be a local height function for the boundary of $S$ and let $\Lyap\left(f\right)$ be the Lyapunov exponent of $f$, which is well-defined for any rational map $f$ of degree $d \ge 2$ over a non-trivially valued field $K$. We then have the following uniform bound (Proposition \ref{UniformEnergy1}):

\begin{proposition} \label{IntroUniformEnergy1}
Assume that the local energy pairing $\left\langle f , g \right\rangle_{C,s}$ is continuous on $S^{\an} \setminus S^{\triv}$. Let $\left(K,\left| \cdot \right|\right)$ be an complete non-trivially valued field of characteristic $0$. Let $\delta, c_1,c_2,\ldots$ be positive constants. Then there exists constants $N > 0$ and $c > 0$ depending on $K,\delta,c_1,c_2,\ldots$ such that if $s = \left(f,g,C\right) \in S\left(K\right)$ satisfies:
\begin{enumerate}
\item $\lambda_{V_i}\left(s\right) \leq c_i \lambda_{\partial S}\left(s\right)$ for all $i \geq 1$;

\item $\Lyap\left(f\right) \geq \delta \lambda_{\partial S}\left(s\right)$;

\item $\lambda_{\partial S}\left(s\right) \geq N$.
\end{enumerate}
Then we have the uniform lower bound 
\[\left\langle f,g \right\rangle_C \geq c \lambda_{\partial Z_k}\left(s\right).\]
\end{proposition}

To prove Proposition \ref{IntroUniformEnergy1}, we use a degeneration arguments by ultrafilters, see Section \ref{sec: DegenUltra}. Degeneration arguments have been previously used to prove uniform bounds in \cite{DKY20} and \cite{Poi24b}. We remark the methods of Poineau \cite{Poi24a, Poi24b} allow for an alternative way to obtain Proposition \ref{IntroUniformEnergy1} given Theorem \ref{IntroRigidity1}. The second author thanks Poineau for sharing this with him. Degenerations of rescaling limits by ultrafilters was first introduced in complex dynamics by Luo \cite{Luo21, Luo22} and formalized using Berkovich spaces by Favre--Gong \cite{FG24}. The second author has also used these technique to obtain uniform bounds on $K$-rational small points for both dynamical systems and abelian varieties \cite{Yap25, LY26}. After degenerating, it suffices to prove the following non-archimedean rigidity statement. We note that using non-archimedean/tropical objects to understand degenerations has shown up in various settings in both dynamics and Diophantine geometry, for example \cite{DF14, Kiw15, DF16, dJ19, Fav20, Fal21, Wil21}. 
\par 
The archimedean analogue of Theorem \ref{IntroRigidity1} was proven by Mavraki--Schmidt \cite{MS25}, building on ideas of Levin--Przytycki \cite{LP97}. A weaker version played a crucial role in the proof of dynamical Manin--Mumford for $\left(\bb{P}^1\right)^n$ by Ghioca--Nguyen--Ye \cite{GNY18,GNY19}. However there are significant difficulties in adapting the archimedean proof to the non-archimedean setting due to the failure of Montel's theorem, see \cite{FKT12} for more details. We are able to overcome this in Section \ref{sec: Rigidity}, adapting the proof of Ji--Xie \cite{JX25} in the archimedean case. On the way, we have to use many facts about non-archimedean dynamics which were recently proven by Favre and Rivera-Letelier \cite{FRL25}. 
\par 
{\blue Using Proposition \ref{IntroUniformEnergy1} and other similar inequalities, we establish Theorem \ref{IntroTheorem2} for $n = 2$ in Section \ref{sec: Curves} and deduce Theorem \ref{IntroUniformCurve1} in Section \ref{sec: Specialization}.} One crucial ingredient is Ingram's result \cite{Ing18} that critical heights are comparable to an ample height on $\rat_d$ outside the Latt{\`e}s locus. This result relies on McMullen's rigidity theorem \cite{McM87} that multipliers of periodic points give rise to a finite morphism of $\rat_d$ to affine space outside of the Latt{\`e}s locus. As a consequence, we are able to prove that the Lyapunov exponents cannot be too small for non-Latt{\`e}s maps, see Theorem \ref{GlobalLyapunov1}. Having such a lower bound on the Lyapunov exponents is necessary for us as Theorem \ref{IntroRigidity1} only applies to maps with positive Lyapunov exponent. 
\par 
To obtain Theorem \ref{IntroTheorem2} for $n > 2$, we will rely on a method used by David--Philippon in their work for uniform bounds on $\bb{G}_m^n$ and $E^g$ \cite{DP02, DP07} that allows us to reduce Theorem \ref{IntroTheorem2} to the case of $n = 2$. This is carried out in Section \ref{sec: HigherDim} and relies on the arithmetic B{\'e}zout theorem. An argument with the quantitative dynamical Zhang's inequality \cite{MY25} gives us Theorem \ref{IntroTheorem1} over $\ovl{\bb{Q}}$. To obtain the case of $\mathbb{C}$, we follow the specialization strategy in \cite{MY25} that was used to deduce the geometric uniform Bogomolov in characteristic $0$ from the case of number fields. This is carried out in Section \ref{sec: Specialization} and finishes the proof of Theorem \ref{IntroTheorem1}.  

\begin{remark}
We make a remark that there is a theory of globally valued fields \cite{Sza23, YDHS24} that axiomatizes the notion of a field $K$ with a global height function for which $\ovl{\bb{Q}}$ with the absolute Weil height is an example of. This is related to the theory of adelic curves of Chen--Moriwaki \cite{CM19, CM21} and also to the theory of Yuan--Zhang \cite{YZ24}. The theory of globally valued fields allows one to degenerate via ultrafilters on sequences of globally valued fields, similar to how we degenerate on sequences of valued fields. The reader may see \cite{Hul26, LLY26} for applications to (semi)-abelian varieties. One can use this theory to obtain an alternative argument for Theorem \ref{IntroTheorem2} in the case of $n  = 2$ using Theorem \ref{IntroRigidity1} through the theory of adelic curves and globally valued fields.
\end{remark}

{\blue \subsection{Organization of the paper}
In Section~\ref{sec: CanonicalHeights}, we recall canonical heights of subvarieties and the quantitative dynamical Zhang inequality. In Section~\ref{sec: LocalHeightsDegenerations}, we recall local height functions and degenerations by using  ultrafilter. In Section~\ref{sec: Rigidity}, we prove the non-archimedean rigidity theorem, including the stronger local form in Theorem~\ref{Rigidity1}, and deduce Theorem~\ref{IntroRigidity1}. In Section~\ref{sec: Flat}, we construct flat metrics in families and establish continuity of local energy pairings. In Section~\ref{sec: EnergyPairing}, we prove uniform local energy bounds, including Proposition~\ref{IntroUniformEnergy1}. In Section~\ref{sec: Curves}, we establish the height gap for curves in Theorem~\ref{HeightGapCurve1}. In Section~\ref{sec: HigherDim}, we prove Theorem~\ref{IntroTheorem2} and deduce uniform dynamical Manin--Mumford over $\ovl{\bb{Q}}$. In Section~\ref{sec: Specialization}, we use a specialization argument to prove Theorem~\ref{IntroTheorem1}, and then deduce Theorems~\ref{IntroUniformCurve1} and~\ref{IntroCommonPreperiodic1}.}

\subsection{Acknowledgements} The authors would like to thank Laura DeMarco, Chen Gong, Myrto Mavraki, Jérôme Poineau, Junyi Xie for helpful discussions regarding the paper. The second author would especially like to thank Myrto Mavraki for inspiring many of the ideas in the proof.  Substantial progress on this work was made during the second author’s visit to Westlake University in December 2025. The hospitality provided during that visit is greatly appreciated. 

\subsection{Funding}
Zhuchao Ji is supported by the National Key R\&D Program of China (No. 2025YFA\\1018300), NSFC Grant (No.12401106), and ZPNSF Grant (No. XHD24A0201). 

\subsection{AI disclosure} 
{\blue The writing of this paper began in January 2026. All ideas and proofs in this paper are due to the authors. GPT was used to check and edit the paper after a draft was written.}

\section{Canonical Heights of Subvarieties} \label{sec: CanonicalHeights}
Let $K$ be a number field and let $X$ be a projective variety over $K$. Let $L$ be an ample line bundle on $X$ and $f: X \to X$ an endomorphism such that $f^*L \simeq L^d$ for some $d \ge 2$. Then $\left(X,f,L\right)$ is a polarized dynamical system and hence by \cite[Section 2]{Zha95b}, there exists an adelic line bundle $\ovl{L}_f$ with underlying line bundle $L$ such that $f^* \ovl{L}_f \simeq \ovl{L}_f^d$. 
\par 
Using the arithmetic intersection theory of Gillet--Soul{\'e} \cite{GS90}, we can define the canonical height of any irreducible subvariety $V \subseteq X_{\ovl{K}}$ as follows. Let $K'$ be any finite extension such that $V$ is well-defined over, and view $\ovl{L}_f$ as an adelic line bundle over $K'$. We then set 
\[h_f\left(V\right) = \frac{\ovl{L}_f^{\dim V + 1} \cdot V}{\left[K':\bb{Q}\right]} \in \bb{R}\]
and the normalized canonical height as 
\[\h_{f}\left(V\right) = \frac{h_f\left(V\right)}{\deg_L\left(V\right)}.\]
This quantity is well-defined independent of the finite extension $K'$. Another way to define $\h_f\left(V\right)$ is to let $V_{\gal}$ be the closed subvariety of $X_K$ corresponding to $V$. We then set 
\[\h_f\left(V\right) = \frac{1}{\left[K:\bb{Q}\right]} \frac{\ovl{L}_f^{\dim V + 1} \cdot V_{\gal}}{\deg_L\left(V_{\gal}\right) }.\]

When $V = x \in X\left(\ovl{K}\right)$ is a point, this recovers the usual definition of the canonical height \cite{CS93} given by 
\[\h_{f}\left(x\right) = \lim_{n \to \infty} \frac{h_L\left(f^n\left(x\right)\right)}{d^n}\]
where $h_L$ is an absolute Weil height for our ample line bundle $L$,  Hence we may call the quantity $\h_f\left(V\right)$ the canonical height of $V$. It satisfies the formula $\h_f\left(f\left(V\right)\right) = d \h_f\left(V\right) \text{ and } \h_f\left(V\right) = 0 \text{ if } V \text{ is preperiodic.}$
\par 
There is also an alternative way to define heights using Chow forms instead of arithmetic intersection theory, developed by Philippon \cite{Phi91, Phi94, Phi95}. We refer readers to \cite[Section 2.2]{MY25} for a summary of the theory. When $X = \bb{P}^n$ is projective space, we let $h_{\Ph}\left(V\right)$ denote the Philippon height of the subvariety $V$ and 
\[\h_{\Ph}\left(V\right) = \frac{h_{\Ph}\left(V\right)}{\deg V}\]
be the normalized Philippon height. If $f: \bb{P}^n \to \bb{P}^n$ is an endomorphism of degree $d \ge 2$ defined over $K$, then it polarizes $L = O\left(1\right)$ and we obtain a canonical height $\h_f$. Then \cite[Equation (2.8) and Proposition 4.1]{MY25} imply that for any subvariety $V$ of $\bb{P}^n$, we have
\begin{equation} \label{eq: Compare1} 
\left|\h_f\left(V\right) - \h_{\Ph}\left(V\right)\right| \leq c \left(\dim V + 1\right)
\end{equation}
for some $c = c\left(f\right) > 0$ depending only on $f$. If $V$ is not geometrically irreducible, we define $h_{\Ph}\left(V\right)$ to be the sum of the heights of its geometrically irreducible components. Similarly $h_f\left(V\right)$ extends to a general subvariety $V$ and \eqref{eq: Compare1} still holds. 
\par 
As the unnormalized canonical height $h_f$ comes from an adelic line bundle $\ovl{L}_f$, it satisfies Zhang's fundamental inequality \cite{Zha92, Zha95, Zha95b}. For an irreducible subvariety $V \subseteq X_{\ovl{K}}$, we let $e_1(\h_f,V)$ denote Zhang's first essential minimum with respect to $\h_f$, i.e. 
\[
e_1({\blue \h_f},V)
\;=\;
\sup_{U\subseteq V}\;\inf_{x\in U(\overline{K})} \h_f(x),
\]
where we take the supremum over all {\blue nonempty} Zariski open subsets $U$ of $V$. Then we have the inequality 
{\blue \[e_1(\h_f,V) \geq \frac{1}{\dim V+1} \h_f(V),\]}
i.e. for any {\blue $0<\delta < \frac{1}{\dim V+1}$}, the set of points
\[\{x \in V(\ovl{K}) \mid \h_f(x) {\blue <} \delta \h_f(V)\}\]
is contained in a proper Zariski closed subset $W$ of $V$. The main result of \cite{MY25} is to give quantitative control of $\deg\left(W\right)$ in terms of the ratio $\frac{\h_f\left(V\right)}{\max\left\{h\left(f\right),1\right\}}$, where $h\left(f\right)$ is some Weil height function on $f$. To be precise, we may assume that $X$ embeds into $\bb{P}^N$ via $L$ and $f$ extends to an endomorphism $f: \bb{P}^N \to \bb{P}^N$ of degree $d \ge 2$. Then we let $h\left(f\right)$ be the standard Weil height on $\End^N_d$. We have the following uniform control which can be deduced from \cite[Theorem 1.1]{MY25}. 

\begin{theorem} \label{QuantitativeZhang1}
Let $\delta > 0$ be a constant so that 
\[\h_f\left(V\right) \geq \delta \max\left\{h\left(f\right),1\right\}.\]
Then there exists a constant $M = M\left(\delta,N,d,\deg V\right)$ and $\eps = \eps\left(\delta,N,d,\deg V\right)$ such that the set of points 
\[\left\{x \in V\left(\ovl{K}\right) \mid \h_f\left(x\right) < \eps \h_f\left(V\right)\right\}\]
is contained in a proper Zariski closed subvariety $W$ of degree at most $M$. 
\end{theorem}

We remark that the cases of $\bb{G}_m^n$ and abelian varieties were due to David and Philippon \cite{DP99, DP02}. For split endomorphisms $(f,g)$ on $(\bb{P}^1)^2$ and $V = \Delta$ is the diagonal, such results can be deduced from Favre--Rivera-Letelier's quantitative equidistribution framework \cite{FRL06}, for example see the paper of Ang and the second author \cite{AY23} or of Gauthier \cite{Gau24}.

\section{Local Height Functions and Degenerations} \label{sec: LocalHeightsDegenerations}
We will develop a theory of local height functions for closed subschemes on $X^{\an}$ when $X$ is a quasi-projective variety over a number field $K$. For function fields, this was developed in \cite[Section 2]{LY26}. The same arguments largely port over, except we have to be a little more careful when dealing with errors caused by different choices in constructing the local height functions. 

\subsection{Berkovich Spaces} Let $\left(B, \left\lVert  \cdot \right\rVert \right)$ be a commutative Banach ring, i.e. $B$ is a commutative ring with unity $1$ and $\left\lVert  \cdot \right\rVert  : B \to \bb{R}_{\geq 0}$ a norm for which $B$ is complete. Furthermore, we have the triangle inequality $\left\lVert a+b\right\rVert  \leq \left\lVert a\right\rVert  + \left\lVert b\right\rVert $ and we will assume that our norm $|| \cdot ||$ is multiplicative. 
\par 
A multiplicative seminorm $\left| \cdot \right|: B \to \bb{R}_{\geq 0}$ is said to be bounded if $\left|b\right| \leq \left\lVert b\right\rVert $ for all $b \in B$. We define the Berkovich spectrum $\cal{M}\left(B\right)$ to be the set of all bounded multiplicative seminorms on $B$. We endow $\cal{M}\left(B\right)$ with the Gelfand topology, which is the weakest topology such that for all $b \in B$, the function $\left|b\right|: \cal{M}\left(B\right) \to \bb{R}_{\geq 0}$ is continuous. 
\par 
Now given a scheme $X$ over $k$, not necessary of finite type, we may define $X^{\an}$ following \cite[Section 3.1.1]{YZ24}. If $X = \Spec A$ is affine, then we define $X^{\an}$ as $\cal{M}\left(A/B\right)$, the set of multiplicative seminorms on $A$ that restricts to an element of $\cal{M}\left(B\right)$ on $B$. In general if $X$ is given by an open cover $U_i = \Spec A_i$, we may glue $U_i^{\an}$ to obtain $X^{\an}$. 
\par 
In our setting, we will take our base Banach ring to be $\left(\bb{Z}, \left| \cdot \right|_{\infty}\right)$, where $\left| \cdot \right|_{\infty}$ is the Euclidean norm on $\bb{Z}$. The description of $\cal{M}\left(\bb{Z}\right)$ is given in \cite[1.4.3]{Ber90}. Then any scheme $X$ can be viewed canonically as a scheme over $\bb{Z}$ and we define $X^{\an}$ as the Berkovich analytification over $\left(\bb{Z}, \left| \cdot \right|_{\infty}\right)$. 
\par 
If $X$ is an irreducible scheme, then any rational function $f$ on $X$ induces a continuous function $f: X^{\an} \to \bb{R} \cup \left\{\pm \infty\right\}$.

\subsection{Local Height Functions} \label{sec: LocalSubSec}
Let $X$ be an irreducible projective variety over a number field $K$ and let $Z$ a closed subscheme of $X$. Following Silverman \cite[Theorem 2.1]{Sil87} (see also Matsuzawa \cite{Mat20}), we will define a local height function $\lambda_Z : X^{\an} \setminus Z^{\an} \to \bb{R} \cup \left\{+\infty\right\}$. These functions will be well-defined up to a function that is bounded above by $\log^+ \max\left\{\left|p_1\right|,\ldots,\left|p_m\right|\right\}$ for some $p_i \in K^{\times}$. 
\par 
First, let $D$ be a Cartier divisor of $X$. Let $L,M$ be basepoint free line bundles on $X$ for which $\mathcal{O}\left(D\right) \simeq L \otimes M^{-1}$. Fix a meromorphic section $s_D$ of $\mathcal{O}\left(D\right)$. Picking bases of global sections $s=\left\{s_0,\ldots,s_n\right\}$ for $L$ and $t=\left\{t_0,\ldots,t_m\right\}$ for $M$, we may define the local height $\lambda_D: X^{\an} \to $ relative to the presentation $\cal{D} = \left(s_D,L,s,M,t\right)$ to be 
\[\lambda_{\cal{D}}\left(p\right) = \max_{k} \min_{l} \log \left|\frac{s_k}{t_l s_D}\left(p\right)\right|\]
where $\frac{s_k}{t_l s_D}$ is a rational function on $X$. This gives us a well-defined continuous function $\lambda_{\cal{D}}: X^{\an} \to \bb{R} \cup \left\{+\infty\right\}$. If $D$ is an effective Cartier divisor, then this is real-valued if $x \in X^{\an} \setminus D^{\an}$. 
\par 
For an affine variety $U$ with an embedding $\iota: U \xhookrightarrow{} \bb{A}^M$, we say that a point $x \in U\left(K\right)$ has coordinates all $\leq 1$ if $x = \left(x_1,\ldots,x_M\right)$ with $\left|x_i\right| \leq 1$ under the embedding $\iota$.
\begin{proposition} \label{BoundedSubset1}
There exists a finite collection of affine opens $U_1,\ldots,U_m$ with embeddings $\iota_i: U_i \xhookrightarrow{} \bb{A}^{M_i}$ such that for any $x \in X\left(K\right)$ for a valued field $K$, there exists $i$ such that the coordinates of $x$ in $U_i$ are all $\leq 1$. 
\end{proposition}

\begin{proof}
Embedding $X$ into $\bb{P}^N$, it suffices to prove it for $\bb{P}^N$. In this case, if $U_0,\ldots,U_N$ are the $N+1$ basic affine opens where $U_i$ is given by $\left(\frac{z_0}{z_i},\frac{z_1}{z_i},\ldots, \frac{z_N}{z_i}\right)$, then the coordinates of $x = \left[x_0 : \cdots : x_N\right]$ in $U_i$ are $\leq 1$ if $\left|x_i\right|$ is maximal. Since $\left|x_i\right|$ will be maximal for some $0 \leq i \leq N$, we are done. 
\end{proof}

This implies the following. 

\begin{proposition} \label{BoundedSubset2}
Let $\cal{D}$ and $\cal{D}'$ be two presentations of $D$, Then there exists $p_1,\ldots,p_m \in K^{\times}$ so that 
\[\left|\lambda_{\cal{D}} - \lambda_{\cal{D}'}\right| \leq \log^+\max\left\{\left|p_1\right|,\ldots,\left|p_m\right|\right\}\]
on $X^{\an}$. 
\end{proposition}

\begin{proof}
This follows from the proof of \cite[Theorem 2.2.11 and Remark 2.2.13]{BG06}.     
\end{proof}

We also have a formula for $\lambda_{\cal{D}}$ when $D$ is locally cut out by a function $f$. 

\begin{proposition} \label{BoundedSubset3}
Let $U$ be an affine open and assume that $D = \divv\left(f\right)$ locally on $U$. Let $\cal{D}$ be a presentation for $D$, where we set $D \simeq L \otimes M^{-1}$. Fix an embedding $U \xhookrightarrow{} \bb{A}^N$. Then there exists $p_1,\ldots,p_m \in K^{\times}$ and $A > 0$ so that for any $x \in X^{\an}$ with coordinates $\left|x_i\right| \leq 1$ in $U$, we have 
\[\left| \log \left|f\left(x\right)\right|^{-1}  - \lambda_{\cal{D}}\left(x\right) \right| \leq A \log^+ \max\left\{\left|p_1\right|,\ldots,\left|p_m\right|\right\}.\]
\end{proposition}

\begin{proof}
Let $D \simeq L \otimes M^{-1}$ be the line bundles used for the presentation for $\cal{D}$. If $L$ and $M$ are both trivial on $U$, then this follows from the proof of \cite[Lemma 4.2.4]{Mat20}. Note that we obtain the analogous statement if we look at points with affine coordinates $\leq C$ instead of just $\leq 1$, where the data of $\left\{p_1,\ldots,p_m\right\}$ and $A$ will be allowed to depend on $C$ too.      
\par 
In general, we replace $U$ with a finite affine open cover $U_1,\ldots,U_n$ with  $U_i = \divv\left(f_i\right)$ where each $L$ and $M$ are trivial on $U_i$. Note that if $\left(x_1,\ldots,x_N\right)$ are affine coordinates for $U$, then $\left(x_1,\ldots,x_N, f_i^{-1}\right)$ may serve as affine coordinates for $U_i$. As the $U_i$'s cover $U$, we may find $c_1,\ldots,c_n \in \ovl{\bb{Q}}$ so that 
\[\sum_{i=1}^{n} c_i f_i = 1\]
on $U$. Then is some $\delta > 0$, depending only on $m$ and $c_i$'s, so that $\left|f_i\left(x_1,\ldots,x_N\right)\right| \leq \delta$ for at least one $i$. Then for $U_i$ our affine coordinates are bounded above by $\max \left\{1,\delta^{-1}\right\}$, which imply that $\left|\log \left|f\left(x\right)\right|^{-1} - \lambda_{\cal{D}}\left(x\right)\right| \leq A \log^+ \max\left\{\left|p_1\right|,\ldots,\left|p_m\right|\right\}$ with the $p_i$'s and $A > 0$ depending only on our affine open cover as desired.
\end{proof}

Now let $Z$ be a closed subscheme of $X$. Let $D_i$ be Cartier divisors and let $\cal{D}_i$ be presentations of each $D_i$. We say that $\cal{Z} = \left\{\cal{D}_1,\ldots,\cal{D}_m\right\}$ is a presentation of $Z$ if $\cap_{i=1}^{m} D_i = Z$ as closed subschemes of $X$. We then define 
\[\lambda_{\cal{Z}} = \min \left\{\lambda_{\cal{D}_1},\ldots,\lambda_{\cal{D}_m}\right\}.\]

This is a continuous real valued function on $X^{\an} \setminus Z^{\an}$. We have the following analogues of Propositions \ref{BoundedSubset2} and \ref{BoundedSubset3}. They both follow from the proof of \cite[Proposition 4.25]{Mat20}.

\begin{proposition} \label{BoundedSubset4}
Let $\cal{Z},\cal{Z}'$ be two presentations for $Z$. Then there exists $p_1,\ldots,p_m \in K^{\times}$ such that 
\[\left|\lambda_{\cal{Z}} - \lambda_{\cal{Z}'}\right| \leq \log^+ \max\left\{\left|p_1\right|,\ldots,\left|p_m\right|\right\}.\]
\end{proposition}

\begin{proposition} \label{BoundedSubset5}
Let $U$ be an affine open of $X$ and let the ideal of $Z$ be given by $\left\{f_1,\ldots,f_k\right\}$. Let $\cal{Z}$ be a presentation for $Z$. Fix an embedding $U \xhookrightarrow{} \bb{A}^N$. Then there exists $p_1,\ldots,p_m \in K^{\times}$ and $A > 0$ so that 
\[\left| \log \min\left\{\left|f_1\left(x\right)\right|^{-1},\ldots,\left|f_k\left(x\right)\right|^{-1}\right\} - \lambda_{\cal{Z}} \right| \leq A \log^+ \max\left\{\left|p_1\right|,\ldots,\left|p_m\right|\right\}\]
for all $x \in U\left(K\right)$ with coordinates $\leq 1$. 
\end{proposition}

The above propositions imply that for each closed subscheme $Z$, we can define a local height function $\lambda_Z$ on $X^{\an}$ that is well-defined up to an error of $\log^+ \max\left\{\left|p_1\right|,\ldots,\left|p_m\right|\right\}$ for a finite set of elements $p_1,\ldots,p_m \in K^{\times}$. We will now consider our local height functions up to such equivalence. This is essentially equivalent to considering functions up to $M_K$-constants in the usual context of local height functions over global fields. 

\begin{proposition} \label{LocalHeightProperties1}
We have the following properties. 

\begin{enumerate}
\item Let $Z \subseteq Y$. Then $\lambda_Z \leq \lambda_Y$.

\item If $\Supp Z \subseteq \Supp Y$, then there is a constant $c > 0$ such that $\lambda_Z \leq c \lambda_Y$.

\item Let $\vphi: X \to X'$ be a morphism of projective varieties over $\mathbb{C}$. Let $Y$ be a closed subscheme of $X'$. Then 
$\lambda_{\vphi^* Y} = \lambda_Y.$

\item Let $U \subseteq X$ be an open set and let $W \subseteq X$ be the complement of $U$. If $Z \vert_U \subseteq Y \vert_U$, then there exists a constant $c > 0$ such that $\lambda_Z \leq \lambda_Y + c \lambda_W$.
\end{enumerate}
\end{proposition}

\begin{proof}
The proof is analogous to that of \cite[Theorem 2.1]{Sil87}.     
\end{proof}

We now let $X$ be a quasi-projective variety. Fix a compactification $\ovl{X}$ of $X$. We thus obtain a local height function for the boundary $\lambda_{\partial X}$ where we set $\partial X = \ovl{X} \setminus X$. A priori this depends on the choice of compactification. 

\begin{proposition} \label{BoundaryHeight1}
Let $\ovl{X}$ and $\ovl{X}'$ be two compactifications of $X$ and let $\lambda_{\partial X}, \lambda'_{\partial X}$ be the two local height functions for the  boundary. Then there exists a constant $c > 0$ such that 
\[c^{-1} \lambda_{\partial X}' \leq \lambda_{\partial X} \leq c \lambda_{\partial X}'.\]
\end{proposition}

\begin{proof}
The proof is analogous to that of \cite[Theorem 5.1]{Sil87}.     
\end{proof}

Now let $Z$ be a closed subscheme of $X$. If we fix a compactification $\ovl{X}$ and replace $Z$ with $Z \cup \left(\ovl{X} \setminus X\right)$, we thus obtain a well-defined $\lambda_Z$ up to a multiple of $\lambda_{\partial X}$. 
\par 
We now move to the setting of global heights oveer a number field $K$. After fixing a compactification and a presentation for $\ovl{X} \setminus X$, we define
\[h_{\partial X}\left(x\right) = \sum_{v \in M_K} n_v \lambda_{\partial X,v}\left(x\right) \text{ for all } x \in X\left(K\right)\]
where $\lambda_{\partial X,v}$ is the local height function restricted to $X_v^{\an}$. This is invariant under finite extensions and we get a function on $X\left(\ovl{\bb{Q}}\right)$. If $h'_{\partial X}$ is the function coming from another choice of data, then 
\[h_{\partial X} \leq c_1 h'_{\partial X} + c_2\]
for some $c_1, c_2 > 0$. Now for a closed subscheme $Z$ of $X$, we may fix  a presentation $\cal{Z}$ for $Z \cup \left(\ovl{X} \setminus X\right)$ which gives us a local height function $\lambda_Z$. We can now define 
\[h_Z\left(x\right) = \sum_{v \in M_K} n_v \lambda_{\cal{Z},v}\left(x\right) \text{ for all } x \in \left(X \setminus Z\right)\left(K\right).\]
Again, this quantity is invariant under finite extensions and so we obtain a well-defined function on $\ovl{\bb{Q}}$. If $h'_Z$ is the global height function for another compactification and presentation, we obtain 
\[\left|h_Z - h'_{Z}\right| \leq c_1 h_{\partial X} + c_2\]
for some $c_1, c_2 > 0$. One fact that we will need is the following.

\begin{proposition} \label{AmpleHeight1}
Let $h_L$ be an ample height function on $X\left(\ovl{\bb{Q}}\right)$. Then there exists $c_1, c_2 > 0$ so that 
\[h_Z \leq c_1 h_L + c_2.\]
\end{proposition}

\begin{proof}
See \cite[Proposition 4.2]{Sil87}.     
\end{proof}

\subsection{Degenerations with Ultrafilters} \label{sec: DegenUltra}
We state and explain some basic general facts about degenerating via ultrafilters. For more details, the readers may consult \cite{FG24}, \cite{Yap25} or \cite{GY26}. Let $X$ be a quasi-projective variety over a number field $K$ which is irreducible. Let $\left(K_n\right)$ be a sequence of complete valued fields with norm $\left|\cdot \right|_n$ that contains $K$. 

We define 
\[A^{\eps} = \{ (x_n) \in \prod_{n=1}^{\infty} K_n \mid |x_n|^{\eps_n} \text{ is bounded}\}.\]
This is a Banach ring under the supremum norm $\left|\left(x_n\right)\right| = \sup_n \left|x_n\right|^{\eps_n}$. We have natural projection maps $\pi_n: A^{\eps} \to K_n$ which induce projections $\pi_n: X\left(A^{\eps}\right) \to X\left(K_n\right)$. The Berkovich spectrum $\cal{M}\left(A^{\eps}\right)$ is isomorphic to $\beta \bb{N}$, the set of ultrafilters on $\bb{N}$, via the bijection 
\[\omega \mapsto \left|\left(x_n\right)\right|_{\omega} = \lim_{\omega} \left|x_n\right|^{\eps_n}\]
where $\omega \in \beta \bb{N}$ is an ultrafilter on $\bb{N}$ (see \cite{CN74} for the basics of ultrafilters).  

\begin{definition} \label{DegenerateDefinition1}
Let $x_n \in X\left(K_n\right)$ be a sequence of points. An $\left(\eps_n\right)$-degeneration of $\left(x_n\right)$ is a point $x \in X\left(A^{\eps}\right)$ such that $\pi_n\left(x\right) = x_n$ for all $n \in \bb{N}$.     
\end{definition}

If an $\left(\eps_n\right)$-degeneration exists, it is necessarily unique. 

\begin{proposition} \label{DegenerationUniqueness1}
There exists at most one $\left(\eps_n\right)$-degeneration of $\left(x_n\right)$. 
\end{proposition}

\begin{proof}
Let $Y=\bigsqcup_{n\ge1}\Spec K_n$. Each projection 
$\pi_n\colon A^\varepsilon\to K_n$ induces \[\Spec\left(\pi_n\right)\colon\Spec K_n\to\Spec A^\varepsilon,\] and these glue to a single morphism \[f\colon Y\to\Spec A^\varepsilon.\] We claim that $f$ is an epimorphism in the category of separated, reduced schemes. Indeed given $g,h: \Spec A^{\eps} \to X$ such that $g \circ f = h \circ f$, we can consider the diagonal embedding $\Delta: X \to X \times X$, which is a closed immersion by separatedness of $X$. Pulling back $\Delta\left(X\right)$ via $\left(g,h\right): \Spec A^{\eps} \to X \times X$ gives us a closed subscheme $Z$ of $\Spec A^\eps$ and the map $f: Y \to \Spec A^{\eps}$ must factor through $Z$ since $g \circ f = h \circ f$.
\par 
We now claim that $f$ is dominant. Indeed, the image of $f$ consists of the kernels of the projection maps $\pi_n: A^{\eps} \to K_n$. Given any non-zero $a \in A^{\eps}$ with $a = \left(a_n\right)$, there must exist some $a_n \not = 0$ and so $\pi_n\left(a\right) \not = 0$, which implies that $\ker\left(\pi_n\right) \in D\left(a\right)$. Since distinguished opens form a basis, it follows that $f\left(Y\right)$ is dense and hence $f$ is dominant as desired. Now as $f$ is dominant and factors through $Z$, we must have $Z$ be dense in $\Spec A^{\eps}$ and as $\Spec A^{\eps}$ is reduced, we must have $Z = \Spec A^{\eps}$ and hence $g = h$ as desired. 
\par 
Now the points $x_n \in X\left(K_n\right)$ define maps $\Spec K_n\to X$ which glue to \[g\colon Y\to X.\] and $g$ is clearly unique. By definition an $\left(\varepsilon_n\right)$-degeneration is a morphism \[\tilde x\colon\Spec A^\varepsilon\to X\] with $\tilde x\circ f=g$. Given two such lifts $\tilde{x}_1, \tilde{x}_2$, we have $\tilde{x}_1 \circ f = \tilde{x}_2 \circ f$. Since $f$ is an epimorphism of reduced, separated schemes, post-composition with $f$ is injective on morphisms.  It follows that $\tilde x_1=\tilde x_2$, proving uniqueness.
\end{proof}

Clearly an $\left(\eps_n\right)$-degeneration need not always exist. For example if $\eps_n = 1$ and $X = \bb{A}^1$, then $x_n$ corresponds to an element of $K_n$ and an $\left(\eps_n\right)$-degeneration exists if and only if $\left|x_n\right|$ is bounded from above. However if $X$ is a projective variety, then an $\left(\eps_n\right)$-degeneration always exists. 

\begin{proposition} \label{DegenerationExistence1}
Assume that $X$ is a projective variety over $K$. Let $x_n \in X\left(K_n\right)$ for all $n$. Then an $\left(\eps_n\right)$-degeneration of $\left(x_n\right)$ always exists. 
\end{proposition}

\begin{proof}
Let $\left\{U_1,\ldots,U_n\right\}$ be an affine open cover for $X$ and let $Y = \bigsqcup_{i=1}^{n} U_i$. Then we have a natural morphism $\pi: Y \to X$ and $Y$ is affine. Let $x_n \in X\left(K_n\right)$. Then since $X$ is projective, $x_n$ lifts to $\tilde{x}_n \in X\left(R_n\right)$ where $R_n$ is the valuation ring of $K_n$. Hence we have an induced morphism $\tilde{x}_n: \Spec R_n \to X$. Since $\left\{U_i\right\}$ is an open cover and $R_n$ is a valuation ring, if $\tilde{x}_n^{-1}\left(U_i\right)$ contains the maximal ideal of $R_n$ then it contains the entire $\Spec R_n$. Thus $\tilde{x}_n \in U_i\left(R_n\right) \subseteq Y\left(R_n\right)$. 
\par 
Since $Y$ is affine, we may let $x_1,\ldots,x_N$ be a choice of affine coordinates for $Y$ over $K$. Then $\tilde{x}_n$ has affine coordinates $a_{n,1},\ldots,a_{n,N}$ with $a_{n,i} \in R_n$ and $\left|a_{n,i}\right| \leq 1$ for all $1\le i\le N$. Hence if we set $a_i = \left(a_{n,i}\right)$, we have $a_i \in A^{\eps}$. We thus obtain an $A^{\eps}$-point $\left(a_i\right)_{i=1}^N$ of $Y$ that commutes with the projection map $\pi_n: A^{\eps} \to K_n$. Projecting it down to $X$, we obtain our $\left(\eps_n\right)$-degeneration as desired.
\end{proof}

\begin{remark}\label{rmk:Bep} Observe that in Proposition \ref{DegenerationExistence1}, the resulting $\left(\varepsilon_n\right)$-degeneration is actually a $B^{\eps}$-point of $X$ for \[B^{\eps} = \prod_{n=1}^{\infty} R_n,\] where $R_n$ is the valuation ring of $K_n$.\end{remark}
\par 
Now if $X$ is only quasi-projective, we may embed it into some projective compactification $\ovl{X}$. Then the $\left(\eps_n\right)$-degeneration exists on $\ovl{X}$ and we can try to check whether our $A^{\eps}$-point lies on $X$. Recall that for each valued field $K'$ containing our number field $K$, we have a local distance function $\lambda_{\partial X}$ depending on the compactification $\ovl{X}$. On an affine open $U_i \subseteq \overline{X}$ such that the reduced closed subscheme $\left(\overline{X} \setminus X \right) \cap U_i$ is given by $\left\{f_1 = f_2 = \cdots = f_m = 0\right\}$ for $f_j \in\mathcal{O}\left(U_i\right)$, Proposition \ref{BoundedSubset5} implies that for $x\in U_i\left(K\right)$ whose affine coordinates in $U_i$ have norm $\leq 1$, we can find $p_1,\ldots,p_m \in K^{\times}$ and $A > 0$ so that  
\begin{equation}\label{eqn:distancemin}
\left|\lambda_{\partial X}\left(x\right) - \min_{1 \leq i \leq n} -\log \left|f_i\left(x\right)\right| \right| \leq A \log^+\max\left\{\left|p_1\right|,\ldots,\left|p_m\right|\right\}.\end{equation} The function $\lambda_{\partial X}$ depends on the choice of compactification, but $\frac{\lambda_{\partial X}}{\lambda'_{\partial X}}$ (up to an additive constant of $\log^+ \max \{|p_1|,\ldots,|p_m\}$is bounded from above and below for any two compactifications.

\begin{proposition}\label{DegenerationExistence2}
Let $X$ be a quasi-projective variety over $K$. Let $x_n \in X\left(K_n\right)$ and let $\left(\eps_n\right)$ be a sequence of positive reals.  Assume that for any $p\in K^{\times}$, the sequence $\{\eps_n \log^+|p|_n\}$ is uniformly bounded from above. Then an $\left(\eps_n\right)$-degeneration of $\left(x_n\right)$ exists on $X$ if and only if  
$\eps_n \lambda_{\partial X}\left(x_n\right)$
is uniformly bounded from above. 
\end{proposition}

\begin{proof}
Pick any projective compactification $\ovl{X}$ of $X$ and let $U_i$ be a finite affine open cover of $\ovl{X}$. Let $Y = \bigsqcup U_i$ and let $\pi: Y \to \ovl{X}$ be the natural morphism. Then $\pi^{-1}\left(X\right)$ is Zariski open in $Y$. Let $f_1,\ldots,f_m \in \mathcal{O}\left(Y\right)$ cut out $Y \setminus \pi^{-1}\left(X\right)$. Note that $\bigcup_{i=1}^{m} D\left(f_i\right) = X$ where $D\left(f_i\right)$ is the distinguished open for $f_i$. 
\par  
Now given $x_n \in \ovl{X}\left(K_n\right)$, we lift (as in the proof of Proposition \ref{DegenerationExistence1}) $x_n$ to an $R_n$-point of some $U_i$ and hence to an $R_n$-point $\tilde{x}_n$ of $Y$. We then have 
\[\left| \eps_n \lambda_{\partial X}\left(x_n\right) - \min_{1 \leq i \leq n} - \left(\eps_n\right) \log \left|f_i\left(\tilde{x}_n\right)\right|_n \right| \leq \eps_n A \log^+ \max\left\{\left|p_1\right|_n,\ldots,\left|p_m\right|_n\right\}\] as given in \eqref{eqn:distancemin}.
If we let $a_{n,1},\ldots,a_{n,N}$ be the affine coordinates of $x_n$, then the affine coordinates of our $A^{\eps}$-point $\left(x_n\right)$ are given by $\left(a_1,\ldots,a_N\right)$ where each $a_i = \left(a_{n,i}\right)$. We wish to check whether our $A^{\eps}$-point lies on $\pi^{-1}\left(X\right)$. This is equivalent to checking that 
\[\left(f_1\left(a_1,\ldots,a_N\right), \ldots, f_m\left(a_1,\ldots,a_N\right)\right)\]
generate the unit ideal in $\left(A^{\eps}\right)$ as the condition is equivalent to the pullback of the opens $D\left(f_i\right)$ covering $\Spec A^{\eps}$. But an element $b = \left(b_n\right)$ in $A^{\eps}$ is in $\left(A^{\eps}\right)^{\times}$ if and only if $\left(\left|b_n\right|^{-1}\right)^{\eps_n}$ is uniformly bounded from above \cite[Remark 3.7]{FG24}. Hence if there exists $c_i$ with $c_i = \left(c_{n,i}\right)^{\infty}_{n=1}$ with 
\begin{equation} \label{eq: Inverse1}
\sum c_i f_i\left(a_1,\ldots,a_N\right) \in \left(A^{\eps}\right)^{\times} 
\end{equation} then as $\left|c_{i,n}\right|^{\eps_n}$ is bounded from above as $n$ varies, \eqref{eq: Inverse1} implies that
\[\max_{1 \leq i \leq m} \left|f_i\left(a_{n,1},\ldots,a_{n,N}\right)\right|_n^{\eps_n}\] is uniformly bounded away from $0$ and so $\eps_n \lambda_{\partial X}\left(x_n\right)$ is uniformly bounded from above since $\eps_n A \log^+\max\left\{\left|p_i\right|\right\}$ is uniformly bounded from above. Conversely suppose that $\eps_n \lambda_{\partial X}\left(x_n\right) \leq C$ for some $C > 0$. For each $n$, let $i_n$ be an index such that 
\[\left|-\log \left|f_{i_n}\left(a_{n,1},\ldots,a_{n,N}\right)\right|_n - \lambda_{\partial X}\left(x_n\right) \right| \leq A \log^+ \max \left\{\left|p_1\right|_n,\ldots,\left|p_m\right|_n\right\}.\]
We define $c_i \in A^{\eps}$ as $c_i = \left(c_{1,i},\ldots,c_{n,i},\ldots\right)$ where $c_{n,i} = 1$ if $i_n = i$ and $0$ otherwise. Then 
\[\left|\sum c_{n,i} f_i\left(a_{n,1},\ldots,a_{n,N}\right)\right|^{\eps_n} = e^{- \lambda_{\partial X}\left(x_n\right) \eps_n + O\left(1\right)}\]
and hence the left-hand side is uniformly bounded away from $0$ as $\lambda_{\partial X}\left(x_n\right) \eps_n$ is uniformly bounded from above. Thus $\sum c_i f_i\left(a_1,\ldots,a_N\right) \in \left(A^{\eps}\right)^{\times}$ and so our $A^{\eps}$-point lies in $X$ as desired.
\end{proof}

Note that there are two natural situations where $\eps_n \log^+\left|p\right|_n$ is uniformly bounded from above. The first is if the valued field $\left(K_n,\left| \cdot \right|_n\right)$ is the same for all $n$ and $\eps_n$ is bounded from above. The second is if the residue characteristic of $K_n$ is either $0$ for all sufficiently large $n$ or goes to $\infty$. This will be the two specific cases for which we will apply Proposition \ref{DegenerationExistence2} to.

\section{Non-Archimedean Rigidity} \label{sec: Rigidity}

{\blue The main result of this section is Theorem~\ref{Rigidity1}, which is a strengthening of the non-archimedean rigidity statement of Theorem~\ref{IntroRigidity1}.} Let $\left(K,\left| \cdot \right|\right)$ be a complete non-archimedean algebraically closed field with residue characteristic $0$. Assume that $\left| \cdot \right|$ is non-trivial. We will prove the following statement. 

{\blue \begin{theorem} \label{Rigidity1}
Let $f,g: \bb{P}^1 \to \bb{P}^1$ be two rational maps of degrees $d_1,d_2 \geq 2$ respectively, that are defined over $K$. Let $C \subseteq \bb{P}^1 \times \bb{P}^1$ be an irreducible curve over $K$ that projects dominantly to both coordinates. Assume that $\Lyap\left(f\right) > 0$ and that there exist an open subset $U\subseteq C^{\an}$ and a constant $\rho>0$ such that
\[\left(\pi_1^*\mu_f\right)\vert_U=\rho\left(\pi_2^*\mu_g\right)\vert_U\ne0.\]
Then there exist positive integers $a,b$ with $d_1^a=d_2^b$ such that $C$ is preperiodic under $\left(f^a,g^b\right)$.
\end{theorem}}

When $K = \bb{C}$ is archimedean {\blue and $U=C^{\an}$}, a weaker version of Theorem \ref{Rigidity1} is the key complex dynamical input to the solution of the dynamical Manin--Mumford for $\left(\bb{P}^1\right)^n$ by Ghioca--Nguyen--Ye \cite{GNY18, GNY19}. In their setting, they are able to assume that $C$ contains infinitely many preperiodic points for $\left(f,g\right)$ which is a necessary ingredient in their proof. Mavraki--Schmidt \cite{MS25} subsequently removed this assumption and used it in their proof of the uniform dynamical Bogomolov for one-parameter families of split maps. 
\par 
Both the arguments of \cite{GNY18} and \cite{MS25} for $K = \bb{C}$ proceed by constructing a family of symmetries of the equilibrium measure $\mu_f$ using the condition on equilibrium measures, before applying Levin--Przytycki \cite{LP97} to obtain that outside of the exceptional maps, the family of symmetries has to be finite in size. It then easily follows that $C$ has to be preperiodic for $\left(f,g\right)$. 
\par 
There are two main obstacles in adapting this argument to the non-archimedean case. The first is that Levin--Przytycki relies on Montel's theorem, i.e. that any family of holomorphic maps $\psi_n: U \to V$ for $U,V$ bounded open subsets of $\mathbb{C}$, has a convergent subsequence. This is false in the non-archimedean case due to the failure of local compactness. The second is even if we knew that $\left\{\psi_n\right\}$ had a convergent subsequence, Levin--Przytycki's argument relies on the Euclidean topology and does not work for non-archimedean fields. 
\par 
{\blue Both obstacles are handled by the following simple observation. Let $\psi_n:D\to\psi_n\left(D\right)$ be injective analytic maps such that
\[\mu_f\left(\psi_n\left(E\right)\right)\geq c\mu_f\left(E\right)\]
for every Borel subset $E\subseteq D$, with $c>0$ independent of $n$. If $\mu_f\left(D\right)>0$, the images $\psi_n\left(D\right)$ cannot be pairwise disjoint, since $\mu_f$ has finite mass. Applied on smaller discs, this observation produces a convergent subsequence.} 
\par 
We will follow the proof strategy by Ji--Xie \cite{JX25} which uses ergodicitiy properties of $\mu_f$, see also \cite{DFG23}. We first collect some basic facts of $f$ and $\mu_f$ which are classical in the archimedean case, but have only been recently proved in the non-archimedean setting. We then state some independent lemmas that will help us prove that if our constructed family $\left\{\psi_n\right\}$ has a convergent subsequence, then the family has to be finite. Finally, we will prove Theorem \ref{Rigidity1}. 
\par 
\subsection{Basic Facts}

{\blue For the next two subsections, let $f:\bb{P}^1\to\bb{P}^1$ have degree $d\geq2$ and satisfy $\Lyap\left(f\right)>0$. In particular, the valuation of $K$ is non-trivial.}

The measure $\mu_f$ (and similarly for $\mu_g$) is the equilibrium measure of $f$. It is a probability measures that is supported on the Berkovich space $\berkP_k$. By \cite{FRL10}[Proposition 3.5], the measure $\mu_f$ is ergodic. When working on $\bb{P}^1$, we will let $B\left(x,r\right)$ denote the ball of radius $r$ and centered at $x$ using the spherical metric. 
\par 
The quantity $\Lyap\left(f\right)$ is the Lyapunov exponent of $f$. Working with the spherical metric. we may let $|f'|$ denote the spherical derivative which satisfies the formula 
\[|f'|(x) = |f'(x)| \frac{\max\{1,|x|\}^2}{\max\{1,|f(x)|\}^2}\]
for $x \not \in f^{-1}(\infty)$. This extends continuously to the entire $\berkP$. We may now define 
\[\Lyap(f) = \int \log |f'|(x) d \mu_f(x).\]
{\blue By \cite[Theorems B and C, Corollary D]{FRL25}, in residue characteristic $0$, positive Lyapunov exponent is equivalent to $\mu_f\left(\bb{P}^1\left(K\right)\right)=1$, and to the existence of a type I repelling periodic point. Moreover, \cite[Theorem B\,(3)]{FRL25} gives the equidistribution of type I repelling periodic points towards $\mu_f$. The measure $\mu_f$ has no atoms, since an atom would imply potential good reduction by \cite[Theorem E]{FRL10}.} 

\begin{proposition} \label{Lyapunov1}
For $\mu_f$-almost every $x$, we must have 
\[\lim_{n \to \infty} \frac{1}{n} \log| (f^n)|'(x) \to \Lyap(f).\]
\end{proposition} 

\begin{proof}
Since $\mu_f$ is ergodic and $\log |f'|(x) \in L^1(\mu_f)$, by Birkhoff's ergodic theorem we must have 
\[\lim_{n \to \infty} \frac{1}{n} \sum_{i=0}^{n-1}\log |f'|(f^i(x)) \to \Lyap(f)\]
for $\mu_f$-almost every $x$. Since 
\[|f_1 \circ f_2|'(x) = |f_1|'( f_2(x)) \cdot |f_2|'(x),\]
this is equivalent to 
\[\lim_{n \to \infty} \frac{1}{n} \log |(f^n)|'(x) \to \Lyap(f)\]
as desired.
\end{proof}

We now state \cite{FRL25}[Lemma 3.18]. We set
\[\hat{J} = \{(x_i)_{i \in \bb{Z}} \in (J_f \cap \bb{P}^1(K))^{\bb{Z}}, f(x_i) = x_{i+1}\}\]
and we let $\sigma$ be the left-shift operator. We have a projection map $\pi: \hat{J} \to J_f \cap \bb{P}^1(K)$ that satisfies $f \circ \pi = \pi \circ \sigma$. We let $\hat{\mu_f}$ be the unique probability measure that is invariant under $\sigma$ and projects to $\mu_f$. 

\begin{proposition} \label{Lyapunov2}
There exists a subset $\hat{J}_0$ of $\hat{J}$, with full measure, and a measurable function $\tau,L: \hat{J}_0 \to \left(0,\infty\right)$ such that for any $\hat{x} \in \hat{J}_0$ and $n \in \bb{N}$, the function $f^n$ admits an inverse branch sending $x_0$ to $x_{-n}$ defined on $B\left(x_0, \tau\left(\hat{x}\right)\right)$. 
\end{proposition}

\begin{proof}
This is \cite{FRL25}[Lemma 3.18]    
\end{proof}

With this, we can deduce the non-archimedean analogue of \cite{JX25}[Lemma 4.10]. Given a subset $A$ of $\bb{N}$, we let $d\left(A\right)$ denote its density if it exists. 

\begin{proposition} \label{Lyapunov3}
For every $\eps > 0$, there is $r>0$ such that for $\mu_f$-almost every $x \in \bb{P}^1\left(K\right)$, there exists a subset $A = A_x$ of positive integers with $d\left(A\right) > 1 - \eps$ such that for every $n \in A$, if we denote by $W$ the connected component of $f^{-n}\left(B\left(f^n\left(x\right),r\right)\right)$ containing $x$, then the map $f^n: W \to B\left(f^n\left(x\right),r\right)$ is a bijection. 
\end{proposition}

\begin{proof}
 We apply Proposition \ref{Lyapunov2}. For any fixed $\eps > 0$, we have a $r>0$ so that 
 \[\hat{F} = \left\{\hat{x} \in \hat{J}_0 \mid \tau\left(\hat{x}\right) \geq r\right\}\]
 has measure $\hat{\mu}\left(\hat{F}\right) > 1 - \eps$. By Birkhoff's ergodic theorem, there is some $\hat{G}$ with $\hat{\mu}_f\left(\hat{G}\right) = 1$, so that if $A\left(\hat{x}\right) = \left\{n \geq 1 \mid \sigma^{n}\left(\hat{x}\right) \in \hat{F} \right\}$, then $d\left(A\left(\hat{x}\right)\right) = \hat{\mu}\left(\hat{F}\right)$ for all $\hat{x} \in \hat{G}$. We now set $G_0 = \pi_0\left(\hat{G}\right)$ so that $\mu_f\left(G_0\right) = 1$. For $x \in G_0$, we may choose $\hat{x} \in \pi_0^{-1}\left(x\right) \cap \hat{G}$ and set $A_x = A_{\hat{x}}$, which gives us what we want. 
\end{proof}

\subsection{Convergent Sequences of Symmetries} In this section, we will study the situation of a family of analytic maps $\sigma_n:D\to D$, that preserves $\mu_f$, and also has a convergent subsequence. With some assumptions, we will be able to prove that $\{\sigma_n\}$ must have only finitely many distinct maps. This is the non-archimedean analogue of \cite[Theorem 3.2]{JX25}. 

{\blue Continue to assume $\Lyap\left(f\right)>0$, and write $\mu=\mu_f$. Let $D$ be a closed analytic disc of positive radius with $\mu\left(D\right)>0$. Since $\mu$ gives full mass to classical points and has no atoms, the interior of $D$ has positive measure.

For an injective analytic map $u:D\to u\left(D\right)$, we interpret $u^*\mu=c\mu$ on $D$ as
\[\mu\left(u\left(E\right)\right)=c\mu\left(E\right)\qquad\text{for every Borel }E\subseteq D.\]}

\begin{lem}\label{lem:move}
Let $p$ be a repelling $N$-periodic point of $f$ in $D$, $N \geq 1$. Let $\mu$ denote the equilibrium measure of $f$. Then there does not exist a family of injective analytic maps  $\sigma_n:D\to D$ such that 
	\begin{itemize}
		\item $\sigma_n^*\mu=c_n\mu$ on $D$, $c_n>0$.
		\item $\sigma_n\left(p\right)\ne p$. 
		\item $\sigma_n\to\id$ uniformly on $D$.
	\end{itemize}

\end{lem}

\begin{proof}
Assume by contradiction such family $\sigma_n$ exists. By passing to an iterate of $f$ we may assume {\blue that $p$ is fixed, and we write $d$ for the degree of this iterate of $f$}. We now use Poincare's linearization at $p$ \cite[Section 4.2]{Ben19}, which gives us a neighborhood $U$ of $p$ so that $f$ is conjugated to $[\la]: z \mapsto \lambda z$, where $\lambda$ is the multiplier at $p$. Since $\sigma_n\to\id$ uniformly, for sufficiently large enough $n$ we have $\sigma_n(U) = U$. We may then replace $D$ with $U$.
\par 
After conjugating, we may further assume that $D = D(0,1)$ is the closed unit disc, $p = 0$, and and $f$ is given by $[\la]$. If we view $[\la]$ as a map from $D(0, |\la|^{-1}) \to D(0,1)$, we then have $[\la]^* \mu = d \cdot \mu$. Note that as $p$ is a repelling fixed point, it is in the closure of the Julia set $J_f$ and so $\mu(D(0,\eps)) > 0$ for any $\eps > 0$. 
\par 
We first claim that there is a uniform $\delta > 0$ such that $c_n\geq \delta$ for every $n\geq 1$.  Since $\sigma_n\to\id$, we can  choose $M$ large enough such that $|\sigma_n(0)|<1/2$ and $|\sigma'_n(0)|>1/2$ for $n\geq M$.  Thus $D':=D(0,1/2)\subset \sigma_n(D)$ for every $n\geq M$.  Then for $n\geq M$,
$$c_n\mu(D)=\sigma^*_n\mu(D)=\mu(\sigma_n(D))\geq \mu(D')>0.$$
Thus $c_n\geq \mu(D')/\mu(D) > 0$ for $n\geq M$, which implies the claim since $\mu(D')/\mu(D)$ is a fixed positive constant. {\blue Also $c_n\leq 1$, since $\sigma_n(D)\subseteq D$ and $\mu(\sigma_n(D))=c_n\mu(D)$.}

For each $n\geq 1$, let  $l_n$ be the maximal positive integer such that $$ |\la^{-1}|<|\la^{l_n}\sigma_n(0)|\leq 1.$$
Then $l_n\to +\infty$,  since $\sigma_n(0)\to 0$.  Consider the injective analytic map 
$$\tau_n:=[\la^{l_n}]\circ \sigma_n\circ [\la^{-l_n}]:D\to D.$$
{\blue Since $\sigma_n$ is an isometry and $|\sigma_n(0)|\leq |\lambda|^{-l_n}$, it maps $D(0,|\lambda|^{-l_n})$ onto itself. Thus $\tau_n$ is an analytic automorphism of $D$.}

Then $\tau_n^*\mu=c_n\mu$ on $D$. Expanding $\sigma_n$ as a power series centered at $0$, if $a_k$ is the coefficient for $z^k$, then the coefficient of $z^k$ for $\tau_n$ is $\lambda^{-l_n(k-1)} a_k$. Since $|a_k| \leq 1$, it follows that if we define 
$$\eps_n(z):=\tau_n(z)-\tau_n(0)-\tau'_n(0)z,$$
then $\eps_n(z)\to 0$ uniformly on $D$ as $l_n \to \infty$. In particular, $\tau_n(z)$ becomes close to linear as $n \to \infty$. {\blue Uniform convergence $\sigma_n\to\id$ also gives $\tau'_n(0)=\sigma'_n(0)\to1$.}

\medskip

We now have two cases, depending on whether $\tau_n(0)$ has a subsequential limit or not. 

\medskip

{\bf Case 1:} We assume here that $\tau_n(0)$ has a subseuqential limit. Hence by passing to a subsequence, we may assume that $\tau_n(0)\to b$ for some $b\in D$. We may also assume that $c_n$ has a subsequential limit $c > 0$. By our choice of $l_n$, we have $|b| \geq |\lambda^{-1}|$ and thus $b \not = 0$.  Then $\tau_n\to T_b$ uniformly on $D$, and $T_b^*\mu=c\mu $ on $D$, where $T_b(z):=z+b$ is a translation.  Since $T_b:D\to D$ is a bijection,  we have 
$$c\mu(D)=T_b^*\mu(D)=\mu(T_b(D))=\mu(D),$$
thus $c=1$, i.e. $T_b$ preserves the measure $\mu$ on $D$.

Since the residue field of $K_v$ has characteristic $0$,   we have $|mb|=|b|$ for every every integer $m\neq 0$.  Let $D_0=D(0,r)$ be a closed disk with $0<r<|b|$,  define  $D_n:=D(nb,r)=T_b^n(D_0)\subset D$.  Assume $i\neq j$, since  $|(i-j)b|=|b|$, we have $D_i\cap D_j=\emptyset$.   Since $T_b^n$ preserves the measure $\mu$ on $D$, we have $\mu(D_n)= \mu(D_0)>0$  for every $n\geq 0$. Thus 
$$\mu(D)\geq \mu(\bigcup_{n=0}^\infty D_n)=\sum_{n=0}^\infty \mu(D_n)=+\infty,$$
which contradicts to $\mu(D)\leq 1$.

\medskip

{\bf Case 2:} We now assume that $\tau_n(0)$ has no subsequential limit. Hence there exists $r_0>0$ such that after passing to a subsequence, we have  $|\tau_i(0)-\tau_j(0)|>r_0$ for every $i\neq j$.   Let $D_0=D(0,r)$ be a closed disk with $0<r<r_0$, let $D_n:=\tau_n(D_0)$. Then for every $n\geq 0$, $D_n$ has radius at most $r$.  Let $i\neq j$, since $|\tau_i(0)-\tau_j(0)|>r_0$, we have  $D_i\cap D_j=\emptyset$.

Since $c_n\geq \delta>0$, we have  $(\tau_n)^*(\mu)\geq \delta \mu$ on $D$.  Thus $\mu(D_n)\geq \delta  \mu(D_0)>0$.  Hence 
$$\mu(D)\geq \mu(\bigcup_{n=0}^\infty D_n)=\sum_{n=0}^\infty \mu(D_n)=+\infty,$$
which contradicts to $\mu(D)\leq 1$.

\medskip

In both cases, we obtain a contradiction. This finishes our proof.
\end{proof}
\medskip

We now remove the hypothesis of $\sigma_n\left(p\right)\ne p$ from Lemma \ref{lem:move}. In this setting, such a family of $\{\sigma_n\}$ may exist but they must be equal to the identity map eventually. 

\begin{lem} \label{lem:move2}
Let  $\sigma_n:D\to D$  be a  family of injective analytic maps  such that 
\begin{itemize}
	\item $\sigma_n^*\mu=c_n\mu$ on $D$, $c_n>0$.

	\item $\sigma_n\to\id$ uniformly on $D$.
\end{itemize} 

Then $\sigma_n=\id$ for $n$ large enough.

\end{lem}

\begin{proof}
We assume by contradiction that our lemma is false. By passing to a subsequence we may assume $\sigma_n\neq \id$ for every $n\geq 1$. Again as in Lemma \ref{lem:move}, after passing to an iterate of $f$ and using Poincar\'e's  linearization for a repelling fixed point in $D$, we may assume that $D = D(0,1)$ is the closed unit disk, $f$ is the map given by $[\la]:z\mapsto \la z$, $|\la|>1$ and $[\la]^*\mu=d \mu$ on $D$.  By Lemma \ref{lem:move} we may assume $\sigma_n(0)=0$ for all $n\geq 1$.

\medskip

We now split into two cases, depending on whether there exists $i$ such that the multiplier $\sigma_i'(0)$ is $1$. 

\medskip

{\bf Case 1:}   There is $i \geq 1$ such that $\sigma_i'(0)=1$.  Let $\tau_n$ be the map
$$\tau_n:=[\la^{n}]\circ \sigma_i\circ [\la^{-n}]:D\to D.$$
Then since $\sigma_i'(0)=1$, we may check that $\tau_n\to \id$ uniformly on $D$. Furthermore $\tau_n\neq \id$ for every $n\geq 1$ as $\sigma_n\neq \id$, and $\tau_n^*\mu=c_i\mu$ on $D$. In particular, the family $\{\tau_n\}$ satisfies the conditions of Lemma \ref{lem:move}.  
\par 
 For each repelling periodic point $p\in D$, let $N_p\subset \Z_{\geq 0}$ be the set of $n$ such that $\tau_n(p)=p$. By Lemma \ref{lem:move},  $\Z_{\geq 0}\setminus N_p$ must be  a finite set. Fix a repelling periodic point $p \not = 0$ in $D$. Then we may find infinitely many $n$ so that $\tau_n(p)=p$ and thus $\lambda^{-n} p $ is a fixed point $\sigma_i$. This implies that $\{\sigma_i(z) = z\}$ has $\{0\}$ as an accumulation point and hence $\sigma_i=\id$, which is a contradiction. 
 
 \medskip
 
 {\bf Case 2:} We now assume that $\sigma_i'(0)\neq 1$ for every $i \geq 1$. For each $i,n \geq 1$, let  
 $$\tau_{i,n} = [\la^{n}]\circ \sigma_i\circ [\la^{-n}].$$ 
 As $n \to \infty$, this converges uniformly to $\tau_i: z\mapsto \sigma'_i(0)z$ and {\blue $\tau_i^*\mu=c_i\mu$ on $D$.} Furthermore as $\sigma'_i(0) \to 1$, we have $\tau_i \to \id$. Hence we may replace $\{\sigma_i\}$ with the family of linear maps $\{\tau_i\}$ instead.
 \par 
 Now let $p\in D$ be a repelling periodic point which is not $0$. As $\sigma_n$ are linear maps which is not $\id$, we must have $\sigma_n\left(p\right)\ne p$ for every $n\geq 1$. Lemma \ref{lem:move} then gives us a contradiction. 
 
 \medskip
 
 We get a contradiction in both cases and thus we may conclude our lemma.
 
\end{proof}

\subsection{Proof of Theorem \ref{Rigidity1}} We now begin the proof of Theorem \ref{Rigidity1}. {\blue Let $U$ and $\rho>0$ be as in its statement. If the valuation on $K$ were trivial, then $f$ would have good reduction and its equilibrium measure would be the Dirac mass at the Gauss point, where its spherical derivative has norm $1$. This would give $\Lyap\left(f\right)=0$, contrary to our assumption. Thus the valuation is non-trivial.}
\par
{\blue By \cite[Theorem B]{FRL25}, $\mu_f$ gives full mass to $\bb{P}^1\left(K\right)$, and it has no atoms. Hence $\pi_1^*\mu_f$ gives full mass to $C\left(K\right)$ and no mass to the finite singular and ramification loci of the two projections. We may therefore choose a classical smooth point $\left(x,y\right)\in U$ in its support, where both projections are unramified. Shrinking about this point, we obtain a disc $V$ around $x$ and an analytic bijection $h:V\to h\left(V\right)$ whose graph lies in $U$, with $\mu_f\left(V\right)>0$ and
\[h_*\left(\mu_f\vert_V\right)=\rho\mu_g\vert_{h\left(V\right)}.\]
In particular, $\mu_g$ gives positive mass to classical points. Since the residue characteristic is $0$, \cite[Theorem C and Corollary D]{FRL25} imply that $\Lyap\left(g\right)>0$. Fix an $L>1$ larger than the spherical Lipschitz constants of both $f$ and $g$, and let $K\geq1$ be a Lipschitz constant for both $h$ and $h^{-1}$.}
\par 
We now state a lemma that gives a uniform lower bound for $\mu_g\left(B\left(g^n\left(y\right),r\right)\right)$ from for a fixed $r>0$.  

\begin{proposition} \label{Volume1}
\begingroup\emergencystretch=3em
Fix $r>0$ and $\eps > 0$. Then there exists a finite collection of balls $\left\{B\left(x_i,r\right)\right\}_{i=1}^{n}$ such that that for $\mu_g$-almost all $y$, the set 
\[A_y = \left\{ n \in \bb{N} \mid B\left(g^n\left(y\right),r\right) = B\left(x_i,r\right) \text{ for some } i \right\}\]
satisfies $\uline{d}\left(A_y\right) \geq 1 - \eps.$ In particular there exists $\delta > 0$ such that for $n \in A_y$, we have $\mu_g(B(g^n(y),r)) \geq \delta$. 
\par\endgroup
\end{proposition}

\begin{proof}
First as $\Lyap\left(g\right)>0$, we have $\mu_g\left(\bb{H}\right) = 0$ where $\bb{H}$ denotes $\berkP \setminus \bb{P}^1\left(K\right)$. Since $\bb{H}$ is Borel measurable \cite[Proposition 5.5]{BR10} and $\mu_g$ is a Radon measure \cite[Theorem 13.37]{Ben19}, there exists an open set $U$ containing $\bb{H}$ such that $\mu_g\left(U\right) < \eps$. Now for every type I point $x \in J_g$, we consider the union of all balls $B\left(x,r\right)$ along with $U$. This is a collection of open sets that contain $J_g$ which is compact. Hence there exists a finite set of balls $B\left(x_1,r\right),\ldots,B\left(x_n,r\right)$ that along with $U$ cover $J_g$. 
\par 
By Birkhoff's ergodic theorem, the density of $n$ so that $g^n\left(x\right)$ does not lie in $U$ is $\geq 1 - \eps$. For such $n$'s we must have $B\left(g^n\left(x\right),r\right) = B\left(x_i,r\right)$ for some $i$. Since each $B\left(x_i,r\right)$ has positive mass, there must exist a uniform $\delta > 0$ so that $\mu_g\left(B\left(g^n\left(x\right),r\right)\right) \geq \delta$ as desired.
\end{proof}

Now for any $\eps > 0$, by Proposition \ref{Lyapunov2} there exists $r>0$ so that if we set 
\[B_y = \left\{ n \in \bb{N} \mid g^n: W \to B\left(g^n\left(y\right),r\right) \text{ is a bijection}  \right\},\]
where $W$ is the connected component containing $y$ in $g^{-n}(B(g^n(y),r)$, then $\uline{d}\left(B_y\right) \geq 1 - \eps$ for $\mu_g$-almost all $y$. Now applying Proposition \ref{Volume1}, there exists a finite collection of discs $\left\{B\left(x_i,r\right)\right\}_{i=1}^{n}$ so that if we let $B'_y$ be the subset of $B_y$ for which $\left(B\left(g^n\left(y\right),r\right)\right) = B\left(x_i,r\right)$, then $\uline{d}\left(B'_y\right) \geq 1 -2 \eps $. 
\par 
{\blue We may similarly apply Proposition \ref{Lyapunov3} to $f$. Since $h_*\left(\mu_f\vert_V\right)=\rho\mu_g\vert_{h\left(V\right)}$ with $\rho>0$, if a property holds for $\mu_g$-almost all $y$ in $h\left(V\right)$, then it holds for $h\left(x\right)$ for $\mu_f$-almost all $x$ in $V$. We thus obtain the following.} 

\begin{proposition} \label{Conical1}
Fix an $\eps > 0$. For $\mu_f$-almost all $x$ in $V$ where $y = h\left(x\right)$, we have: 
\begin{enumerate}
\item $\lim_{n \to \infty} \frac{1}{n} \log\left|g^n\right|'\left(y\right) = \Lyap\left(g\right)$ and $\lim_{n \to \infty} \frac{1}{n} \log \left|f^n\right|'\left(x\right) = \Lyap\left(f\right)$.

\item There exists $r>0$ and $B'_y \subseteq \bb{N}$ with $\uline{d}\left(B'_y\right) \geq 1 -2 \eps $ so that for $n \in B'_y$, the ball $B\left(g^n\left(y\right),r\right)$ is one of finitely many possible balls $\left\{B\left(x_i,r\right)\right\}_{i=1}^{n}$. Furthermore,  
\[g^n: W \to B\left(g^n\left(y\right),r\right)\]
is a bijection where $W$ is the component of $g^{-n}\left(B\left(g^n\left(y\right),r\right)\right)$ containing $y$.  We will denote the inverse branch by $g_{-n}$. 

\item There exists $A_x \subseteq \bb{N}$ with $d\left(A_x\right) \geq 1 - \eps$ so that for $n \in A_x$, if $W'$ is the component of $f^{-n}\left(B\left(f^n\left(x\right),r\right)\right)$ containing $x$, then $f^n: W' \to B\left(f^n\left(x\right),r\right)$ is a bijection, where $r$ is the same $r$ as in (2).
\end{enumerate}
\end{proposition}

Fix an $\eps > 0$, so that we have a corresponding $x \in V$ and $r>0$ that satisfies (1),(2) and (3) in Proposition \ref{Conical1}. For each positive integer $n$, we may consider $h^{-1} \circ g_{-n}$ which is a map from $B\left(g^n\left(y\right),r\right)$ to $B\left(x,r'\right)$, where $r'$ is very small. We want to precompose it with an iterate of $f$, say $f^{\theta\left(n\right)}$, so that 
\[f^{\theta\left(n\right)} \circ h^{-1} \circ g_{-n} : B\left(g^n\left(y\right),r\right) \to B\left(f^{\theta\left(n\right)}\left(x\right),R_n\right)\]
is still injective, and $\theta\left(n\right)$ is almost as large as possible, so that $R_n$ is uniformly bounded away from $0$. 
\par 
We will quantify this as follows. Recall that we have fixed a constant $L>1$ that is larger than the spherical Lipschitz constant for both $f$ and $g$ and also a $K \geq 1$ which serves as a Lipschitz constant for both $h$ and $h^{-1}$. We define $\alpha_{f,x}\left(n\right)$ to be the minimal $m$ so that {\blue $\log \left|f^m\right|'\left(x\right) \geq \left(n-1\right) \log L - \log K$}. Then $\alpha_{f,x}$ is strictly increasing for sufficiently large $n$ and there exists constants $\lambda, Q > 1$ so that 
\[m-Q \leq \alpha_{f,x}\left(m\right) \leq \frac{m \log L + Q}{\log \lambda}.\]
We similarly define $\alpha_{g,y}\left(n\right)$ to be the minimal $m$ so that {\blue $\log \left|g^m\right|'\left(y\right) \geq n \log L$}. Then again $\alpha_{g,y}$ is strictly increasing and we have a constant $Q$ so that 
\[m \leq \alpha_{g,y}\left(m\right) \leq \frac{m \log L + Q}{\log \lambda}.\]
We now set 
\[G = B'_y \cap \left(\alpha_{g,y} \circ \alpha_{f,x}^{-1}\left(A_x\right)\right).\]
For $n \in G$, we define $\theta\left(n\right) = \alpha_{f,x} \circ \alpha_{g,y}^{-1}\left(n\right).$ Given $A \subseteq \bb{N}$, we will let $\uline{d}\left(A\right)$ denote the lower density of $A$, i.e. 
\[\uline{d}\left(A\right):= \liminf_{n} \frac{\left|A \cap \left\{1,2,\ldots,n\right\}\right|}{n}.\]

\begin{proposition} \label{Conical2}
For $\eps$ sufficiently small, we have $\uline{d}\left(G\right)$ and $\uline{d}\left(\theta\left(G\right)\right) > 0$. For sufficiently large $n \in G$, we have that 
\[f^{\theta\left(n\right)} \circ h^{-1} \circ g_{-n} : B\left(g^n\left(y\right),r\right) \to B\left(f^{\theta\left(n\right)}\left(x\right),R_n\right)\]
where $\frac{r}{K^2 L^2} \leq R_n \leq r$ for all sufficiently large $n$
\end{proposition}

\begin{proof}
The proof of this is very similar to \cite[Lemma 4.14]{JX25}. As $\alpha_{f,x}, \alpha_{g,y}$ are both strictly increasing functions and there exists $c > 0$, {\blue depending only on $L$, $K$ and $\lambda$}, such that 
{\blue \[\frac{\left|\alpha_{f,x}(\bb{N}) \cap \{1,2,\ldots,m\}\right|}{m} \geq c\]}
for all sufficiently large $m$ and also for $\alpha_{g,y}$, we have 
\[\uline{d}\left(\alpha_{f,x}^{-1}\left(A_x\right)\right) \geq 1 - \frac{\eps}{c}, \quad \uline{d}\left(\alpha_{g,y} \circ \alpha_{f,x}^{-1}\left(A_x\right)\right) \geq c \left(1 - \frac{\eps}{c}\right).\]
As $\uline{d}\left(B'_y\right) \geq 1 -2 \eps $, if $\eps$ is sufficiently small we have $\uline{d}\left(G\right) > 0$. Similarly we have $\uline{d}\left(\theta\left(G\right)\right) > 0$. 
\par 
 Now let $n \in G$. Let $\alpha_{g,y}\left(m\right) = n$. Then $n$ is the smallest positive integer so that {\blue $\log \left|g^n\right|'\left(y\right) \geq m \log L$} and in particular 
{\blue \[\left(m+1\right) \log L \geq \log \left|g^n\right|'\left(y\right) \geq m \log L.\]}
Similarly since $\theta\left(n\right) = \alpha_{f,x}\left(m\right)$, for all sufficiently large $n$ we must have 
{\blue \begin{equation} \label{eq: Lipschitz1}
m\log L - \log K \geq \log \left|f^{\theta\left(n\right)}\right|'\left(x\right) \geq \left(m-1\right) \log L - \log K.
\end{equation}}
Now let $W$ be the connected component of $g^{-n}\left(B\left(g^n\left(y\right),r\right)\right)$ that contains $y$, so that 
\[g^n: W \to B\left(g^n\left(y\right),r\right)\]
is injective. It follows that 
\[\log r - \left(m+1\right) \log L \leq \log \diam\left(W\right) \leq \log r - m \log L.\]
Now applying $h^{-1}$, we obtain a disc $h^{-1}\left(W\right)$ containing $x$ satisfying 
\[\log r - \left(m+1\right) \log L - \log K \leq \log \diam\left(h^{-1}\left(W\right)\right) \leq \log r - m \log L + \log K.\]
{\blue Since $\theta(n)\in A_x$, Proposition \ref{Conical1} (3) gives a component $W'$ containing $x$ such that
\[f^{\theta(n)}:W'\longrightarrow B(f^{\theta(n)}(x),r)\]
is bijective. By \eqref{eq: Lipschitz1},
\[\log\diam(W')\geq\log r-m\log L+\log K.\]
Thus $h^{-1}(W)\subseteq W'$, so $f^{\theta(n)}$ is injective on $h^{-1}(W)$ and its image has radius at most $r$. The lower bounds for $\diam(h^{-1}(W))$ and in \eqref{eq: Lipschitz1} give $R_n\geq r/(K^2L^2)$, as desired.}
\end{proof}

 Let $h_n = f^{\theta\left(n\right)} \circ h^{-1} \circ g_{-n}$ for $n \in G$. We now claim that we may choose a subsequence $n_1,n_2,\ldots$ with $n_i \in G$ so that $f^{\theta\left(n_i\right)}\left(x\right)$ converges to a Type I point in $J_f$. 

 \begin{proposition} \label{Conical3}
 There exists a subsequence $n_1,n_2,\ldots$ with $n_i \in G$ so that $f^{\theta\left(n_i\right)}\left(x\right)$ converges to a Type I point in $J_f$.     
 \end{proposition}

 \begin{proof}
 Let $Z$ be the closure of $\left\{f^{\theta\left(n\right)}\left(x\right) \mid n \in G\right\}$. Recall that $x$ has was chosen to be $\mu_f$-generic. Since $\uline{d}\left(\theta\left(G\right)\right) > 0$, it follows that $\mu_f\left(Z\right) > 0$. As $\mu\left(\bb{H}\right) = 0$, it must be that $Z$ contains a type I point as desired.     
 \end{proof}

\begingroup\emergencystretch=3em
 Passing to this subsequence $\left\{n_k\right\}$ and then to a possibly further subsequence, we may assume that $B\left(g^{n_k}\left(y\right),r\right)$ are all the same disc $D$ by (2) of Proposition \ref{Conical1}. Hence we have a family of injective maps 
 \[h_{n_k}: D \to B\left(f^{\theta\left(n_k\right)}\left(x\right),R_{n_k}\right)\]
 with $f^{\theta\left(n_k\right)}\left(x\right)$ converging to a Type I point, say $z\in J_f$. Since $z\in J_f$, it follows that {\blue $\mu_f(B\left(z, \frac{r}{2K^2L^2}\right)) = \delta' > 0$}. As {\blue $R_{n_k} \geq \frac{r}{K^2L^2}$}, passing to a subsequence, we may assume that 
 {\blue \[B\left(z, \frac{r}{2K^2L^2}\right) \subseteq B(f^{\theta(n_k)}(x),R_{n_k}).\]}
 In particular $\mu_f\left(B\left(f^{\theta\left(n_k\right)}\left(x\right),R_{n_k}\right)\right) \geq \delta'$.

\par\endgroup
{\blue \begin{proposition} \label{Volume2}
For any Borel set $E \subseteq D$, we have
\[\mu_f\left(h_{n_k}\left(E\right)\right) = \rho d_1^{\theta\left(n_k\right)} d_2^{-n_k} \mu_g\left(E\right).\]
Furthermore, there exists $\delta > 0$ such that  
\[d_1^{\theta\left(n_k\right)} d_2^{-n_k} \geq \delta > 0\]
for all $k \geq 1$. 
\end{proposition}

\begin{proof}
As $f^*\mu_f=d_1\mu_f$, it follows that if $f:E\to f\left(E\right)$ is injective, then $\mu_f\left(f\left(E\right)\right)=d_1\mu_f\left(E\right)$. By our assumption about  proportionality of two measures,   $\mu_f\left(h^{-1}\left(E\right)\right)=\rho\mu_g\left(E\right)$ for Borel sets $E\subseteq h\left(V\right)$. Since $g_{-n_k}\left(D\right)\subseteq h\left(V\right)$, $g^{n_k}$ is injective on $g_{-n_k}\left(D\right)$, and $f^{\theta\left(n_k\right)}$ is injective on $h^{-1}\circ g_{-n_k}\left(D\right)$, we have
\[\mu_g\left(g_{-n_k}\left(D\right)\right) = d_2^{-n_k} \mu_g\left(D\right)\]
and 
\[\mu_f\left(f^{\theta\left(n_k\right)} \circ h^{-1} \circ g_{-n_k} \left(D\right)\right) = \rho d_1^{\theta\left(n_k\right)} d_2^{- n_k} \mu_g\left(D\right).\]
The same argument applies to any Borel set $E \subseteq D$. Now since $\mu_f\left(h_{n_k}\left(D\right)\right) \geq \delta' > 0$, we have
\[1 \geq \rho d_1^{\theta\left(n_k\right)} d_2^{-n_k} \mu_g\left(D\right) \geq \delta',\]
giving us $\delta > 0$ such that $d_1^{\theta\left(n_k\right)}d_2^{-n_k} \geq \delta$ as desired.
\end{proof}}

We will now show that we can further choose a subsequence of $h_{n_k}$ so that it converges as a power series on $D$, after possibly shrinking $D$. 

\begin{proposition} \label{Convergence1}
We may replace $n_k$ with a subsequence and $D$ with a smaller disc $D'$ so that $\mu_g\left(D'\right) > 0$ and $h_{n_k} : D' \to h_{n_k}\left(D'\right)$ converges to a power series $h': D' \to h(D')$.
\end{proposition}

\begin{proof}
Let $y' \in J_g$ be a point that lies on $D$. After translating and possibly inverting, we may assume that $D$ is centered at $y' = 0$. Similarly after passing to a subsequence and possibly inverting, we may assume there exists $R > 0$ such that $h_{n_k}\left(D\right) \subseteq D\left(0,R\right)$, where $D\left(0,R\right)$ represents the disc centered at $0$ with Euclidean radius $\leq R$ instead of the spherical metric.  
\par 
We may then expand $h_{n_k}$ as a convergent power series $h_{n_k} = a_{k,0} + a_{k,1} z + a_{k,2} z^2 + \cdots$. Since $D$ has radius $r$, we must have $\left|a_{k,n}\right| \leq R r^{-n}$ as $h_{n_k}\left(D\right) \subseteq D\left(0,R\right)$. We now let $D' = D\left(0, \frac{r}{2}\right)$. Since $0$ is a point of $J_g$, it follows that $\mu_g\left(D'\right) > 0$. Then to show that there is a subsequence of $h_{n_k}$ that converges on $D'$, it suffices to show that there is a subsequence of $h_{n_k}$ such that $a_{k,i}$ converges pointwise for each $i$. To show this, we will induct on the assumption that we can find a subsequence so that $a_{k,i}$ converges for all $i < N$. We then show we can further extract a subsequence so that $a_{k,N}$ converges too. A diagonalization argument then gives us the subsequence we want. 
\par 
Now let's assume that $a_{k,i}$ converges to $a_k$ for all $i < N$. If $a_{k,N}$ has a subsequence that converges to a Type I point, then we can simply pass to that subsequence and we are done. Else, passing to a subsequence, we may assume that there is $r' > 0$ such that $\left|a_{k,N} - a_{k',N}\right| \geq r'$ for all $k \not = k'$. We will show that this leads to a contradiction. 
\par 
Consider a $z \in D$ with $0 < \left|z\right| < 1$ small enough so that 
\begin{equation} \label{eq: Rigidity1}
\left|a_{k,n} z^{n-N}\right| < r'
\end{equation}
for all $k$ and $n \geq N+1$. For such a $z$, as it is non-zero, we can choose $M > 0$ so that for all $k,k' \geq M$, we have 
\[\left|a_{k,i} - a_{k',i}\right| < r' \left|z\right|^N\]
for all $i < N$. It follows that 
\[\left|h_{n_k}\left(z\right) - h_{n_{k'}}\left(z\right)\right| = \left|a_{k,N} - a_{k',N}\right| z^N \geq r' \left|z\right|^N.\]
Let $C$ be a uniform Lipschitz constant for $h_{n_k}$ for all $k \geq 1$. Then if we take $\eps > 0$ such that $C \eps < r' \left|z\right|^N$, it follows that $h_{n_k}\left(D\left(z, \eps\right)\right)$ are all disjoint for all $k \geq M$. 
\par 
Now since $0$ is a Type I point in $J_g$, we may find a non-zero $z$ satisfying \eqref{eq: Rigidity1} that is also a Type I point in $J_g$. Then we for any $\eps > 0$, it follows that $\mu_g\left(D\left(z,\eps\right)\right) > 0$. If we choose $\eps < \frac{r}{2} \left|z\right|^N$, it follows that 
{\blue \[\mu_f\left(h_{n_k}\left(D\left(z,\eps\right)\right)\right) \geq \rho d_1^{\theta\left(n_k\right)} d_2^{-n_k} \mu_g\left(D\left(z,\eps\right)\right) \geq \rho \delta \mu_g\left(D\left(z,\eps\right)\right).\]}
Since each $h_{n_k}\left(D\left(z, \eps\right)\right)$ is disjoint, this implies that $\mu_f$ has infinite mass which is a contradiction. Hence $a_{k,N}$ must have a Type I point as a subsequential limit. This finishes our induction as desired. 
\end{proof}
 
We may now assume that $h_{n_k}$ converges to a power series $h'$ on $D'$ with $\mu_g\left(D'\right) > 0$. Since $h_{n_k}$ is injective, it follows that 
\[\diam\left(h_{n_k}\left(D\right)\right) = \left|h'_{n_k}\left(x\right)\right| \diam\left(D\right)\]
for any $x \in D$. As $\diam\left(h_{n_k}\left(D\right)\right)$ is bounded from below, it follows that $\liminf_{x \in D} \left|h'_{n_k}\left(x\right)\right| > 0$ and in particular, we must have $\left|h'_{n_k}\left(x\right)\right|$ being all equal for sufficiently large $k$. It then follows that $h_{n_k}\left(D'\right)$ are all equal for $k$ sufficiently large, say to $D''$. We may then assume that $h_{n_k}\left(D'\right) = D''$ for all $k$. Our aim now is to show that $\left\{h_{n_k}\right\}$ is a finite set. This will then imply that $h_{n_k} = h_{n_{k'}}$ for $k' > k$ which will provide for us a local semiconjugacy between $f$ and $g$.

\begin{proof}[Proof of Theorem \ref{Rigidity1}]
{\blue The limit $h':D'\to D''$ is an analytic bijection. By Proposition \ref{Volume2}, the common image $D''$ implies
\[\rho d_1^{\theta\left(n_k\right)}d_2^{-n_k}=\frac{\mu_f\left(D''\right)}{\mu_g\left(D'\right)}=:A>0.\]
Hence $\left(h_{n_k}\right)_*\left(\mu_g\vert_{D'}\right)=A^{-1}\mu_f\vert_{D''}$ and  $h'_*\left(\mu_g\vert_{D'}\right)=A^{-1}\mu_f\vert_{D''}$. For each positive integer $k$, consider the analytic map
\[\sigma_k = \left(h'\right)^{-1} \circ h_{n_k} : D' \to D'.\]
Then $\sigma_k$ is injective and $\sigma_k^*\left(\mu_g\vert_{D'}\right)=\mu_g\vert_{D'}$. Furthermore as $k \to \infty$, we have $\sigma_k \to \id$ and $\mu_g\left(D'\right) > 0$. We may now apply Lemma \ref{lem:move2} to the family and obtain that $\left\{\sigma_k\right\}$ is a finite family. Hence $\left\{h_{n_k}\right\}$ must be a finite family, and passing to a subsequence, we may assume that all of them are equal to each other. We thus get
\[f^{\theta\left(n_i\right)} \circ h^{-1} \circ g_{-n_i} = f^{\theta\left(n_j\right)} \circ h^{-1} \circ g_{-n_j}\]
 as maps from $D'$ to $D''$. For each $j$, the infinite set
\[X_j=\left\{\left(h^{-1}\circ g_{-n_j}\left(t\right),g_{-n_j}\left(t\right)\right)\mid t\in D'\right\}\]
lies in $C$, and its image under $\left(f^{\theta\left(n_j\right)},g^{n_j}\right)$ is the same graph $\left\{\left(h_{n_i}\left(t\right),t\right)\mid t\in D'\right\}$, independently of $j$.}
Hence for two $j' > j$, we have that 
\[\left(f^{\theta\left(n_j\right)}, g^{n_j}\right)\left(C\right) \cap \left(f^{\theta\left(n_{j'}\right)}, g^{n_{j'}}\right)\left(C\right)\]
has infinite intersection and thus they must actually be equal as curves. Let $a = \theta\left(n_{j'}\right) - \theta\left(n_j\right)$ and $b = n_{j'} - n_j$. Then $C' = \left(f^{\theta\left(n_j\right)}, g^{n_j}\right)\left(C\right)$ is fixed by $\left(f^a,g^b\right)$. By the discussion after the statement of Lemma 6.1 in \cite{GNY19}, it follows that $\deg\left(f^a\right) = \deg\left(g^b\right)$, i.e. $d_1^a=d_2^b$. 
\par 
\begingroup\emergencystretch=3em
Now consider the forward iterates of $C$ under $\left(f^a,g^b\right)$, which we will label as $C_1, C_2,\ldots$. Since $\left(f^{\theta\left(n_j\right)},g^{n_j}\right)$ commutes with $\left(f^a,g^b\right)$, it follows that $(f^{\theta(n_j)},g^{n_j})(C_i) = C'$. Hence $C_i \subseteq (f^{\theta(n_j)},g^{n_j})^{-1}(C')$. This implies that the sequence $C_1,C_2,\ldots$ take on only finitely many possible curves and in particular $C$ must be preperiodic under $\left(f^a,g^b\right)$ as desired. 

\par\endgroup
\end{proof}

We need another rigidity statement when $f$ is a Latt{\`e}s map {\blue that does not have potentially good reduction} and $g$ is a map with good reduction. 

\begin{proposition} \label{Rigidity2}
Let $f,g$ be two rational maps of degree $d \ge 2$ and assume that $f$ is a Latt{\`e}s map {\blue that does not have potentially good reduction} and $g$ has good reduction. Then for any curve $C \subseteq \left(\bb{P}^1\right)^2$ such that $\pi_1,\pi_2$ projects dominantly, we  have 
\[
\left(\deg\pi_2\right)\pi_1^*\mu_f
\neq
\left(\deg\pi_1\right)\pi_2^*\mu_g.
\]
\end{proposition}

\begin{proof}
Since $g$ has good reduction, we know that $\mu_g$ is a delta mass at the Gauss point. As $\pi_2: C^{\an} \to \berkP$ is a finite morphism, this implies that $\pi_2^*\mu_g$ is supported on a finite set of points. On the other hand, we know that $\Supp \mu_f$ is an interval and thus infinite in size. As $\pi_1: C^{\an} \to \berkP$ is also finite, it cannot be that $\Supp (\pi_1^* \mu_f)$ is a finite set of points and hence 
\[
\left(\deg\pi_2\right)\pi_1^*\mu_f
\neq
\left(\deg\pi_1\right)\pi_2^*\mu_g.
\]
\end{proof}

\section{Flat Metrics of Curves in Families} \label{sec: Flat}

In this section, given two rational maps $f,g$ and a curve $C$ over a non-trivially valued field $K$, we will define a local energy pairing $\langle f , g \rangle_{v,C}$. We will then prove that it is continuous under some mild hypothesis in Proposition \ref{FlatContinuity1}.  
\par
Let $S$ be a quasi-projective variety over a number field $K$ and let $\cal{C}$ be a family of smooth projective curves over $S$, i.e. a flat family whose fibers are smooth, projective and geometrically irreducible curves. Then for each non-trivial $s \in S^{\an}$, we obtain a non-trivially valued field $\cal{H}\left(s\right)$ and a geometrically irreducible curve $\cal{C}_s$ over $\cal{H}\left(s\right)$. Let $L$ be a line bundle on $\cal{C}$ such that $L_s$ is numerically trivial on $\cal{C}_s$ for each $s \in S$ and hence $s \in S^{\an}$. 
\par 
Recall from \cite{YZ17} that for a smooth projective curve $C$ and numerically trivial line bundle $L$ over a non-trivially valued field $K$, there exists a metrized line bundle $\ovl{L}$ with underlying bundle $L$ such that $c_1\left(\ovl{L}\right) = 0$ on $C^{\an}$, where $c_1\left(\ovl{L}\right)$ is the Chambert--Loir measure associated to $\ovl{L}$. Furthermore this metric is unique up to constants. We will say that $\ovl{L} = \left(L, \left| \cdot \right|\right)$ is flat, or that $\left| \cdot \right|$ is a flat metric. 
\par 
Hence for each $s \in S^{\an}$, we obtain a flat metric on $L_s$. We would like to say that these flat metrics vary in a family, i.e. there exists an integrable adelic line bundle $\ovl{L} \in \wh{\Pic}\left(\cal{C}/\bb{Z}\right)$ such that the restriction to each $s \in S^{\an}$ gives us a flat line bundle $\ovl{L}_s$. Here, $\wh{\Pic}\left(\cal{C}/\bb{Z}\right)$ is the space of adelic line bundles on $\cal{C}/K$, where we view $K$ as a finitely generated field over $\bb{Q}$, see \cite[Section 2]{YZ24} for the definitions and basic facts.  
\par 
When $L_s$ is numerically trivial for each $s$, such an adelic line bundle $\ovl{L}$ is also called an admissible adelic extension of $L$ \cite[Definition 2.5]{Yua24}. 

\begin{proposition} \label{FlatMetric1}
Let $\cal{C}/S$ be a flat family of smooth projective curves over $S$ where $S$ is a quasi-projective variety over a number field $K$. Assume further that $S$ is normal and integral. Let $L$ be a line bundle on $\cal{C}$ for which $L_s$ is numerically trivial for each $s \in S$. Then there exists an admissible adelic extension $\ovl{L}$ of $L$. 
\end{proposition}

\begin{proof}
\cite[Proposition 2.6]{Yua24} proves this in the case that $S$ is a quasi-projective variety over $\bb{Z}$ instead of $K$. In our case, we may take some model $\cal{C}'/\cal{S}$ of $\cal{C}/S$, along with $\cal{L}$ extending $L$, so that $\cal{S}$ is now a quasi-projective variety over $\bb{Z}$. We may then apply \cite[Proposition 2.6]{Yua24} and conclude. 
\end{proof}

The flat metric is important as it allows us to generalize the energy pairing of Favre--Rivera-Letelier \cite{FRL06} (see also Fili \cite{Fil17}), of two rational maps $\left\langle f ,g \right\rangle$ to the more general setting of $\left\langle f ,g \right\rangle_C$ for a pair of rational maps $\left(f,g\right)$ and a non-vertical/horizontal curve $C \subseteq \left(\bb{P}^1\right)^2$.  When $C = \Delta$ is the diagonal, we recover the original Favre--Rivera-Letelier pairing. The energy pairing has been used by DeMarco--Krieger--Ye \cite{DKY20, DKY21} and Poineau \cite{Poi24b} to prove instances of the uniform dynamical Bogomolov when $C = \Delta$. This more general pairing has been used by Mavraki--Schmidt \cite{MS25} in the proof of the uniform dynamical Bogomolov for $1$-parameter families. 
\par 
We first define the local energy pairing for a pair of rational maps $\left(f,g\right)$ of degree $d \ge 2$ that are defined over a non-trivially valued field $\left(K,\left| \cdot \right|\right)$. Let $C \subseteq \left(\bb{P}^1\right)^2$ be our geometrically irreducible curve over $K$, not necessarily smooth, and let it be of bidegree $\left(d_1,d_2\right)$, with $d_i > 0$. {\blue Here $d_i=\deg(\pi_i|_C)$.} Let $\ovl{L}_f$ be the metrized line bundles, unique up to constants, on $\bb{P}^1$ with underlying line bundle $O\left(1\right)$ and with $c_1\left(\ovl{L}_f\right) = \mu_f$, the equilibrium measure of $f$. We similarly define $\ovl{L}_g$. We then set 
\[\ovl{L} = d_2 \pi_1^* \ovl{L}_f - d_1 \pi_2^* \ovl{L}_g.\]
If $L$ is the underlying line bundle of $\ovl{L}$, then $L$ is numerically trivial. Let $\pi: C' \to C$ be its normalization. Then we may pullback and obtain $\pi^*\ovl{L}$ with underlying line bundle $\pi^*L$ and furthermore, $\pi^*L$ is still numerically trivial. Hence has a flat metric, which we denote by $\ovl{M}$. We now set $\ovl{N} = \pi^*\ovl{L} - \ovl{M}$, where the underlying line bundle of $\ovl{N}$ is $N \simeq O_{C'}$ the trivial bundle. We then define the local energy pairing 
\[\left\langle f , g \right\rangle_C = - \ovl{N}^2 \geq 0,\]
where we use the sections $\left(\bf{1}, \bf{1}\right)$ to compute $\ovl{N}^2$. The quantity $-\ovl{N}^2$ is independent of the choice of isomorphism between $N$ and $O_{C'}$ and thus $\left\langle f ,g \right\rangle_C$ is well-defined. Now if $\left(f,g,C\right)$ are defined over a number field $K$, we define the global energy pairing 
\[\left(f,g\right)_C = \sum_{v \in M_K} n_v \left\langle f ,g \right\rangle_{C,v}\]
where $\left\langle f , g \right\rangle_{C,v}$ is the local pairing over the valued field $\bb{C}_v$ and $n_v = \left[K_v:\bb{Q}_p\right]/\left[K:\bb{Q}\right]$ is the normalization factor. 
\par 
The global energy pairing is related to the {\blue normalized} canonical height $\h_{\left(f,g\right)}\left(C\right)$ by the following inequality. Note that by definition, we have 
{\blue \[
\h_{(f,g)}(C)
=\frac{(\pi_1^*\ovl{L}_f+\pi_2^*\ovl{L}_g)^2}{[K:\bb{Q}](d_1+d_2)}
=\frac{2(\pi_1^*\ovl{L}_f)\cdot(\pi_2^*\ovl{L}_g)}{[K:\bb{Q}](d_1+d_2)}
\]}
as $\ovl{L}_f^2=\ovl{L}_g^2=0$. 

\begin{proposition} \label{GlobalPairing1}
We have the inequality
{\blue \[\h_{(f,g)}(C) \geq \frac{1}{d_1d_2(d_1+d_2)}(f,g)_C.\]}
\end{proposition}

\begin{proof}
Working over the number field $K$ {\blue on the normalization of $C$}, let $\ovl{L}_f,\ovl{L}_g$ be the invariant adelic line bundles for $f,g$ and $\ovl{L}=d_2\pi_1^*\ovl{L}_f-d_1\pi_2^*\ovl{L}_g$, then 
{\blue \[
\begin{aligned}
\ovl{L}^2
&=(d_2\pi_1^*\ovl{L}_f-d_1\pi_2^*\ovl{L}_g)^2
=-2d_1d_2(\pi_1^*\ovl{L}_f)\cdot(\pi_2^*\ovl{L}_g)\\
&=-[K:\bb{Q}]d_1d_2(d_1+d_2)\h_{(f,g)}(C).
\end{aligned}
\]}
On the other hand, we may find a flat adelic line bundle $\ovl{M}$ with same underlying line bundle as $\ovl{L}$, so that $\ovl{N}=\ovl{L}-\ovl{M}$ has the trivial underlying line bundle and 
{\blue \[-\ovl{N}^2 = [K:\bb{Q}](f,g)_C.\]}
Due to flatness of $\ovl{M}$ and $N$ being trivial, we have 
\[\ovl{L}^2 = \ovl{N}^2 + \ovl{M}^2\]
and furthermore, $\ovl{M}^2 \leq 0$ \cite[Section 2.2]{YZ17}. thus 
{\blue \[
d_1d_2(d_1+d_2)\h_{(f,g)}(C)
=\frac{-\ovl{L}^2}{[K:\bb{Q}]}
\geq\frac{-\ovl{N}^2}{[K:\bb{Q}]}
=(f,g)_C
\]}
as desired.
\end{proof}

Now let $S$ be a quasi-projective variety over s number field $K'$ and let $S$ parametrize a family of triples $\left(f,g,C\right)$ where $\left(f,g\right)$ are rational maps on $\bb{P}^1$ of degree $d \ge 2$ and $C$ a degree $\left(d_1,d_2\right)$ curve. For each non-trivially valued $s \in S^{\an}$, we obtain a complete valued residue field $\cal{H}\left(s\right)$ and a triple $\left(f,g,C\right)$ defined over $\cal{H}\left(s\right)$ and hence a the pairing $\left\langle f , g \right\rangle_{C,s}$. This gives us a function
\[\left\langle f , g \right\rangle_{C,s} : S^{\an} \setminus S^{\triv} \to \bb{R}_{\geq 0}\]
where $S^{\triv}$ is the set of $s \in S^{\an}$ for which the residue field $\cal{H}\left(s\right)$ is trivially valued. We wish to prove that this function is continuous. We will instead prove a weaker claim that we may cover $S$ with finitely many locally closed subvarieties $Z_i$ so that the function restricted to each $Z_i^{\an}$ is continuous.

\begin{proposition} \label{FlatConstruction2}
Let $\cal{C}/S$ be a flat family of geometrically irreducible, projective curves where $S$ is a quasi-projective variety over a number field $K$. We may cover $S$ with finitely many locally closed irreducible normal subvarieties $Z_1,\ldots,Z_m$ such that for each $i$:

\begin{enumerate}
\item There exists a flat family of smooth projective curves $\cal{Y}_i/Z_i$ where $Z_i$'s is a quasi-projective variety over $K$ with a morphism $\pi:\cal{Y}_i/Z_i \to \cal{C}_{Z_i}/Z_i$ where $\cal{C}_{Z_i}$ is the restriction of $X$ to $Z_i$.

\item For each $s \in Z_i$, the morphism $\pi: \cal{Y}_{i,s} \to \cal{C}_{i,s}$ is birational.
\end{enumerate}
\end{proposition}

\begin{proof}
We may assume that $S$ is irreducible. We note that it suffices to construct a Zariski open $Z_1$ satisfying the properties we want by induction on the complement $S \setminus Z_1$, where finiteness follows from noetherianess. 
\par 
To obtain $Z_1$, let $\eta \in S$ be the generic point and let $\cal{C}_{\eta}/k\left(S\right)$ be the generic fiber where $k = K'$. We may resolve its singularities by taking a normalization and so obtaining $\pi: Y_{\eta} \to X_{\eta}$ over $k\left(S\right)$ with $Y_{\eta}$ being a smooth projective curve. 
\par 
This spreads out to a family $\cal{Y}/U'$ of smooth projective curves with a morphism $\pi: \cal{Y}/U' \to \cal{C}/U$ for some Zariski open $U$ of $S$. By shrinking $U$, we may assume that $U$ is normal since $S$ is integral. Furthermore, each fiber $\pi: \cal{Y}_s \to \cal{C}_s$ is birational. Finally we may shrink $U'$ so that $\pi: U' \to U$ is isomorphic as $k\left(U'\right)$ is isomorphic to $k\left(U\right)$. We thus obtain our $Z_1$ by setting $Z_1 = U$ as desired.    
\end{proof}

We now show that when restricted to such $Z_i$'s, our function $\left\langle f , g \right\rangle_{C,s}$ is continuous. Let $S$ be our quasi-projective variety parametrizing our triples $\left(f,g,C\right)$. By Proposition \ref{FlatConstruction2}, we may cover $S$ with $Z_i$'s so that on each $Z_i$, our family of curves $\cal{C}_{Z_i}/Z_i$ has a simultaneous resolution $\pi: \cal{C}'/Z_i \to \cal{C}_{Z_i}/Z_i$. Before we begin the proof, as we will have to use the theory of adelic line bundles over quasi-projective varieties by Yuan--Zhang \cite{YZ24}, we will quickly recall some definitions and facts here. 
\par 
Let $X$ be a flat integral projective scheme over $\Spec \bb{Z}$. Then an arithmetic $\bb{Q}$-divisor $\cal{D}$ is a pair $\left(D,g_{\cal{D}}\right)$ where $D$ is a $\bb{Q}$-Cartier divisor and $g_{\cal{D}}$ a continuous Green's function for the divisor $D$ on $X\left(\bb{C}\right)$. Now if $U$ is a quasi-projective scheme flat over $\Spec \bb{Z}$, we define the space of arithemtic model divisors $\wh{\Div}_{\bb{Q}}\left(U\right)_{\mathrm{mod}}$ as 
\[\wh{\Div}_{\bb{Q}}\left(U\right)_{\mathrm{mod}} = \lim_{\rightarrow} \wh{\Div}_{\bb{Q}}\left(X\right)\]
where $X$ ranges over all compactifications of $U$. Now fixing some compactification $\ovl{U}$, we may define a boundary divisor $\cal{E}$ which is a model divisor whose support is $\ovl{U} \setminus U$ and $g_{\cal{E}} > 0$ everywhere. We may then endow $\wh{\Div}_{\bb{Q}}\left(U\right)_{\mathrm{mod}}$ with a suitable boundary topology and form its completion, giving us the space of adelic divisors $\wh{\Div}_{\bb{Q}}\left(U\right)$. Quotienting out by a suitable equivalence gives us the space of adelic line bundles $\wh{\Pic}\left(U/\bb{Z}\right)$. 
\par 
Finally if $X$ is simply a quasi-projective variety over a number field $K$, we will let 
\[\wh{\Pic}\left(X/\bb{Z}\right):= \lim_{\rightarrow} \wh{\Pic}\left(\cal{U}/\bb{Z}\right)\]
where the limit is taken over all quasi-projective models $\cal{U}$ of $X$ over $\Spec \bb{Z}$. 
\par 
The space $\wh{\Pic}\left(X/\bb{Z}\right)$ is usually too big and we will work with a subset of it known as the integrable adelic line bundles. All adelic line bundles/divisors we will use are integrable and so we will simply impose integrability as an additional condition to being an adelic divisor. We can now begin our proof that the function $\left\langle f , g \right\rangle_{C,s}$ is continuous. We will further assume that our base $S$ is normal, so that we can apply the theory of Deligne pairings for adelic line bundles \cite[Section 4]{YZ24}. 

\begin{proposition} \label{FlatContinuity1}
Assume that $S$ is both integral and normal. Also assume that our family of curves $\cal{C}/S$ has a simultaneous resolution $\pi: \cal{C}'/S \to \cal{C}/S$. Then the function
\[\left\langle f , g \right\rangle_{C,s} : S^{\an} \setminus S^{\triv} \to \bb{R}_{\geq 0}\]
is continuous. 
\end{proposition}

\begin{proof}
As $S$ parametrizes a family of triples $\left(f,g,C\right)$, we obtain a morphism $f: \bb{P}^1_S \to \bb{P}^1_S$ which is polarizable for the line bundle $O\left(1\right)$. By \cite[Theorem 6.1.1]{YZ24}, we have a $f$-invariant adelic line bundles $\ovl{L}_f \in \wh{\Pic}\left(\bb{P}^1_S/\bb{Z}\right)$ with underlying line bundle $\tilde{L}_f \in \wh{\Pic}\left(\bb{P}^1_S/\bb{Q}\right)$. Similarly we obtain $\ovl{L}_g$. Then viewing $\cal{C}/S$ as embedded in $\left(\bb{P}_S^1\right)^2$, we may form the adelic line bundle 
\[\ovl{L} = d_2 \pi_1^* \ovl{L}_f - d_1 \pi_2^* \ovl{L}_g \in \wh{\Pic}\left(\cal{C}/\bb{Z}\right).\]
Then if $\tilde{L}$ is its underlying line bundle, restricting to each fiber gives us a numerically trivial line bundle $L_s$. Now we may pull this back under $\pi$ to obtain $\pi^*\ovl{L} \in \wh{\Pic}\left(\cal{C}'/\bb{Z}\right)$ with underlying line bundle $\pi^*L$. As $\pi^*L$ is numerically trivial for each $s \in S$, there exists a flat adelic metric on it given by an adelic line bundle $\ovl{M} \in \wh{\Pic}\left(\cal{C}'/\bb{Z}\right)$. Then 
\[\ovl{N} = \pi^*\ovl{L} - \ovl{M} \in \wh{\Pic}\left(\cal{C}'/\bb{Z}\right)\]
is an adelic line bundle with underlying line bundle isomorphic to $O_{\cal{C}'}$. Furthermore for each $s \in S^{\an}$, we have 
\[-\ovl{N}^2_s = \left\langle f , g \right\rangle_{C,s}\]
where $\ovl{N}^2_s$ is the induced metrized line bundle over the residue valued field $\cal{H}\left(s\right)$. Now consider the Deligne pairing 
\[\left\langle \ovl{N}, \ovl{N} \right\rangle \in \wh{\Pic}\left(S/\bb{Z}\right).\]
Then there is a global section of $\left\langle N , N \right\rangle$ given by $\left\langle \mathbf{1}, \mathbf{1} \right\rangle$ where $\bf{1}$ is a constant section in $O_{\cal{C}'}$. Taking its norm, it follows that 
\[-\log \left|\left\langle \mathbf{1}, \mathbf{1} \right\rangle\right|_s: S^{\an} \to \bb{R}\]
is a continuous function. When $\cal{H}\left(s\right)$ is not trivially valued, the fiberwise formula \cite[Sections 4.2.2 and 4.6.2]{YZ24} imply that 
\[-\log \left|\left\langle \mathbf{1}, \mathbf{1} \right\rangle\right|_s = \ovl{N}_s^2\]
and so the function $-\ovl{N}_s^2 = \left\langle f ,g \right\rangle_{C,s}$ is continuous on $S^{\an} \setminus S^{\triv}$ as desired.
\end{proof}

\section{Uniform Energy Pairing Bounds} \label{sec: EnergyPairing}
{\blue The main results of this section are the uniform local energy bounds in Proposition~\ref{IntroUniformEnergy1} and their variants in Propositions~\ref{UniformEnergy2}--\ref{UniformEnergy4}.} Recall from Section \ref{sec: Flat} that if $\left(K,\left| \cdot \right|\right)$ is a complete valued field, $f,g: \bb{P}^1 \to \bb{P}^1$ two rational maps of degree $d \ge 2$ over $K$ and $C \subseteq \left(\bb{P}^1\right)^2$ a geometrically irreducible degree $\left(d_1,d_2\right)$-curve over $K$, then we have a local energy pairing $\left\langle f ,g \right\rangle_C$. We have the following basic positivity property of the pairing. 

\begin{proposition} \label{Pairing1}
The inequality 
\[\left\langle f , g \right\rangle_C \geq 0\]
holds, with equality if and only if 
\[d_2 \pi_1^* \mu_f = d_1 \pi_2^* \mu_g.\]
\end{proposition}

\begin{proof}
The inequality follows from the local Hodge index theorem applied on $\ovl{N}$ and the curve $C'$ \cite[Theorem 1.1]{YZ17}. Then equality holds if and only if $c_1\left(\ovl{N}\right) = 0$, which is equivalent to
\[d_2 \phi^* \pi_1^* \mu_f = d_1 \phi^* \pi_2^* \mu_g.\]
Since $\phi: C' \to C$ is birational, the map $\phi: \left(C'\right)^{\an} \to C^{\an}$ is an isomorphism outside of a finite set of Type I points and as $\mu_f,\mu_g$ do not charge any Type I points, it follows that 
\[d_2 \pi_1^*\mu_f = d_1 \pi_2^* \mu_g\]
as desired. 
\end{proof}

We note if $K'$ is another complete valued field containing $K$, then the quantity $\left\langle f ,g \right\rangle_C$ remains invariant if we work over $K'$ instead of $K$. When $K$ is archimedean this is clear. For non-archimedean $K$, because $\ovl{N}$ is constructed from line bundles coming from dynamics, there exists a sequence of model line bundles $\left\{\ovl{N}_{i}\right\}_{i \geq 1}$ such that if $K'$ is any complete value field extending $K$, then the base change $\left\{\ovl{N}_{K',i}\right\}$ converges to $\ovl{N}_{K'}$. By \cite[Remark 3.10]{Gub07}, the local intersection number for model metrics is invariant, i.e. $\ovl{N}^2_i = \ovl{N}_{K',i}^2$. Thus taking limits, we obtain $\ovl{N}^2 = \ovl{N}_{K'}^2$ as desired.
\par 
Now if $K$ is a number field and $f,g,K$ are defined over $K$, we may consider the global energy pairing 
\[\left(f,g\right)_C := \sum_{v \in M_K} n_v \left\langle f , g \right\rangle_{C,v}\]
where $\left\langle f , g \right\rangle_{C,v}$ is the energy pairing over the field $\bb{C}_v$.
\par 
Let $X = \Rat_d \times \Rat_d \times C_{d_1,d_2}$ be the moduli space of all triples $\left(f,g,C\right)$ where $f,g$ are rational maps of degree $d \ge 2$ and $C_{d_1,d_2}$ is a degree $\left(d_1,d_2\right)$-curve. This is a fine moduli space defined over $\bb{Q}$, and even over $\bb{Z}$ although the latter is not important for us. By Propositions \ref{FlatConstruction2} and \ref{FlatContinuity1} and taking normalizations, we may find finitely many normal quasi-projective subvarieties $\left\{Z_k\right\}$ that cover $X$. Then each $Z_k$ can be considered as paremetrizing a family of triples $\left(f,g,C\right)$. As $Z_k$ is normal, we know that the function 
\[\left\langle f , g \right\rangle_{C,s}: Z_k^{\an} \setminus Z_k^{\triv} \to \bb{R}_{\geq 0}\]
is continuous. We will now restrict ourselves to each $Z_k$. We first study its preperiodic locus.

\begin{proposition} \label{WeaklySpecial1}
There exists a countable collection of subvarieties $V_1,V_2,\ldots,$ of $Z_k$ such that if $s \in Z_k$ corresponds to a triplet $\left(f_s,g_s,C_s\right)$ over the residue field $\ovl{\bb{Q}}\left(s\right)$, then $C_s$ is preperiodic if and only if $s \in V_i$ for some $i$. 
\end{proposition}

\begin{proof}
It suffices to handle the case of $X = \Rat_d \times \Rat_d \times C_{d_1,d_2}$. For each pair $\left(m,n\right)$ of positive integers, we let $V_{m,n}$ be the locus for which $\left(f_s,g_s\right)^m\left(C_s\right) = \left(f_s,g_s\right)^n\left(C_s\right)$. We claim that this cuts out a Zariski closed set in $X$. Indeed, let $C'_{d_1,d_2}$ denote the space of all degree {\blue $\left(d_2,d_1\right)$}-bihomogeneous forms up to scalars. We may view $C_{d_1,d_2}$ as an open variety of $C'_{d_1,d_2}$. Then we have an algebraic map 
\[\Phi^m: \Rat_d \times \Rat_d \times C'_{d_1,d_2} \to \Rat_d \times \Rat_d \times C'_{d^m d_1, d^m d_2}\]
given by
\[\left(f,g,C\right) \mapsto \left(f,g, \left(f^m,g^m\right)_*C\right)\]
where $\left(f^m,g^m\right)_*$ denotes the pushforward operator on cycles (for example, see the proof of \cite[Theorem 3.8]{Hut18}.  
For $m > n$, we have a multiplication by $d^{m-n}$ map sending $C'_{d^n d_1, d^n d_2}$ to $C'_{d^m d_1, d^m d_2}$, given by $C \mapsto C^{d^{m-n}}$. 
\par 
Since $C$ is irreducible, so is $\left(f^m,g^m\right)\left(C\right)$ and thus the locus we are interested in is thus given by $V_{m,n} = \left\{\Phi^m = d^{m-n} \cdot \Phi^n\right\}$ which is Zariski closed. Restricting to $C_{d_1,d_2}$ gives us a closed subvariety. Now letting $\left(m,n\right)$ vary over all pairs of non-negative integers gives us our countable collection $V_1,V_2,\ldots$ as desired.    
\end{proof}

We can now our state our first main uniform lower bound on the energy pairing. Let $V_1,V_2,\ldots$ be our countable collection of closed subvarieties whose union is the preperiodic locus. Following Section \ref{sec: LocalSubSec}, we fix local height functions $\lambda_{\partial V_i}$ for each subvariety $V_i$ and a local height function $\lambda_{\partial Z_k}$ for the boundary of $Z_k$. We will prove a uniform energy bound on $\left\langle f ,g \right\rangle_C$, assuming that our triplet $\left(f,g,C\right)$ is not too close to any $V_i$ and also that the Lyapunov exponent $\Lyap\left(f\right)$ of $f$ is not too small. 

\begin{proposition} \label{UniformEnergy1}
Let $\left(K,\left| \cdot \right|\right)$ be a complete non-trivially valued field of characteristic $0$. Let $\delta, c_1,c_2,\ldots$ be positive constants. Then there exists constants $N > 0$ and $c > 0$ depending on $K,\delta,c_1,c_2,\ldots$ such that if $s = \left(f,g,C\right) \in Z_k\left(K\right)$ satisfies:
\begin{enumerate}
\item $\lambda_{V_i}\left(s\right) \leq c_i \lambda_{\partial Z_k}\left(s\right)$ for all $i \geq 1$;

\item $\Lyap(f) \geq \delta \lambda_{\partial Z_k}(s)$;

\item $\lambda_{\partial Z_k}\left(s\right) \geq N$.
\end{enumerate}
Then we have the uniform lower bound 
\[\left\langle f,g \right\rangle_C \geq c \lambda_{\partial Z_k}\left(s\right).\]
\end{proposition}

\begin{proof}
We assume otherwise. Then we have a sequence of points $s_n = \left(f_n,g_n,C_n\right)$ such that 
\[\lambda_{\partial V_i}(s_n) \leq c_i \lambda_{\partial Z_k}(s_n), \Lyap(f_n) \geq \delta \lambda_{\partial Z_k}(s_n), \lambda_{\partial Z_k}(s_n) \geq n \text{ and } \langle f_n, g_n \rangle_{C_n} \leq \frac{1}{n} \lambda_{\partial Z_k}(s_n).\]
We now degenerate using ultrafilters with the scaling $\eps_n = \lambda_{\partial Z_k}\left(s_n\right)^{-1}$ as in Section \ref{sec: DegenUltra}. By Proposition \ref{DegenerationExistence2}, as $\eps_n \to \infty$, we  obtain a morphism 
$s: \Spec A^{\eps} \to Z_k$ that restricts to $s_n: \Spec K \to Z_k$ for each $n$. Observe that we have an induced morphsim $s: \cal{M}\left(A^{\eps}\right) \to Z_k^{\an}$ which is continuous. To see why, observe that we have a continuous morphism $(\Spec A^{\eps})^{\an} \to Z_k^{\an}$ coming from the analytification of $\Spec A^{\eps} \to Z_k$. Then we have a natural inclusion $\cal{M}\left(A^{\eps}\right) \xhookrightarrow{} \left(\Spec A^{\eps}\right)^{\an}$ which is continuous as the topology on both sides is the Gelfand topology. Then $s$ is the composition of the two. 
\par 
Letting $\omega$ be a non-principal ultrafilter, we obtain a triple $\left(f_{\omega},g_{\omega},C_{\omega}\right)$ over the residue field $\cal{H}\left(\omega\right)$, given by our morphism $s$. Observe that $\cal{H}\left(\omega\right)$ is of residue characteristic $0$ since $\eps_n \to 0$. Furthermore, \[\Lyap\left(f_{\omega}\right) = \lim_{\omega} \eps_n \Lyap\left(f_n\right) \geq \delta > 0\] 
as the Lyapunov exponent is a continuous function on $Z_k^{\an}$ an thus on $\cal{M}\left(A^{\eps}\right)$ by \cite[Theorem 6.3.7]{YZ24}. Hence $f$ has bad reduction and in particular $\cal{H}\left(\omega\right)$ is not trivially valued. 
\par 
As $\left\langle f,g \right\rangle_{C}$ is continuous on $Z_k^{\an} \setminus Z_k^{\triv}$, it is also continuous on $\cal{M}\left(A^{\eps}\right)$ and so we obtain 
\[\left\langle f_{\omega}, g_{\omega} \right\rangle_{C_{\omega}} = \lim_{\omega} \eps_n \left\langle f_n , g_n \right\rangle_{C_n} = 0.\]
By Proposition \ref{Pairing1}, we must have $d_2 \pi_1^* \mu_{f_{\omega}} = d_1 \pi_2^* \mu_{g_{\omega}}$. Since $\Lyap\left(f_{\omega}\right) > 0$, by Proposition \ref{Rigidity1}, our curve $C_{\omega}$ must be preperiodic for $\left(f_{\omega},g_{\omega}\right)$. Thus $s_{\omega}$ must lie in $V_m$ for some $m \geq 1$. However, we also have 
\[\lambda_{V_m}\left(s_{\omega}\right) = \lim_{\omega} \eps_n \lambda_{V_m}\left(s_n\right) \leq c_m < +\infty,\]
which implies that $s_{\omega} \not \in V_m$. This is a contradiction and we are done. 
\end{proof}

We can now obtain independence from $K$ and remove the requirement that $\lambda_{\partial K}$ assuming the residue characteristic is large enough. 

\begin{proposition} \label{UniformEnergy2}
Let $\delta, c_1,c_2,\ldots$ be positive constants. Then there exists constants $N > 0$ and $c > 0$ depending on $\delta, c_1,c_2,\ldots$ such that for any algebraically closed complete non-trivially valued field $\left(K,\left| \cdot \right|\right)$ with residue characteristic $\geq N$, if $s = \left(f,g,C\right) \in Z_k\left(K\right)$ satisfies:
\begin{enumerate}
\item ${\blue \lambda_{V_i}}(s) \leq c_i \lambda_{\partial Z_k}(s)$ for all $i \geq 1$;

\item $\Lyap\left(f\right) \geq \delta \lambda_{\partial Z_k}\left(s\right),$
\end{enumerate}
then we have the uniform lower bound 
\[\left\langle f,g \right\rangle_C \geq c \lambda_{\partial Z_k}\left(s\right).\]
\end{proposition}

\begin{proof}
Similar to the proof of Proposition \ref{UniformEnergy1}, we assume otherwise. Then we obtain a sequence of valued fields $\left(K_n,\left| \cdot \right|\right)$ with residue characteristic of $K_n$ being $\geq n$, along with a sequence $s_n = \left(f_n,g_n,C_n\right) \in Z_k\left(K_n\right)$ so that:

\[{\blue \lambda_{V_i}}(s_n) \leq c_i \lambda_{\partial Z_k}(s_n), \Lyap(f_n) \geq \delta \lambda_{\partial Z_k}(s), \text{ and } \langle f_n, g_n \rangle_{C_n} < \frac{1}{n} \lambda_{\partial Z_k}(s_n).\]
We may also assume that $\lambda_{\partial Z_k}\left(s_n\right) > 0$ as the energy pairing is non-negative. We degenerate with the scaling $\eps_n = \lambda_{\partial Z_k}\left(s_n\right)^{-1}$. Since the residue characteristic of $K_n$ is either $0$ or goes to $\infty$, by Proposition \ref{DegenerationExistence2} we obtain a morphism $s: \Spec A^{\eps} \to Z_k$ restricting to $s_n: \Spec K_n \to Z_k$ for each $n$ and thus a morphism $s: \cal{M}\left(A^{\eps}\right) \to Z_k^{\an}$. 
\par 
Picking $\omega$ to be a non-principal ultrafilter, the residue field $\cal{H}\left(\omega\right)$ is of residue characteristic $0$ and we obtain a triple $\left(f_{\omega},g_{\omega},C_{\omega}\right)$ over $\cal{H}\left(\omega\right)$ using $s$. We then have 
\[\left\langle f_{\omega}, g_{\omega} \right\rangle_{C_{\omega}} = \lim_{\omega} \eps_n \left\langle f_n , g_n \right\rangle_{C_n} = 0.\]
Since $\Lyap\left(f_{\omega}\right) \geq \delta > 0$, it must be that $d_2 \pi_1^* \mu_{f_{\omega}} = d_1 \pi_2^* \mu_{g_{\omega}}$. Hence $C_{\omega}$ must be preperiodic and so $s_{\omega} = \left(f_{\omega},g_{\omega},C_{\omega}\right)$ must lie on $V_m$ for some $m \geq 1$, But this can't happen as $\lambda_{V_{m}}\left(s_{\omega}\right) \leq c_m < +\infty$ which gives us a contradiction.  
\end{proof}

We next handle the setting where $f$ is a Latt{\`e}s map and $g$ is a map with good reduction. We note that the locus of triplets $(f,g,C) \in Z_k$ with $f$ being Latt{\`e}s is Zariski closed since the flexible Latt{\`e}s locus is closed in $\rat_d$ and the rigid Latt{\`e}s maps correspond to a finite set of points. Hence their union in $\rat_d$ is closed and so is their preimage in $\Rat_d$. We fix a local height function $\lambda_{\partial \rat_d}$ for the boundary of $\rat_d$. For $f \in \Rat_d\left(K\right)$, we let $\left\langle f \right\rangle$ be the induced point on $\rat_d\left(K\right)$.

\begin{proposition} \label{UniformEnergy3}
Let $\left(K,\left| \cdot \right|\right)$ be a complete non-trivially valued field of characteristic $0$. Let $\delta > 0$ be a positive constant. Then there exists constants $N > 0$ and $c > 0$ depending on $K,\delta$ such that if $s = \left(f,g,C\right) \in Z_k\left(K\right)$ satisfies:
\begin{enumerate}
\item $f$ is a Latt{\`e}s map with $\lambda_{\rat_d}(\langle f \rangle) \geq \delta \lambda_{\partial Z_k}(s)$ and $\lambda_{\rat_d}(\langle f \rangle) \geq N \lambda_{\partial \Rat_d}(g)$;

\item $\lambda_{\partial Z_k}\left(s\right) \geq N$.
\end{enumerate}
Then we have the uniform lower bound 
\[\left\langle f,g \right\rangle_C \geq c \lambda_{\partial Z_k}\left(s\right).\]
\end{proposition}

\begin{proof}
The proof is similar to Proposition \ref{UniformEnergy1}. We assume otherwise and obtain a sequence of triplets $(f_n,g_n,C_n) \in Z_k(K)$ where $f_n$ is a Latt{\`e}s map satisfying 
\[\lambda_{\partial \rat_d}(\langle f_n \rangle) \geq \delta \lambda_{\partial Z_k}(s_n), \lambda_{\partial \rat_d}(\langle f_n \rangle) \geq n \lambda_{\partial \Rat_d}(g_n) \text{ and } \lambda_{\partial Z_k}(s_n) \geq n.\]
Degenerating with $\eps_n = \lambda_{\partial Z_k}\left(s_n\right)^{-1}$ gives us $\left(f_{\omega},g_{\omega},C_{\omega}\right)$ such that $f_{\omega}$ is a Latt{\`e}s map with $\lambda_{\partial \rat_d}(\langle f_{\omega} \rangle) > 0$ and $\lambda_{\partial \Rat_d}(g_{\omega}) = 0$. Here, $f_{\omega}$ is a Latt{\`e}s map because the locus of Latt{\`e}s maps is closed. {\blue This means that $f_{\omega}$ does not have potentially good reduction and $g_{\omega}$ has good reduction.} By Proposition \ref{Rigidity2}, we cannot have $d_2 \pi_1^* \mu_{f_{\omega}} = d_1 \pi_2^* {\blue \mu_{g_{\omega}}}$ which implies that $\langle f_{\omega}, g_{\omega} \rangle_{C_{\omega}} > 0$. However this contradicts
\[\langle f_{\omega}, g_{\omega} \rangle_{C_{\omega}} = \lim_{\omega} \eps_n \langle f_n, g_n, \rangle_{C_n} = 0\]
and we are done.
\end{proof}

Similarly we have the analogue of Proposition \ref{UniformEnergy2}. 

\begin{proposition} \label{UniformEnergy4}
Let $\delta > 0$ be a positive constant. Then there exists constants $N > 0$ and $c > 0$ depending on $\delta$ such that for any algebraically closed complete non-trivially valued field $\left(K,\left| \cdot \right|\right)$ with residue characteristic $\geq N$, if $s = \left(f,g,C\right) \in Z_k\left(K\right)$ satisfies  \[\lambda_{\rat_d}(\langle f \rangle) \geq \delta \lambda_{\partial Z_k}(s) \text{ and } \lambda_{\partial\rat_d}(\langle f \rangle) \geq N \lambda_{\partial\Rat_d}(g),\]
then we have the uniform lower bound 
\[\left\langle f,g \right\rangle_C \geq c \lambda_{\partial Z_k}\left(s\right).\]
\end{proposition}

\begin{proof}
This is similar to Propositions \ref{UniformEnergy2} and \ref{UniformEnergy3}. We omit the proof.     
\end{proof}

\section{Uniform Dynamical Bogomolov for Curves} \label{sec: Curves}
{\blue The main result of this section is Theorem~\ref{HeightGapCurve1}, a uniform dynamical height gap for non-weakly special curves on $\bb{P}^1\times\bb{P}^1$.} As $\Rat_d, \rat_d, C_{d_1,d_2}$ are affine, we fix height functions $h_{\Rat_d}, h_{\rat_d}$ and $h_{C_{d_1,d_2}}$ coming from a height function of a choice of compactification so that their boundary is ample and hence they correspond to ample height functions. We similarly let $\lambda_{\Rat_d}, \lambda_{\rat_d}, \lambda_{C_{d_1,d_2}}$ be their local height functions.

{\blue \begin{definition}\label{def:WeaklySpecialCurves}
Let $f,g$ be rational maps of degree $d\geq2$ and let $C\subseteq\left(\bb{P}^1\right)^2$ be an irreducible curve over an algebraically closed field of characteristic $0$. We say that $C$ is weakly $\left(f,g\right)$-special if either:

\begin{enumerate}
\item $C$ does not project dominantly to one of the coordinates;
\item $C$ is $\left(f,g\right)$-preperiodic;
\item Both $f$ and $g$ are exceptional, with associated algebraic groups $G_f,G_g$, isogenies $\psi_f,\psi_g$ and dominant maps $\vphi_f,\vphi_g$ to $\bb{P}^1$ satisfying
\[f\circ\vphi_f=\vphi_f\circ\psi_f,\qquad g\circ\vphi_g=\vphi_g\circ\psi_g,\]
and there exists a coset $H\subseteq G_f\times G_g$ such that
\[C\cap\operatorname{Im}\left(\vphi_f,\vphi_g\right)=\left(\vphi_f,\vphi_g\right)\left(H\right).\]
Here each of $G_f,G_g$ is either $\bb{G}_m$ or an elliptic curve, and a coset is a translate of a connected algebraic subgroup.
\end{enumerate}
\end{definition}}

\begin{theorem} \label{HeightGapCurve1}
Let $f,g$ be two rational maps of degree $d \ge 2$ defined over $\ovl{\bb{Q}}$. Let $C \subseteq \left(\bb{P}^1\right)^2$ be a degree $\left(d_1,d_2\right)$ curve which is not weakly $\left(f,g\right)$-special. Then there exists $c = c\left(d_1,d_2,d\right)$ such that 
\[\h_{\left(f,g\right)}\left(C\right) \geq c \max\left\{h_{\rat_d}\left(f\right), h_{\rat_d}\left(g\right),1\right\}.\]
\end{theorem}

Before we begin proving Theorem \ref{HeightGapCurve1}, we need to first compare the sum of Lyapunov exponents of $f$ at every place $v$ to its height $h_{\rat_d}$. For $f$ a rational map of degree $d \ge 2$ over a number field $K$, we define the Lyapunov height 
\[h_{\Lyap}\left(f\right) = \sum_{v \in M_K} n_v \Lyap\left(f_v\right)\]
We will show that if $f$ is not a Latt{\`e}s map, then this is comparable to $h_{\rat_d}\left(f\right)$.

\begin{theorem} \label{GlobalLyapunov1}
Let $f$ be a rational map of degree $d \ge 2$ over $\ovl{\bb{Q}}$ that is not a Latt{\`e}s map. Then there exists constants $c_1, c_2 > 0$, depending only on $d$, such that 
\[ h_{\Lyap}\left(f\right) \geq c_1 h_{\rat_d}\left(\left\langle f \right\rangle\right) - c_2.\]
\end{theorem}

\begin{proof}
Let 
\[h_{\crit}\left(f\right) = \sum_{i=1}^{2d-2} \h_{f}\left(c_i\right)\]
be the critical height of $f$, where $c_1,\ldots,c_{2d-2}$ are the critical points of $f$ counted with multiplicity. By Ingram\cite{Ing18}, using McMullen's theorem on multipliers of rational maps\cite{McM87}, we know that there are constants $A,B > 0$ depending only on $d$ such that
\[h_{\crit}\left(f\right) \geq A h_{\rat_d}\left(\left\langle f \right\rangle\right) - B\]
for $f$ outside of the Latt{\`e}s locus. We now apply \cite[Proposition 2.7]{GOV20} which states that after choosing a homogeneous lift $F$ of $f$, for any place $v$, we have 
\[\Lyap\left(f_v\right) \geq - \log \left|d\right|_v + \sum_{c \in \mathrm{Crit}\left(f\right)} g_{F,v}\left(c\right) - \frac{2}{d} \log \left|\Res F\right|_v + \log \left|C_F\right|_v\]
where $C_F$ is some constant associated to $F$ and $g_F$ is the local homogeneous escape rate function. Summing up over all places and using the product formula, we obtain 
\[h_{\Lyap}\left(f\right) \geq h_{\crit}\left(f\right) \geq A h_{\rat_d}\left(\left\langle f \right\rangle\right) - B\]
for $f$ outside the Latt{\`e}s locus as desired. 
\end{proof}

We begin the proof of Theorem \ref{HeightGapCurve1} by making some reductions. First, we may assume there is a constant $A = A\left(d\right) > 0$ depending only on $d$ so that 
\begin{equation} \label{eq: Assumption1}
h_{\Rat_d}\left(f\right) \leq A \max\left\{1,h_{\rat_d}\left(f\right)\right\}, h_{\Rat_d}\left(g\right) \leq c_1 \max\left\{1, h_{\rat_d}\left(g\right)\right\}.
\end{equation}
Indeed for any $\vphi_1,\vphi_2 \in \GL_2\left(\ovl{\bb{Q}}\right)$, we may replace $\left(f,g,C\right)$ with $\left(\vphi_1  f \vphi_1^{-1} , \vphi_2 g   \vphi_2^{-1},  \left(\vphi_1, \vphi_2\right)\left(C\right)\right)$ since $\h_{\vphi_1  f \vphi_1^{-1}}\left(\vphi\left(x\right)\right) = \h_{f}\left(x\right)$ and similarly for $g$. By \cite[Lemma 6.3.2]{Sil12}, we know there exists $A,B > 0$, depending only on $d$, so that 
\[\min_{f' \in \left\langle f \right\rangle} h_{\Rat_d}\left(f'\right) \leq A h_{\rat_d}\left(f\right) + B,\]
which we may clearly assume that $A \geq B$. By symmetry, we may assume that $h_{\rat_d}\left(f\right) \geq h_{\rat_d}\left(g\right)$. Next, we may assume that 
\begin{equation} \label{eq: Assumption2}
h_{C_{d_1,d_2}}\left(C\right) \leq M_1 \max\left\{h_{\rat_d}\left(f\right),1\right\}
\end{equation}
for some $M_1 = M_1\left(d_1,d_2,d\right) > 0$ depending only on $\left(d_1,d_2,d\right)$. This is because by Section \ref{sec: CanonicalHeights}, we can find $M_1', M_2' > 0$ depending only on $\left(d_1,d_2,d\right)$
\begin{equation} \label{eq: CurveDifference}
M'_1 \h_{\left(f,g\right)}\left(C\right) + M'_2 \max\left\{h_{\Rat_d}\left(f\right), h_{\Rat_d}\left(g\right),1\right\} \geq  h_{C_{d_1,d_2}}\left(C\right).
\end{equation}
Hence if $h_{C_{d_1,d_2}}\left(C\right)$ is too large, then Theorem \ref{HeightGapCurve1} is automatically true. Thus we may assume \eqref{eq: Assumption2} holds. 
\par 
We now prove Theorem \ref{HeightGapCurve1} when $h_{\Rat_d}\left(f\right)$ is sufficiently large. For flexibility, we will assume the weaker inequality $M_2 h_{\rat_d}\left(f\right) \geq h_{\rat_d}\left(g\right)$ for some fixed constant $M_2 > 0$. 

\begin{proposition} \label{HeightGapCurve2}
Let $f,g$ be two rational maps of degree $d \ge 2$ defined over $\ovl{\bb{Q}}$. Let $C \subseteq \left(\bb{P}^1\right)^2$ be a degree $\left(d_1,d_2\right)$ curve which is not weakly $\left(f,g\right)$-special. Assume that $f$ is not a Latt{\`e}s map and there exists constants $M_1,M_2 > 1$ such that  
\[h_{C_{d_1,d_2}}\left(C\right) \leq M_1 \max\left\{h_{\rat_d}\left(f\right),1\right\} \text{ and } M_2 h_{\rat_d}\left(\left\langle f \right\rangle\right) \geq \max\left\{1, h_{\Rat_d}\left(f\right), h_{\Rat_d}\left(g\right)\right\}.\]
Then there exists $c = c\left(d_1,d_2,d,M_1,M_2\right) > 0$ and $N = N\left(d_1,d_2,d,M_1,M_2\right) > 0$ so that if $h_{\Rat_d}\left(f\right) \geq N$, then 
\[\h_{\left(f,g\right)}\left(C\right) \geq c \max\left\{h_{\rat_d}\left(f\right), h_{\rat_d}\left(g\right),1\right\}.\]
\end{proposition}

\begin{proof}
Let $s = \left(f,g,C\right)$ be a triple satisfying the hypothesis of Proposition \ref{HeightGapCurve2} and let $K$ be a number field that they are defined over. We may assume that $h_{\rat_d}\left(f\right) \geq 1$. Recall from Section \ref{sec: EnergyPairing} that we have a covering of $X = \Rat_d \times \Rat_d \times C_{d_1,d_2}$ by locally closed subvarieties $Z_1,\ldots,Z_m$. Let $\left(f,g,C\right) \in Z_k\left(\ovl{\bb{Q}}\right)$. Recall that we have a countable collection of subvarieties $V_1,V_2,\ldots$ of $Z_k$ so that $\left(f,g,C\right) \in V_m$ for some $m$ if and only if $C$ is preperiodic for $\left(f,g\right)$. We fix local height functions $\lambda_{V_i}$ for each $V_i$ along with a local height function $\lambda_{\partial Z_k}$ for the boundary of $Z_k$. Observe that the boundary $\partial X$ is ample 
\par 
Now since $X = \Rat_d \times \Rat_d \times C_{d_1,d_2}$ is affine, there exists a compactification $\ovl{X}$ for which $\partial X$ is ample. We may thus choose $\lambda_{\partial Z_k}$ so that $\lambda_{\partial Z_k}\left(s\right) \geq \lambda_{\Rat_d}\left(f\right)$ and thus $h_{\partial Z_k}\left(s\right) \geq h_{\Rat_d}\left(f\right)$. Also note that 
\[h_{\Rat_d}\left(f\right) + h_{\Rat_d}\left(g\right) + h_{C_{d_1,d_2}}\left(C\right)\]
is an ample height for $X$. By our hypothesis, this quantity is dominated by some multiple of  $\max\left\{h_{\rat_d}\left(f\right),1\right\} = h_{\rat_d}\left(f\right)$ and thus both $h_{\rat_d}\left(f\right)$ and $h_{\partial Z_k}\left(s\right)$ dominates some ample height. Since $h_{\rat_d}\left(f\right) \geq 1$, we hence obtain positive constants $a_1,a_2\ldots$, depending on $\left(d_1,d_2,d,M_1,M_2\right)$ so that 
\[h_{\partial Z_k}\left(s\right) \geq h_{\Rat_d}\left(f\right) \geq a_i h_{V_i}\left(s\right)\]
holds for all $i$ if $s \in \left(Z_k \setminus V_i\right)\left(\ovl{\bb{Q}}\right)$. By Theorem \ref{GlobalLyapunov1}. we also get a constant $a_0 > 0$ depending on $\left(d_1,d_2,d,M_1,M_2\right)$ so that 
\begin{equation} \label{eq: UniCurve1}
h_{\Lyap}\left(f\right) \geq  a_0 h_{\partial Z_k}\left(s\right).
\end{equation}
Note that we also have constants $b_1, b_2 > 0$ so that 
\begin{equation} \label{eq: UniCurve2}
\lambda_{\partial Z_k,v}\left(s\right) \geq \lambda_{\Rat_d,v}\left(f\right) \geq b_1 \Lyap\left(f_v\right) - b_2
\end{equation}
where we may take $b_2 = 0$ for all sufficiently large residue characteristics depending only on $d$. We now wish to apply Propositions \ref{UniformEnergy1} and \ref{UniformEnergy2}. We choose constants 
\[\delta = \frac{a_0}{100}, c_i = 10^i ({\blue a_0b_1}a_i)^{-1}.\]
Let $N'$ and $c'$ be common constants for $N$ and $c$ in both Propositions \ref{UniformEnergy1} and \ref{UniformEnergy2}. We may take $N'$ large enough so that $b_2 = 0$ for all residue characteristic $\geq N'$. We split into two cases. Let $M'_K$ denote the places of $K$ which have residue characteristic $\geq N'$. 
\par 
\textbf{Case 1:} We have 
\[\sum_{v \in M'_K} n_v \Lyap\left(f_v\right) \geq \frac{1}{2} h_{\Lyap}\left(f\right).\]
Let $S$ be the subset of places of $v \in M'_K$ such that 
\[c_i \lambda_{\partial Z_k,v}\left(s\right) \geq \lambda_{V_i,v}\left(s\right) \text{ for all } i \geq 1 \text{ and } \Lyap\left(f_v\right) \geq \delta \lambda_{\partial Z_k,v}\left(s\right).\]

\begin{lemma} \label{UniCurveLemma1}
We have 
\[\sum_{v \in S} n_v \lambda_{\partial Z_k,v}(s) \geq \frac{b_1}{4} {\blue h_{\Lyap}(f)}.\]
\end{lemma} 

\begin{proof}
Let $T_0$ be the set of places of $M'_K$ so that $\Lyap(f_v) < \delta \lambda_{\partial Z_k,v}(s)$ and let $T_i$ be the set of places in  $M'_K$ so that $c_i \lambda_{\partial Z_k,v}(s) < \lambda_{V_i,v}(s)$ {\blue but $v\notin T_0\cup\cdots\cup T_{i-1}$.} Then $S$ is the complement of $\cup_{i \geq 0} T_i$ in $M'_K$. For $i \geq 1$, we have 
\[{\blue h_{\partial Z_k}}(s) \geq a_i \sum_{v \in T_i} n_v \lambda_{V_i,v}(s) \geq a_i c_i \sum_{v \in T_i} n_v \lambda_{\partial Z_k,v}(s)\]
and hence 
\begin{equation} \label{eq: UniCurve3}
\sum_{v \in T_i} n_v \lambda_{\partial Z_k,v}(s) \leq \frac{1}{a_i c_i} {\blue h_{\partial Z_k}}(s) \leq \frac{{\blue a_0b_1}}{10^i} {\blue h_{\partial Z_k}}(s).
\end{equation}
For $T_0$, we have 
\[\sum_{v \in T_0} n_v \Lyap(f_v) \mathrel{\blue \leq} \delta \sum_{v \in T_0} n_v \lambda_{\partial Z_k,v}(s) \leq \frac{a_0}{100} {\blue h_{\partial Z_k}}(s).\]
By \eqref{eq: UniCurve1}, we have  
\[\sum_{v \in M_K} n_v \Lyap(f_v) \geq a_0 {\blue h_{\partial Z_k}}(s)\]
and so 
\begin{equation} \label{eq: UniCurve4}\sum_{v \in T_0} n_v \Lyap(f_v) \leq \frac{1}{100} \sum_{v \in M_K} n_v \Lyap(f_v) = \frac{1}{100} h_{\Lyap}(f).
\end{equation}
Summing up \eqref{eq: UniCurve3} and \eqref{eq: UniCurve4} {\blue and using \eqref{eq: UniCurve2} on the disjoint $T_i$ for $i\geq1$}, we get 
\[\sum_{v \in M'_K \setminus S} n_v \Lyap(f_v) \leq \left(\frac{1}{100} + \sum_{i=1}^{\infty} \frac{1}{10^i} \right) h_{\Lyap}(f) \leq \frac{1}{5} h_{\Lyap}(f).\]
By assumption, we must have 
\[\sum_{v \in S} n_v \Lyap(f_v) \geq \frac{1}{4} h_{\Lyap}(f).\]
We now apply \eqref{eq: UniCurve2} to get 
\[\sum_{v \in S} n_v \lambda_{\partial Z_k{\blue ,v}}(s) \geq b_1 \sum_{v \in S} n_v \Lyap(f_v) \geq \frac{b_1}{4} h_{\Lyap}(f)\]
as desired.
\end{proof}

We now apply Proposition {\blue \ref{UniformEnergy2}}. We hence obtain 
\[\sum_{v \in S} n_v \langle f , g \rangle_{C,v} \geq {\blue c'} \sum_{v \in S} n_v \lambda_{\partial Z_k,v}(s) \geq \frac{b_1 c'}{4} h_{\Lyap}(f).\]
By {\blue Proposition \ref{GlobalPairing1} and }Theorem \ref{GlobalLyapunov1}, this proves Proposition \ref{HeightGapCurve2} in this case. 
\par 
\textbf{Case 2:} Let $T = M_K \setminus M_{K}'$. Then we have 
\[\sum_{v \in T} n_v \Lyap(f_v) \geq \frac{1}{2} h_{\Lyap}(f).\]
Note that $\sum_{v \in T} n_v$ is a finite number bounded only in terms of $\left(d_1,d_2,d,M_1,M_2\right)$ as $T$ consists of {\blue archimedean places and }places with residue characteristic $\mathrel{\blue <} N'$, where $N'$ itself depends only on $\left(d_1,d_2,d,M_1,M_2\right)$. Similar to Lemma \ref{UniCurveLemma1}, if we let $S$ be the subset of places of $T$ so that 
\[c_i \lambda_{\partial Z_k,v}(s) \geq \lambda_{V_i,v}(s) \text{ for all } i \geq 1 \text{ and } \Lyap(f_v) \geq \delta \lambda_{\partial Z_k,v}(s),\]
we again obtain 
\[\sum_{v \in S} n_v \Lyap(f_v) \geq \frac{1}{4} h_{\Lyap}(f).\]

We now apply \eqref{eq: UniCurve2} to obtain that 
\[\sum_{v \in {\blue S}} n_v \lambda_{\partial Z_k,v}(s) \geq \sum_{v \in {\blue S}} n_v (b_1 \Lyap(f_v) - b_2) \geq \frac{b_1}{4} h_{\Lyap}(f) - B_2\]
where $B_2 = b_2 \sum_{v \in T} n_v$ is some constant depending only on ${\blue \left(d_1,d_2,d,M_1,M_2\right)}$. We now apply Proposition \ref{UniformEnergy1} {\blue (or Proposition \ref{UniformEnergy2} when $p\geq P$)}, where we remove the condition $\lambda_{\partial Z_k,v}\left(s\right) \geq N'$ {\blue using nonnegativity of the pairing }and instead have the conclusion $\langle f , g \rangle_{C{\blue ,v}} \geq {\blue c'}(\lambda_{\partial Z_k,v}(s)-N')$. We thus get 
\[\sum_{v \in {\blue S}} n_v \langle f ,g \rangle_{C,v} \geq {\blue c'} \sum_{v \in {\blue S}} n_v (\lambda_{\partial Z_k,v}{\blue (s)} - N') \geq {\blue c'} \left(\frac{b_1}{4} h_{\Lyap}(f) - B_2 \right) - {\blue c'} N' \sum_{v \in T} n_v.\]
Again ${\blue c'} N' \sum_{v \in T} n_v$ is some constant which depends only on $\left(d_1,d_2,d,M_1,M_2\right)$. Hence{\blue , by Proposition \ref{GlobalPairing1} and Theorem \ref{GlobalLyapunov1},} if $h_{\Rat_d}\left(f\right)$ is large enough, we conclude Proposition \ref{HeightGapCurve2}.
\end{proof}

We now handle the case where $f$ is a Latt{\`e}s map and $h_{\rat_d}\left(f\right)$ is much larger than $h_{\rat_d}\left(g\right)$. 

\begin{proposition} \label{LattesGap1}
Let $f,g$ be two rational maps of degree $d \ge 2$ defined over $\ovl{\bb{Q}}$. Let $C \subseteq \left(\bb{P}^1\right)^2$ be a degree $\left(d_1,d_2\right)$ curve. Assume there is a constant $A > 0$, so that $f$ is a Latt{\`e}s map with 
\[h_{\rat_d}\left(\left\langle f \right\rangle\right) \geq A \max \left\{h_{\Rat_d}\left(f\right),1\right\}\]
and that there exists a constant $M > 1$ such that 
\[h_{C_{d_1,d_2}}\left(C\right) \leq M \max\left\{ h_{\rat_d}\left(f\right),1\right\}.\] 
Then there exists constants $c, N > 0$, depending only on $\left(d_1,d_2,d,M,A\right)$, so that if 
\[h_{\rat_d}\left(\left\langle f \right\rangle\right) \geq N \max\left\{h_{\Rat_d}\left(g\right),1\right\},\]
then 
\[\h_{f,g}\left(C\right) \geq c \max\left\{1, h_{\Rat_d}\left(f\right), h_{\Rat_d}\left(g\right)\right\}.\]
\end{proposition}

\begin{proof}
The proof is similar to Proposition \ref{HeightGapCurve2} but is much simpler in this case. We will use Propositions \ref{UniformEnergy3} and \ref{UniformEnergy4}. Let $s = \left(f,g,C\right)$ be such a triplet and assume that $s \in Z_k\left(\ovl{\bb{Q}}\right)$. Let $K$ be a number field such that $\left(f,g,C\right)$ is defined over. As in Proposition \ref{HeightGapCurve2}, after assuming $h_{\rat_d}\left(f\right) \geq 1$, we may obtain constants $a,b > 0$ depending on $\left(d_1,d_2,d,M,A\right)$ so that 
\[b h_{\rat_d}\left(f\right) \geq h_{\partial Z_k}\left(s\right) \geq a h_{\rat_d}\left(f\right).\]
Let $c',N'$ be constants for Propositions \ref{UniformEnergy3} and \ref{UniformEnergy4} for $\delta = \frac{1}{100b}$. We let $M_{K}'$ be the set of places with residue characteristic $\geq N'$. We again split into two cases. Note that we have constants $b_1, b_2 > 0$ so that 
\[\lambda_{\partial Z_k,v}\left(s\right) \geq b_1\lambda_{\rat_d,v}\left(\left\langle f \right\rangle\right) - b_2.\]
where again $b_2 = 0$ for all sufficiently large residue characteristics, which after increasing $N'$ we may assume it holds for all residue characteristics $\geq N'$. 
\par 
\textbf{Case 1:} We have 
\[\sum_{v \in M_K'} n_v \lambda_{\rat_d,v}\left(f\right) \geq \frac{1}{2} h_{\rat_d}\left(f\right).\]
Now assume that $h_{\rat_d}\left(\left\langle f \right\rangle\right) \geq 6N' h_{\Rat_d}\left(g\right)$. Then if $S$ is the subset of places of $M_K'$ so that 
\[\lambda_{\rat_d}(\langle f \rangle) \geq N' \lambda_{\Rat_d,v}(g) \text{ and } \lambda_{\rat_d}(\langle f \rangle) \geq \delta \lambda_{\rat_d}(f).\]
Then 
\[\sum_{v \in M_K' \setminus S} n_v \lambda_{\rat_d,v}\left(\left\langle f \right\rangle\right) \leq N' \sum_{v \in M_K} n_v \lambda_{\Rat_d,v}\left(g\right) + \frac{1}{100 b} \sum_{v \in M_K} n_v \lambda_{\partial Z_k}\left(s\right) \]
\[\leq N' h_{\Rat_d}\left(g\right) + \frac{1}{100b} h_{\partial Z_k}\left(s\right) \leq \frac{1}{3} h_{\rat_d}\left(\left\langle f \right\rangle\right).\]
Hence we must have 
\[\sum_{v \in S} n_v \lambda_{\rat_d,v}\left(f\right) \geq \frac{1}{6} h_{\rat_d}\left(\left\langle f \right\rangle\right).\]
We now apply Proposition \ref{UniformEnergy4} to obtain 
\[\sum_{v \in S} n_v \left\langle f , g \right\rangle_{C,v} \geq c' \sum_{v \in S} n_v \lambda_{\partial Z_k,v}\left(s\right) \geq \frac{c' b_1}{6} h_{\rat_d}\left(\left\langle f \right\rangle\right)\]
which proves what we want. 
\par 
\textbf{Case 2:} Let $T = M_{K'} \setminus M_K$. We have 
\[\sum_{v \in T} n_v \lambda_{\rat_d,v}\left(f\right) \geq \frac{1}{2} h_{\rat_d}\left(f\right).\]
Again let $S$ be the subset of places of $M_K'$ so that 
\[\lambda_{\rat_d}(\langle f \rangle) \geq N \lambda_{\Rat_d,v}(g) \text{ and } \lambda_{\rat_d}(\langle f \rangle) \geq \delta \lambda_{\partial Z_k}(s).\]
Then similarly if $h_{\rat_d}\left(\left\langle f \right\rangle\right) \geq 6N' h_{\Rat_d}\left(g\right)$, we have 
\[\sum_{v \in S} n_v \lambda_{\rat_d,v} \left(\left\langle f \right\rangle\right) \geq \frac{1}{6} h_{\rat_d}\left(\left\langle f \right\rangle\right).\]
Hence we have 
\[\sum_{v \in S} n_v \lambda_{\partial Z_k,v}\left(s\right) \geq \frac{b_1}{6} h_{\rat_d} \left(\left\langle f \right\rangle\right) - b_2 \sum_{v \in T} n_v = \frac{b_1}{6} h_{\rat_d} \left(\left\langle f \right\rangle\right) - B_2\]
where $B_2 = b_2 \sum_{v \in T} n_v$ depends only on $\left(d,d_1,d_2,M,A\right)$. We then apply Proposition \ref{UniformEnergy3}, removing the condition $\lambda_{\partial Z_k,v}\left(s\right) \geq N'$ and replacing the lower bound $c \lambda_{\partial Z_k,v}\left(s\right)$ with $c \left(\lambda_{\partial Z_k,v}\left(s\right) - N'\right)$, we obtain 
\[\sum_{v \in S} n_v \left\langle f , g \right\rangle_{C,v} \geq c \left(\frac{b_1}{6} h_{\rat_d}\left(\left\langle f \right\rangle\right) - B_2 - N' \sum_{v \in T} n_v \right).\]
Again $N' \sum_{v \in T} n_v$ is some constant depending only on $\left(d,d_1,d_2,M,A\right)$ and so if $h_{\rat_d}\left(\left\langle f \right\rangle\right)$ is large enough, Proposition \ref{LattesGap1} follows as desired.
\end{proof}

Combining Propositions \ref{HeightGapCurve2} and \ref{LattesGap1}, we obtain Theorem \ref{HeightGapCurve1} for $\left(f,g,C\right)$ with the quantity $\max\left\{h_{\Rat_d}\left(f\right), h_{\Rat_d}\left(g\right), h_{C_{d_1,d_2}}\left(C\right)\right\}$ being sufficiently large.

\begin{proposition} \label{HeightGapCurve3}
Let $\left(f,g\right)$ be two rational maps of degree $d \ge 2$ defined over $\ovl{\bb{Q}}$. Let $C \subseteq \left(\bb{P}^1\right)^2$ be a degree $\left(d_1,d_2\right)$-curve defined over $\ovl{\bb{Q}}$ {\blue which is not weakly $\left(f,g\right)$-special.} Assume that there is a $A > 0$ so that
\[A h_{\rat_d}\left(\left\langle f \right\rangle\right) \geq  \max\left\{h_{\Rat_d}\left(f\right),1\right\}\]
and similarly for $g$. Then there exists $c, N > 0$ depending only on $\left(d,d_1,d_2,A\right)$ so that if 
\[\max\left\{ h_{\Rat_d}\left(f\right), h_{\Rat_d}\left(g\right), h_{C_{d_1,d_2}}\left(C\right) \right\} \geq N.\]
then
\[\h_{\left(f,g\right)}\left(C\right) \geq c \max\left\{h_{\Rat_d}\left(f\right), h_{\Rat_d}\left(g\right)\right\}.\]
\end{proposition}

\begin{proof}
If both $f$ and $g$ are Latt{\`e}s, it follows directly from the case of uniform Bogomolov for curves in $E_1 \times E_2$ by using either Kuhne's relative Bogomolov for fibered powers of elliptic curves \cite{Kuh23} or Gao--Ge--Kuhne's height gap principle \cite{GGK21}. Hence we may assume that not both $f$ and $g$ are Latt{\`e}s. 
\par 
\begingroup\emergencystretch=3em
By symmetry, we may assume that $h_{\Rat_d}\left(f\right) \geq h_{\Rat_d}\left(g\right)$. By \eqref{eq: Assumption2}, we may assume there exists $M' = M'\left(d,d_1,d_2,d,A\right) > 0$ so that 
\[h_{C_{d_1,d_2}}\left(C\right) \leq M' \max\left\{h_{\rat_d}\left(f\right),1\right\}\]
Let $N_1$ be a constant for $N$ in Proposition \ref{LattesGap1}. Then if $f$ is Latt{\`e}s and $h_{\rat_d}(\langle f \rangle ) \geq N_1 \max\{h_{\rat_d}(g),1\}$, we are done. Hence we may assume either $f$ is non-Latt{\`e}s, or $f$ is Latt{\`e}s but 
\begin{equation} \label{eq: HeightCompare1}
N_1 A \max\left\{h_{\rat_d}\left(g\right),1\right\} \geq A h_{\rat_d}\left(\left\langle f \right\rangle\right)\geq h_{\Rat_d}\left(f\right).
\end{equation}
If $f$ is non-Latt{\`e}s, then the triplet $\left(f,g,C\right)$ satisfies the conditions of Proposition \ref{HeightGapCurve2} with $M_1 = M'$ and $M_2 =  A$. Then if $h_{\Rat_d}\left(f\right)$ is sufficiently large, we are done. In the other case, we have $N_1 \max\left\{h_{\rat_d}\left( \left\langle g \right\rangle\right),1\right\} \geq h_{\rat_d}\left(\left\langle f \right\rangle\right).$ and $g$ is non-Latt{\`e}s. Then the triplet $\left(g,f,C\right)$ satisfies Proposition \ref{HeightGapCurve2} with $M_1 = M' N_1 A, M_2 = N_1 A$.Hence if $h_{\rat_d}\left(g\right)$ is sufficiently large, we are done. By \eqref{eq: HeightCompare1}, this is implied by $h_{\rat_d}\left(f\right)$ being sufficiently large, concluding the proof of our Proposition. 
\par\endgroup
\par    
\end{proof}

We now handle the case where $\left(f,g,C\right)$ are of bounded height. In this case, we will use the archimiedean analogue of our rigidity result, which was proven by Mavraki--Schmidt \cite[Theorem 1.10]{MS25}. 

\begin{theorem} \label{ArchimedeanRigidity} 
Let $f,g \in \bb{C}\left(z\right)$ be two rational maps of degree $d \ge 2$. Let $C \subseteq \left(\bb{P}^1\right)^2$ be an irreducible complex algebraic curve which projects dominantly to both coordinates. {\blue If $\left(\deg \pi_2\right) \pi_1^* \mu_{f} = \left(\deg \pi_1\right) \pi_2^* \mu_g$, then $C$ is weakly $\left(f,g\right)$-special.} 
\end{theorem}

For our subvariety $Z_k \subseteq \Rat_d \times \Rat_d \times C_{d_1,d_2}$, we let $W_k$ be the closed subvariety cutting out triplets $\left(f,g,C\right)$ such that either $f$ and $g$ are both Latt{\`e}s maps, or where each $f$ and $g$ are conjugate to either a power or Chebyshev map. If $s = \left(f,g,C\right) \in Z_k \setminus W_k$, then $C$ is weakly $\left(f,g\right)$-special if and only if $C$ is preperiodic for $\left(f,g\right)$. 

\begin{proposition} \label{BoundedHeight1}
Let $\left(f,g\right)$ be two rational maps of degree $d \ge 2$ defined over $\ovl{\bb{Q}}$ and $C \subseteq \left(\bb{P}^1\right)^2$ an irreducible degree $\left(d_1,d_2\right)$-curve. Let $N > 0$ be a constant such that 
\[\max\left\{h_{\Rat_d}\left(f\right), h_{\Rat_d}\left(g\right), h_{C_{d_1,d_2}}\left(C\right) \right\} \leq N.\]
Assume that the pair $\left(f,g\right)$ are not both Latt{\`e}s maps, {\blue and} that {\blue at least one of them is} not conjugate to either a monomial or a Chebyshev map. Then there exists $\delta = \delta\left(d,d_1,d_2,N\right) > 0$ such that if $C$ is not weakly $\left(f,g\right)$-special, then
\[\h_{\left(f,g\right)}\left(C\right) \geq \delta.\]
\end{proposition}

\begin{proof}
Let $s = \left(f,g,C\right)$ lie on $Z_k\left(\ovl{\bb{Q}}\right)$ satisfying the hypothesis of our Proposition. Then by assumption, $s \in \left(Z_k \setminus W_k\right)\left(\ovl{\bb{Q}}\right)$ and $s \in \left(Z_k \setminus V_i\right)\left(\ovl{\bb{Q}}\right)$ for all $i \geq 1$. Let 
\[h_X\left(s\right) = h_{\Rat_d}\left(f\right) + h_{\Rat_d}\left(g\right) + h_{C_{d_1,d_2}}\left(C\right) + 1\]
so that $h_X$ is an ample height on $X = \Rat_d \times \Rat_d \times C_{d_1,d_2}$ and $h_X\left(s\right) \leq 3N + 1 \leq 4N$ as we may assume $N \geq 1$. Then there are positive constants $a_1,a_2,\ldots$ so that 
\[h_{V_i}\left(s\right) \leq a_i h_X\left(s\right)\]
for all $s \in \left(Z_k \setminus V_i\right)\left(\ovl{\bb{Q}}\right)$. We let $b_1,b_2$ be positive constants $h_{\partial Z_k}\left(s\right) \leq b_1 h_X\left(s\right) \text{ and } h_{W_k}\left(s\right) \leq b_2 h_X\left(s\right).$ Now on $Z_k\left(\bb{C}\right)$, consider the subset $D_k \subseteq Z_k\left(\bb{C}\right)$ consisting of points $s$ satisfying 
\[\lambda_{\partial Z_k}(s) \leq 300N {\blue b_1}, \lambda_{W_k}(s) \leq 300N {\blue b_2}, \lambda_{V_i}(s) \leq 10^{i} (4N) {\blue a_i}\]
where we use the local height functions over $\mathbb{C}$. Then $D_k$ is a closed subset of $\ovl{Z_k}\left(\bb{C}\right)$ for some compactification of $\ovl{Z_k}$ and hence a compact subset of $Z_k\left(\bb{C}\right)$ since it has empty intersection with the boundary. Furthermore, our function $\left\langle \, , \, \right\rangle$ is strictly positive on $D_k$ and hence there exists a $c > 0$ such that 
\[\left\langle f , g \right\rangle_C \geq c > 0\]
for all $s = \left(f,g,C\right) \in D_k$. Now let $K$ be a number field for which $\left(f,g,C\right)$ is defined over and let $T$ be the archimedean places of $K$. Let $S\subseteq T$ be the subset of places for which $s \in D_k$ when embedded via $v\in S$ into $Z_k\left(\bb{C}\right)$. We claim that 
\[\sum_{v \in S} n_v \geq \frac{1}{2} \sum_{v \in T} n_v = \frac{1}{2}.\]
To estimate this, for $i \geq 1$ let $T_i$ be the subset of $T$ so that 
\[{\blue \lambda_{V_i,v}}(s) \geq 10^i (4N) {\blue a_i}\]
and let $T_0$ be the subset so that $\lambda_{\partial Z_k,v}(s) \geq 300 N {\blue b_1}$
and $T_0'$ the subset so that ${\blue \lambda_{W_k,v}}(s) \geq 300N {\blue b_2}.$
For $i \geq 1$, we have 
\[10^i (4N) {\blue a_i} \sum_{v \in T_i} n_v \leq \sum_{v \in T_i} n_v {\blue \lambda_{V_i,v}}(s) \leq {\blue h_{V_i}}(s) \leq a_i h_X(s) \leq 4N a_i\]
and so
\[\sum_{v \in T_i} n_v \leq \frac{1}{10^i}.\]
Similarly, we obtain that 
\[\sum_{v \in T_0} n_v, \sum_{v \in T_0'} n_v \leq \frac{{\blue 4}}{300}.\]
As $S$ {\blue contains} the complement of $T_0,T_0'$ and $\cup_{i \geq 1} T_i$ in $T$, we thus have 
\[\sum_{v \in S} n_v \geq \sum_{v \in T} n_v - \frac{{\blue 8}}{300} -  \sum_{i=1}^{\infty} \frac{1}{10^i} \geq \frac{1}{2}\]
since $\sum_{v \in T} n_v = 1$. Now for each $v\in S$, we have 
\[\left\langle f ,g \right\rangle_{C,v} \geq c \]
and so 
\[\sum_{v \in S} n_v \left\langle f ,g \right\rangle_{C,v} \geq \frac{1}{2} c\]
which gives us our uniform $\delta{\blue {}=c/[2d_1d_2(d_1+d_2)]}$ as desired. 
\end{proof}

We now conclude the proof of Theorem \ref{HeightGapCurve1}.

\begin{proof}[Proof of Theorem \ref{HeightGapCurve1}]
As explained in our reductions at the start of the section, we may assume that $A \max\left\{ h_{\rat_d}\left(\left\langle f \right\rangle\right),1\right\} \geq h_{\Rat_d}\left(f\right)$ and similarly for $g$ for some $A > 0$ depending only on $d$. Then Proposition \ref{HeightGapCurve3} implies that we are done unless $h_{\Rat_d}\left(f\right), h_{\Rat_d}\left(g\right), h_{C_{d_1,d_2}}\left(C\right)$ have height bounded above by some $N = N\left(d,d_1,d_2\right) > 0$. If $\left(f,g\right)$ are not both Latt{\`e}s map, {\blue and at least one of them is} not conjugate to a monomial or Chebyshev map, then Proposition \ref{BoundedHeight1} implies that we are done in the bounded height case. Finally if both $\left(f,g\right)$ are Latt{\`e}s map, then we can conclude by uniform Bogomolov for products of elliptic curves by using Gao--Ge--Kuhne's \cite{GGK21} height gap principle, or by Kuhne's relative Bogomolov for fibered powers of elliptic curves \cite{Kuh23}. Here, one should note that $\max\{h_{\rat_d}(f),1\}$ is comparable to $\{h_{\Fal}(E),1\}$ if $E$ is the elliptic curve that the Latt{\`e}s map $f$ comes from. If both $\left(f,g\right)$ are conjugate to either the monomial or Chebyshev map, then we can instead conclude by the uniform Bogomolov conjecture for $\bb{G}_m^2$ due to Bombieri--Zannier\cite{BZ95}. 
\end{proof}

We can now deduce Theorem \ref{IntroUniformCurve1}. Applying Theorem \ref{QuantitativeZhang1} to Theorem \ref{HeightGapCurve1}, it suffices to handle the weakly $\left(f,g\right)$-special curves, in which case it is easy to see that Theorem \ref{IntroUniformCurve1} holds. We hence conclude Theorem \ref{IntroUniformCurve1}. 

\section{Higher Dimensional Subvarieties} \label{sec: HigherDim}
{\blue The main results of this section are the hypersurface height gap in Theorem~\ref{IntroTheorem2} and uniform dynamical Manin--Mumford over $\ovl{\bb{Q}}$ in Corollary~\ref{UniDMM1}.}
\par
We now consider the general case of subvarieties $V \subseteq \left(\bb{P}^1\right)^n$ where $\Phi = \left(f_1,\ldots,f_n\right)$ is a split endomorphism of $\left(\bb{P}^1\right)^n$ with $\deg f_i = d \geq 2$. {\blue For the height gap for hypersurfaces, we assume only that at least one $f_i$ is non-exceptional. For uniform dynamical Manin--Mumford for arbitrary subvarieties, we assume that at most one $f_i$ is exceptional.}
\par 
We say that a subvariety $V \subseteq \left(\bb{P}^1\right)^n$ is weakly $\Phi$-special if either of the following occurs:
\begin{enumerate} 
\item $V$ does not projects dominantly to all coordinates;

\item $V$ is preperiodic;

\item $V$ is of the form $\pi^{-1}\left(W\right)$ where $\pi: \left(\bb{P}^1\right)^n \to \left(\bb{P}^1\right)^{n-1}$ is a projection away from one of the coordinates. 
\end{enumerate}
We will prove the following height gap principle for hypersurfaces $H \subseteq \left(\bb{P}^1\right)^n$ that are not weakly $\Phi$-special. 

\begin{theorem} \label{HeightGapVarieties1}
Let $\Phi = \left(f_1,\ldots,f_n\right)$ be a split endomorphism of $\left(\bb{P}^1\right)^n$ over $\ovl{\bb{Q}}$ such that $\deg f_i = d \geq 2$ and each $f_i$ is non-exceptional. Let $H \subseteq \left(\bb{P}^1\right)^n$ be a hypersurface of degree $D$ that is not weakly $\Phi$-special. Then there exists $c = c\left(d,n,D\right) > 0$ and $E = E\left(d,n,D\right)$ such that 
\[\left\{x \in H\left(\ovl{\bb{Q}}\right) \mid \h_{\Phi}\left(x\right) \leq c \max\left\{h_{\rat_d}\left(\Phi\right),1\right\} \right\}\]
is contained in a proper Zariski closed subvariety $V$ of degree at most $E$.
\end{theorem}

Here, we have a natural very ample line bundle $L = \sum_{i=1}^{n} \pi_i^* O\left(1\right)$ on $\left(\bb{P}^1\right)^n$. Via the Segre embedding, we may embed it to $\bb{P}^N$ and we will compute degrees and Weil heights with respect to this embedding. The canonical height can be written as a limit of Weil heights of iterates
\[\h_{\Phi}(V) = \frac{1}{d^n} \frac{h(\Phi^n(V))}{\deg(\Phi^n(V))}.\]
Our approach is inspired by the work of David--Philippon \cite{DP99} in the case of $\bb{G}_m^n$ \cite{DP99}, which has also been used in their uniform bounds for subvarieties of $E^g$ \cite{DP07}, see also their work on abelian varieties \cite{DP02}. We will need the following construction. We consider the map 
\[\Phi_m: \left(\bb{P}^1\right)^n \to \left(\bb{P}^1\right)^n \times \left(\bb{P}^1\right)^n\]
be given by $\left(z\right) \mapsto \left(z, \Phi^m\left(z\right)\right)$. We let $\left(\Phi,\Phi\right)$ act on  $\left(\bb{P}^1\right)^n \times \left(\bb{P}^1\right)^n$, so that we have a canonical height on $\left(\bb{P}^1\right)^n \times \left(\bb{P}^1\right)^n$ which we will denote again by $\h_{\Phi}$.

\begin{proposition} \label{Degrees1}
We have 
\[\deg\left(\Phi_m\left(V\right)\right) = \left(d^m + 1\right)^{\dim V} \deg\left(V\right), \quad  \h_{\Phi}\left(\Phi_m\left(V\right)\right) = \left(d^m + 1\right) \h_{\Phi}\left(V\right).\]    
\end{proposition}

\begin{proof}
Let $L_i=\pi_i^*L$ on $\left(\bb{P}^1\right)^n$ where $\pi_1,\pi_2$ are the two projection maps, so that $\deg(\Phi_m(V)) = \deg_{L_1 + L_2}(\Phi_m(V))$. As $\Phi_m: V \to \Phi_m(V)$ is an isomorphism and hence finite, by the projection formula, we have 
\[\deg_{L_1 + L_2}(\Phi_m(V)) = \deg_{\Phi_m^*(L_1) + \Phi_m^*(L_2)}(V).\]
Since $\Phi_m^*L_1 \simeq L$ and $\Phi_m^*L_2 \simeq L^{d^m}$, we have 
\[\deg_{(d^m+1)L}(\Phi_m(V)) = (d^m + 1)^{\dim V} \deg(V).\]
For $\h_{\Phi}$, recall that 
\[\h_{\Phi}(\Phi_m(V)) = \frac{\Phi_m(V) \cdot (\ovl{L}_1 + \ovl{L}_2)^{\dim V + 1}}{\deg(\Phi_m(V))}\]
where $\ovl{L}_i = \pi_i^* \ovl{L}_{\Phi}$ and $\ovl{L}_{\Phi}$ is the $\Phi$-invariant adelic line bundle satisfying $\Phi^*\ovl{L}_{\Phi} = d^m \ovl{L}_{\Phi}$. Then by the projection formula, the numerator is $(d^m + 1)^{\dim V + 1} \h_{\Phi}(\Phi_m(V))$ and so we obtain 
\[\h_{\Phi}(\Phi_m(V)) = (d^m + 1)\h_{\Phi}(V).\]
\end{proof}

We now need the following application of arithmetic B{\'e}zout \cite{BGS94}. For a point $x \in \bb{P}^1\left(\ovl{\bb{Q}}\right)$, we define its $\ell_2$-height as 
\[h_2\left(x\right) = \sum_{v \in M_K} n_v \log \left\lVert x\right\rVert _v\]
where $\left\lVert x\right\rVert _v = \max\left\{\left|x_0\right|_v,\left|x_1\right|_v\right\}$ for non-archimedean $v$ and is $\sqrt{\left|x_0\right|^2 + \left|x_1\right|^2}$ for archimedean $v$. We have the inequalities $h\left(x\right) \leq h_2\left(x\right) \leq h\left(x\right) + \frac{1}{2} \log 2$.

\begin{proposition} \label{ArithmeticBezout1}
Let $z_0 \in \bb{P}^1\left(\ovl{\bb{Q}}\right)$ and let $V \subseteq \left(\bb{P}^1\right)^k$ be an irreducible subvariety. Let $V_{z_0} = V \cap \left\{z_k = z_0\right\}$ where we represent a point on $\left(\bb{P}^1\right)^k$ as $\left(z_1,\ldots,z_k\right)$. Then 
\[\deg\left(V_{z_0}\right) \leq \deg V \text{ and } h_{\Ph}\left(V_{z_0}\right) \leq h_{\Ph}\left(V\right) + \left(\deg V\right) h_2\left(z_0\right)\]
where $h_{\Ph}$ denotes the Philippon height. 
\end{proposition}

\begin{proof}
The Segre embedding $\iota:\left(\bb{P}^1\right)^k\to\bb{P}^{2^k-1}$ corresponds to sending \[([x_1:y_1], [x_2:y_2], \ldots, [x_n,y_n]) \mapsto ( \prod_{i \in I} x_i \prod_{i \not \in I} y_i)_{I}\]
where we vary over all subsets $I$ of $\{1,2,\ldots,k\}$. By composing with an automorphism of $\bb{P}^1$, we may assume that $z_0 \not = \infty$. Then the hyperplane $\{z_k=z_0\}$ corresponds to $\{x_k = z_0 y_k\}$. Under the Segre embedding, if $J$ is any subset of $\{1,2,\ldots,k-1\}$< then we obtain the relation $X_{J \cup n} = z_0 X_J.$ Furthermore, if we intersect the image of $\left(\bb{P}^1\right)^k$ with $X_{J \cup \{k\}} = z_0 X_J$ over all $J$, we obtain back $x_k = z_0 y_k$. Hence $\iota(V_{z_0})$ can be viewed as $\iota(V)$ intersected with $2^{k-1}$ many hyperplanes in $\bb{P}^{2^k-1}$, each of height $h_2(z_0)$.
\par 
We now apply \cite[Proposition 3.3]{PW21}. Since $\dim V_{z_0} \geq \dim V - 1$, we obtain $\deg(V_{z_0}) \leq \deg V$ and $h_{\Ph}(V_{z_0}) \leq h_{\Ph}(V) + (\deg V)h(z_0)$ as desired.
\end{proof}

We are now ready to prove a dynamical height gap principle for non-preperiodic hypersurfaces $H \subseteq \left(\bb{P}^1\right)^n$. We define 
\[h_{\rat_d}\left(\Phi\right):= \sum_{i=1}^{n} h_{\rat_d}\left(f_i\right).\]

\begin{proposition} \label{DynamicalGap1}
Let $\Phi = \left(f_1,\ldots,f_n\right)$ be a split endomorphism of $\left(\bb{P}^1\right)^n$ over $\ovl{\bb{Q}}$ such that $\deg f_i = d \geq 2$ and each $f_i$ is non-exceptional. Let $H \subseteq \left(\bb{P}^1\right)^n$ be a hypersurface of degree $D$ that is not weakly $\Phi$-special. Then there exists $c = c\left(d,n,D\right) > 0$ such that 
\[\h_{\Phi}\left(H\right) \geq c \max\left\{h_{\rat_d}\left(\Phi\right),1\right\}.\]
\end{proposition}

\begin{proof}[{\blue Proof of Theorem~\ref{IntroTheorem2}}]
We proceed by induction on $n$. {\blue When $n=2$, at least one map is non-exceptional, so the two definitions of weakly special agree and Theorem~\ref{HeightGapCurve1} applies. Suppose that $n\geq3$. After rearranging coordinates, let $h_{\rat_d}(f_1)$ be maximal and choose the last coordinate so that $\Phi'=(f_1,\ldots,f_{n-1})$ still contains a non-exceptional map. This is possible since $n\geq3$, and we have}
{\blue \[
h_{\rat_d}(\Phi')\geq h_{\rat_d}(f_1)\geq\frac1n h_{\rat_d}(\Phi).
\]} Let $\pi_n: H \to \bb{P}^1$ be the projection to the last coordinate. For each $z \in \bb{P}^1\left(\ovl{\bb{Q}}\right)$, we let $H_z$ denote the fiber $\pi_n^{-1}\left(z\right)$. Let $\pi: \left(\bb{P}^1\right)^n \to \left(\bb{P}^1\right)^{n-1}$ be the projection to the first $n-1$ coordinates. Then $\pi\left(H_z\right)$ is a union of hypersurfaces in $\left(\bb{P}^1\right)^{n-1}$. Let $O_{f_n}^{-}\left(1\right)$ denote the backward orbit of $1$, i.e. $\cup_{m=1}^{\infty} \left(f_n\right)^{-m}\left(1\right)$. {\blue We call a hypersurface a {\em cylinder}  if it is the inverse image of a hypersurface under a projection forgetting one coordinate.}

\begin{lemma} \label{Hypersurface1}
There exists only finitely many $z \in O_{f_n}^{-}\left(1\right)$ so that every irreducible component of $\pi\left(H_z\right)$ is weakly $\Phi'$-special for $\Phi'=(f_1,\ldots,f_{n-1})$.    
\end{lemma}

\begin{proof}
First, if there are only infinitely many $z \in O_{f_n}^{-}\left(1\right)$ so that $\pi\left(H_z\right)$ is preperiodic, then each such $H_z$ has a Zariski dense set of points with canonical height $\leq \h_{f_n}(z)$. In particular, it follows that $H$ has a Zariski dense set of small points. {\blue Since $H$ is not a cylinder, it projects dominantly onto every collection of $n-1$ coordinates. By \cite[Theorem 2.2]{GNY18}, together with \cite{MSW23}, every $f_i$ would then be exceptional, contradicting our hypothesis.} 
\par 
Next, we assume that we have infinitely many $z$ such that $(\pi')^{-1}(W_z)$ is a component of $\pi\left(H_z\right)$ for some projection $\pi': (\bb{P}^1)^{n-1} \to (\bb{P}^1)^{n-2}$. Let $Z$ be this set of $z$'s. Then $W_z$ is a hypersurface in $\left(\bb{P}^1\right)^{n-2}$ and after rearranging coordinates, we may write $H_z = \bb{P}^1 \times W_z \times \{z\}$. Let $\psi:H\to\left(\bb{P}^1\right)^{n-1}$ be the projection from the first coordinate. Then for $x \in W_z \times \{z\}$, we have $\psi^{-1}(x) = \bb{P}^1 \times \{ x \}$. 
\par 
Let $W$ be the Zariski closure of $\cup_{z \in Z} (W_z \times \{z\})$. Then $W \subseteq \psi(H)$ with $\dim W\geq n-2$. Note that for $w \in W$, we must have $\psi^{-1}(w) = \bb{P}^1 \times \{ w\}$ since it is true for a dense set of points. Then $\dim \psi^{-1}(W) \geq n-1$ and in particular we must have $\psi^{-1}(W) = H$ which contradicts $H$ being non-weakly special. 
\par 
Finally if there exists infinitely many $z$ such that a component of $\pi\left(H_z\right)$ does not project dominantly to one of the coordinates, after rearranging coordinates there must be infinitely many such $z$ so that a component of $\pi\left(H_z\right)$ is equal to $\left(\bb{P}^1\right)^{n-2} \times \left\{ w_z\right\}$ and so is of the form $\pi^{-1}_{n-1}\left(w_z\right)$. This is covered by the previous case and so we are done here too.   
\end{proof}

Now let $Z$ be the subset of $O_{f_n}^-\left(1\right)$ so that all irreducible component of $\pi\left(H_z\right)$ are weakly $\Phi'$-special for $z \in Z$. Since $\pi\left(H_z\right)$ is not weakly $\Phi'$-special for $z \not \in Z$, we may apply induction. Let $\Phi'=(f_1,\ldots,f_{n-1})$ so that $h_{\rat_d}\left(\Phi'\right) \geq \frac{1}{n} h_{\rat_d}\left(\Phi\right)$ since $h_{\rat_d}(f_1)$ is maximal. Then by induction, we have a constant $c' = c'\left(d,n-1,D\right)$ so that 
\begin{equation} \label{eq: Induction1}
\h_{\Phi'}\left(\pi\left(H_z\right)\right) \geq c' \max\left\{h_{\rat_d}\left(\Phi'\right),1\right\} \geq \frac{c'}{n} \max\left\{h_{\rat_d}\left(\Phi\right),1\right\}.
\end{equation}

Note that $\h_{\Phi}\left(H_z\right) \geq \h_{\Phi'}\left(\pi\left(H_z\right)\right)$ and so we have a lower bound $\h_{\Phi}\left(H_z\right) \geq \frac{c'}{n} \max\left\{h_{\rat_d}\left(\Phi\right),1\right\}$ for $z \not \in Z$.
\par 
For $m \geq 1$, we now consider $\Phi_m\left(H\right) \cap F$ for $F =  \left\{z_{2n} = 1\right\}$ where we let $z_1,\ldots,z_{2n}$ be the coordinates for $\left(\bb{P}^1\right)^{2n}$. By Proposition \ref{ArithmeticBezout1}, we have 
\[h_{\Ph}\left(\Phi_m\left(H\right) \cap F\right) \leq h_{\Ph}\left(\Phi_m\left(H\right)\right) + \left(\deg \Phi_m\left(H\right)\right)\left(\log 2\right).\]
On the other hand, for $z$ satisfying $f_n^m\left(z\right) = 1$, we have that $\Phi_m\left(H_z\right)$ correspond to a finite union of irreducible components of $\Phi_m\left(H\right) \cap F$. Furthermore different $z$'s correspond to disjoint components. Hence we have 
\[\sum_{f_n^m\left(z\right) = 1} h\left(\Phi_m\left(H_z\right)\right) \leq h\left(\Phi_m\left(H\right) \cap F\right) \leq h\left(\Phi_m\left(H\right)\right) + \left(\deg \Phi_m\left(H\right)\right) \left( \log 2\right)\]
We now change the Philippon  heights to normalized canonical heights using \eqref{eq: Compare1}. First, note that 
\[\deg\left(\Phi_m\left(H_z\right)\right) = \left(d^m + 1\right)^{\dim H - 1} \deg\left(H_z\right) \geq \left(d^m + 1\right)^{\dim H - 1}\]
and 
\[\deg\left(h\left(\Phi_m\left(H\right)\right)\right) = \left(d^m + 1\right)^{\dim H} \left(\deg H\right).\]
Hence 
\[\frac{1}{\deg \Phi_m\left(H\right)}\sum_{f^m_n\left(z\right) = 1} h_{\Ph}\left(\Phi_m\left(H_z\right)\right) \geq \frac{1}{\left(d^m + 1\right) \deg H}\sum_{f^m_n\left(z\right) = 1} \h_{\Ph}\left(\Phi_m\left(H_z\right)\right)\]
and so we obtain 
\[\frac{1}{\left(d^m + 1\right) \deg H} \sum_{f^m_n\left(z\right) = 1} \h\left(\Phi_m\left(H_z\right)\right) \leq \h\left(\Phi_m\left(H\right)\right) + \log 2.\]
We may then change from $\h\left(\Phi_m\left(H\right)\right)$ to $\h_{\Phi}\left(\Phi_m\left(H\right)\right)$ and occur a constant error $c_1 = c_1\left(\Phi\right)$ which depends only on $\Phi$ by \eqref{eq: Compare1}. Hence 
\[\frac{1}{\left(d^m + 1\right) \deg H} \sum_{f^m_n\left(z\right) = 1} \h_{\Phi}\left(\Phi_m\left(H_z\right)\right) \leq \h_{\Phi}\left(\Phi_m\left(H\right)\right) + \left( c_1 + \log 2\right).\]
Now since $\h_{\Phi}\left(\Phi_m\left(H\right)\right) = \left(d^m + 1\right) \h_{\Phi}\left(H\right)$ and similarly for $H_z$, we get
\[\frac{1}{\left(d^m + 1\right) \deg H} \sum_{f^m_n\left(z\right) = 1} \h_{\Phi}\left(H_z\right) \leq \h_{\Phi}\left(H\right) + \frac{1}{d^m + 1}\left(c_1 + \log 2\right).\]
Since $\h_{\Phi}\left(H_z\right) \geq \frac{c'}{n} \max\left\{h_{\rat_d}\left(\Phi\right),1\right\}$  for all $z \not \in Z$, we obtain 
\[\frac{d^m - \left|Z\right|}{\left(d^m + 1\right) \deg H} \frac{c'}{n} \max\left\{h_{\rat_d}\left(\Phi\right),1\right\} \leq \h_{\Phi}\left(H\right) + \frac{1}{d^m + 1}\left(c_1 + \log 2\right).\]
As $\left|Z\right|$ is finite, taking $M \to \infty$ gives us $\h_{\Phi}\left(H\right) \geq \frac{c'}{n \deg H} \max\left\{h_{\rat_d}\left(\Phi\right),1\right\}$ which completes the induction.

\end{proof}

Applying Theorem \ref{QuantitativeZhang1}, we immediately deduce the following corollary. 

\begin{corollary} \label{UniBogomolov1}
Let $\Phi = \left(f_1,\ldots,f_n\right)$ be a split endomorphism of $\left(\bb{P}^1\right)^n$ over $\ovl{\bb{Q}}$ such that $\deg f_i = d \geq 2$ {\blue and at least one $f_i$ is non-exceptional}. Let $H \subseteq \left(\bb{P}^1\right)^n$ be {\blue an irreducible hypersurface of degree $D$} that is not weakly $\Phi$-special. Then there exists $c = c\left(d,n,D\right) > 0$ and $E = E\left(d,n,D\right)$ such that 
\[\left\{x \in H\left(\ovl{\bb{Q}}\right) \mid \h_{\Phi}\left(x\right) \leq c \max\left\{h_{\rat_d}\left(\Phi\right),1\right\} \right\}\]
is contained in a proper Zariski closed subvariety $V$ of degree at most $E$.
\end{corollary}

\begin{proof}
After conjugating $\Phi$, we may assume that $h\left(f_i\right) \leq M \max\left\{h_{\rat_d}\left(f_i\right),1\right\}$ for some constant $M > 0$ depending only on $d$. Then our corollary follows from applying \cite[Theorem 1.1]{MY25} along with {\blue Theorem \ref{IntroTheorem2}}.     
\end{proof}

We can now handle the general case of subvarieties $V \subseteq \left(\bb{P}^1\right)^n$.

\begin{proposition} \label{UniBogomolov2}
Let $\Phi = \left(f_1,\ldots,f_n\right)$ be a split endomorphism of $\left(\bb{P}^1\right)^n$ over $\ovl{\bb{Q}}$ such that $\deg f_i = d \geq 2$ {\blue and at most one $f_i$ is exceptional}. Let $V \subseteq \left(\bb{P}^1\right)^n$ be {\blue an irreducible subvariety of degree $D$} that is not preperiodic. Then there exists $E = E\left(d,n,D\right) > 0$ such that $\Prep\left(\Phi\right) \cap V\left(\ovl{\bb{Q}}\right)$ is contained in a proper Zariski closed subvariety $W$ of degree at most $E$.
\end{proposition}

\begin{proof}
We induct on $n$. Hence we may assume that $V$ projects dominantly to all its coordinates axes. Let $\dim V = r$. If $V$ is not preperiodic, then as explained in the proof of \cite[Proposition 2.1]{GNY18}, there exists $r+1$ many coordinates and a projection $\pi: \left(\bb{P}^1\right)^n \to \left(\bb{P}^1\right)^{r+1}$ such that the image $\pi\left(V\right)$ is a non-preperiodic hypersurface. Note that the degree of $\pi$ when restricted to $V$ is  bounded in terms of $D$. Then if $W'$ is a proper Zariski closed subvariety of $\pi\left(V\right)$ that contains all preperiodic points of $\pi\left(V\right)$ it follows that $W = \pi^{-1}(W')$ is a proper closed subvariety of $V$ that contains all preperiodic points of $V$. As $\deg W$ is bounded in terms of $\deg W'$ and $D$, it suffices to prove our Proposition for $\pi\left(V\right)$. 
\par 
After rearranging the coordinates, we may assume that $\pi$ is the projection to the first $r+1$ coordinates. Now if $\pi\left(V\right)$ is not weakly $\Phi'$-special for $\Phi'=(f_1,\ldots,f_{r+1})$, then Proposition \ref{UniBogomolov1} implies that we are done. Else $\pi\left(V\right)$ is weakly $\Phi$-special. Since it project dominantly to all coordinates, we are left with the case where $\pi(V) = \pi'^{-1}(V')$ for some projection $\pi': (\bb{P}^1)^{n-1} \to (\bb{P}^1)^{n-2}$. Then $\pi\left(V\right)$ is of the form $V' \times \bb{P}^1$. As $V'$ is not preperiodic, we are done by induction as desired.
\end{proof}

{\blue A simple induction gives uniform dynamical Manin--Mumford on $\left(\bb{P}^1\right)^n$ over $\ovl{\bb{Q}}$ when at most one coordinate map is exceptional.} 

\begin{corollary} \label{UniDMM1}
Let $\Phi = \left(f_1,\ldots,f_n\right)$ be a split endomorphism of $\left(\bb{P}^1\right)^n$ over $\ovl{\bb{Q}}$ such that $\deg f_i = d \geq 2$ {\blue and at most one $f_i$ is exceptional}. Let $V \subseteq \left(\bb{P}^1\right)^n$ be {\blue an irreducible subvariety of degree $D$}. Then there exists $N = N\left(d,n,D\right) > 0$ such that 
\[\ovl{\Prep\left(\Phi\right) \cap V\left(\ovl{\bb{Q}}\right)}^{\Zar} \text{ is a union of at most } N \text{ preperiodic irreducible subvarieties.}\]
\end{corollary}

\begin{proof}
This follows immediately by induction from Proposition \ref{UniBogomolov2}.    
\end{proof}

\section{Specialization from $\mathbb{C}$ to $\ovl{\bb{Q}}$} \label{sec: Specialization}
{\blue The main results of this section are Theorems~\ref{IntroTheorem1}, \ref{IntroUniformCurve1} and~\ref{IntroCommonPreperiodic1} over $\bb{C}$.} 

\begin{theorem} \label{UniFF1}
Let $k$ be an algebraically closed  field of finite transcendence degree and let $K = k\left(B\right)$ be the function field of a smooth projective curve over $B$. Assume that $k$ is of characteristic $0$. Let $\Phi = \left(f_1,\ldots,f_n\right)$ be a split endomorphism of $\left(\bb{P}^1\right)^n$ over $\ovl{K}$ such that $\deg f_i = d \geq 2$ {\blue and at most one $f_i$ is exceptional}. Let $V \subseteq \left(\bb{P}^1\right)^n$ be {\blue an irreducible subvariety of degree $D$}. Then there exists $N = N\left(d,n,D\right) > 0$ such that 
\[\ovl{\Prep\left(\Phi\right) \cap V\left(\ovl{K}\right)}^{\Zar} \text{ is a union of at most } N \text{ preperiodic irreducible subvarieties.}\]
\end{theorem}

To see why Theorem \ref{UniFF1} implies the situation over $\mathbb{C}$, we use the following proposition. 

\begin{proposition} \label{FunctionFieldCurve1}
Let $K$ be a {\blue finitely generated extension of $\ovl{\bb{Q}}$ of positive transcendence degree}. Then there exists $k \subset K$ such that $k$ is algebraically closed in $K$ and $K = k\left(B\right)$ for some geometrically irreducible projective curve $B$ over $k$.  
\end{proposition}

\begin{proof}
Let $\trdeg K = n$. Then we pick $n-1$ algebraically independent elements $x_1,\ldots,x_{n-1}$ and let $k = \ovl{\bb{Q}}(x_1,\ldots,x_{n-1})$. Replacing $k$ with its algebraic closure inside $K$, we obtain that $K$ is of transcendence degree $1$ over $k$. {\blue Since $K$ is finitely generated over $\ovl{\bb{Q}}$, it is finitely generated over $k$. Thus $K$ is a finite extension of $k(t)$ for any $t\in K$ transcendental over $k$.} It follows from \cite[Lemma 1.4.10]{BG06} that $K = k\left(B\right)$ for some irreducible projective curve $B$ over $k$. Since $k$ is algebraically closed in $K$ {\blue and the characteristic is zero, the extension $K/k$ is regular. Hence} $B$ is geometrically irreducible as desired.
\end{proof}

Our strategy to prove Theorem \ref{UniFF1} follows the deduction of the uniform geometric Bogomolov conjecture for characteristic $0$ from the case of number fields by Mavraki and the second author \cite{MY25}. If $\Phi = \left(f_1,\ldots,f_n\right)$ is an endomorphism over $\mathbb{C}$ and $V$ a subvariety of $\left(\bb{P}^1\right)^n$, then we may find some $K$ so that $K$ is finitely generated over $\ovl{\bb{Q}}$ for which $\Phi$ and $V$ are well-defined over. Using Proposition \ref{FunctionFieldCurve1}, we may write $K = k\left(B\right)$ for some smooth projective curve $B$ over $k$ with $k$ algebraically closed in $K$. We can extend by scalars to replace $k$ with $\ovl{k}$ to assume that $k$ is algebraically closed. Then Theorem \ref{UniFF1} implies the number of irreducible components of $\ovl{\Prep(\Phi) \cap V(\bb{C})}^{\Zar}$ is bounded above by $N$ as desired. 
\par 
To prove Theorem \ref{UniFF1}, it suffices to prove the analogue of Proposition \ref{UniBogomolov2} over $K = k\left(B\right)$. This will in turn follow from the analogue of {\blue Theorem \ref{IntroTheorem2}}, which we will state now. Note that as we have fixed ample height functions $h_{\Rat_d}$ and $h_{\rat_d}$, these naturally extend to height functions on $K$-points for all characteristic zero function fields. 

\begin{proposition} \label{DynamicalGapFF1}
Let $k$ be an algebraically closed field of finite transcendence degree and let $K = k\left(B\right)$ be the function field of a smooth projective curve over $B$. Assume that $k$ is of characteristic $0$. Let $\Phi = \left(f_1,\ldots,f_n\right)$ be a split endomorphism of $\left(\bb{P}^1\right)^n$ over $\ovl{K}$ such that $\deg f_i = d \geq 2$ {\blue and at least one $f_i$ is non-exceptional}. Let $H \subseteq \left(\bb{P}^1\right)^n$ be {\blue an irreducible hypersurface of degree $D$} that is not weakly $\Phi$-special. Then there exists $c = c\left(d,n,D\right) > 0$ such that 
\[\h_{\Phi}\left(H\right) \geq c h_{\rat_d}\left(\Phi\right).\]
\end{proposition}

\begin{proof}
We will proceed by induction on $\trdeg K$. Let's first handle the case where $k=\ovl{\bb{Q}}$. We spread out $H$ as a flat family of hypersurfaces $\cal{H}_t$ over $(\bb{P}^1)^n \times U$. We may also assume that $\Phi$ spreads out as a family of endomorphisms $\Phi_t$ over $(\bb{P}^1)^n \times U$. {\blue Shrinking $U$, we may also assume that every $f_{i,t}$ has degree $d$ and that $f_{i,t}$ is non-exceptional whenever $f_i$ is non-exceptional.} We now apply Silverman's specialization theorem  \cite[Proposition 6.2]{MY25}. The statement in \cite{MY25} only applies to abelian varieties but the proof handles the dynamical setting too. We thus obtain 
\begin{equation} \label{eq: Specialization1}
\lim_{\substack{t \in U(\ovl{\bb{Q}}) \\ h(t) \to \infty }} \frac{\h_{\Phi_t}(\cal{H}_t)}{h(t)} = \h_{\Phi}(H).
\end{equation}
Now we may assume that not all $f_i$ are isotrivial as else $\h_{\rat_d}\left(\Phi\right) = 0$ and there is nothing to prove. By Lemma \ref{WeaklySpecialSpecialization1}, as $H$ is not weakly $\Phi$-special, for some constant $C > 0$, there are only finitely many $t\in U(\ovl{\bb{Q}})$ such that $\cal{H}_t$ is weakly $\Phi$-special and $h(t) > C$. For $\cal{H}_t$ not weakly special, we obtain 
\[\h_{\Phi_t}\left(\cal{H}_t\right) \geq c h_{\rat_d}\left(\Phi_t\right)\]
for $c > 0$ depending only on $\left(d,n,D\right)$. By Silverman's specialization again, we know that 
\[\lim_{\substack{t \in U(\ovl{\bb{Q}}) \\ h(t) \to \infty}} \frac{h_{\rat_d}(\Phi_t)}{h(t)} = h_{\rat_d}(\Phi)\]
and so \eqref{eq: Specialization1} implies that 
\[\h_{\Phi}\left(H\right) \geq c h_{\rat_d}\left(\Phi\right)\]
as desired. For the general case, we first replace $k$ with a finitely generated extension over $\ovl{\bb{Q}}$ for which $\Phi$ and $H$ are defined over. By Proposition \ref{FunctionFieldCurve1}, we may then write $k = k'\left(B'\right)$ for $k'$ algebraically closed in $k$ and $B'$ a smooth projective curve over $k'$. After extending scalars, we may assume that $k'$ is algebraically closed. Since $\trdeg\left(k\right) < \trdeg\left(K\right)$, by induction we know that Proposition \ref{DynamicalGapFF1} holds for $k = k'\left(B'\right)$. Then the same argument as in the case of $K = \ovl{\bb{Q}}\left(B\right)$ implies that Proposition \ref{DynamicalGapFF1} for $k$ implies Proposition \ref{DynamicalGapFF1} for $K = k\left(B\right)$ since \cite[Proposition 6.2]{MY25} works for any $K = k\left(B\right)$ where $B$ is a curve over $k$ and $k$ is any function field.
\end{proof}

\begin{lemma} \label{WeaklySpecialSpecialization1}
Assume that not all $f_i$ are isotrivial. If $H$ is not weakly $\Phi$-special, there does not exist a sequence $t_n \in U\left(k\right)$ with $h(t_n) \to \infty$ such that $\cal{H}_{t_n}$ is weakly $\Phi$-special. 
\end{lemma}

\begin{proof}
Since $H$ projects dominantly to each coordinate, clearly $\cal{H}_t$ does so for $t$ in a Zariski open $V$ of $U$. Now if there is a projection $\pi: \left(\bb{P}^1\right)^n \to \left(\bb{P}^1\right)^{n-1}$ so that $\cal{H}_t = \pi^{-1}(W_t)$ for infinitely many $t$, then viewing $\cal{H}$ as a hypersurface inside $(\bb{P}^1)^n \times B$, we have $\cal{H}_t = \bb{P}^1 \times W_t \times \{ t \}$. Let $T$ be the set of such $t$'s and let $W$ be the Zariski closure of $\cup_{t \in T} \cal{H}_t$. Then $W \subseteq \cal{H}$ and $\dim W \geq n$. As we may assume that $\cal{H}$ is irreducible after shrinking $U$, it follows that $W = \cal{H}$. Then clearly $H$ is of the form $\bb{P}^1 \times W'$ for some $W'$ which contradicts $H$ being not weakly $\Phi$-special. 
\par 
{\blue If $n>2$, a preperiodic fiber $\cal{H}_t$ would have a Zariski dense set of preperiodic points. Since it projects dominantly onto every collection of $n-1$ coordinates, \cite[Theorem 2.2]{GNY18}, together with \cite{MSW23}, would force every $f_{i,t}$ to be exceptional. This contradicts the presence of a non-exceptional coordinate. Hence none of these fibers is weakly special.} If $n = 2$, then we will show that if $\cal{H}_t$ is preperiodic for $\Phi_t$ for infinitely many $t$, then $H$ itself is preperiodic for $\Phi$. {\blue Each $\cal{H}_{t_j}$ must then be $\Phi_{t_j}$-preperiodic, since at least one specialized map is non-exceptional and both projections of the curve are dominant.} By the analogue of \eqref{eq: Specialization1}, we obtain that $\h_{\Phi}(H) = 0$. We may now apply \cite[Theorems 4.1 and 4.3]{MS25}. Both these theorems work over general function fields of a curve. Since not both $f_1,f_2$ are isotrivial, we must have that $H$ is $\Phi$-preperiodic which contradicts $H$ being not weakly $\Phi$-special as desired.
\end{proof}

We can now deduce the analogue of Proposition \ref{UniBogomolov2}. 

\begin{proposition} \label{UniBogomolovFF1}
Let $k$ be an algebraically closed field of finite transcendence degree and let $K = k\left(B\right)$ be the function field of a smooth projective curve over $B$. Let $\Phi = \left(f_1,\ldots,f_n\right)$ be a split endomorphism of $\left(\bb{P}^1\right)^n$ over $\ovl{K}$ such that $\deg f_i = d \geq 2$ {\blue and at most one $f_i$ is exceptional}. Let $V \subseteq \left(\bb{P}^1\right)^n$ be {\blue an irreducible subvariety of degree $D$} that is not preperiodic. Then there exists $E = E\left(d,n,D\right) > 0$ such that $\Prep\left(\Phi\right) \cap V\left(\ovl{K}\right)$ is contained in a proper Zariski closed subvariety $W$ of degree at most $E$.
\end{proposition}

\begin{proof}
We again proceed by induction on $n$ and on $\trdeg K$. {\blue The hypothesis that at most one coordinate map is exceptional is preserved under coordinate projections.} Hence we may assume that $V$ projects dominantly to all coordinates. As in Proposition \ref{UniBogomolov2}, since $V$ is not preperiodic, we may find a projection $\pi: \left(\bb{P}^1\right)^n \to \left(\bb{P}^1\right)^{r+1}$ where $r=\dim V$ such that the image $\pi\left(V\right)$ is a non-preperiodic hypersurface. After rearranging, we may assume that this is the first $r+1$ coordinates. Let's first say that $\left(f_1,\ldots,f_{r+1}\right)$ are not all isotrivial so that $h_{\rat_d}(\Phi) > 0$. Then if $\pi\left(V\right)$ is not weakly $\Phi$-special, applying Proposition \ref{DynamicalGapFF1} along with \cite[Theorem 1.1]{MY25} gives us the existence of a proper Zariski closed subvariety $W$ containing all points of canonical height $0$ with $\deg W$ bounded above by $E = E\left(d,n,D\right) > 0$. 
If $\pi\left(V\right)$ is weakly $\Phi$-special, then it is of the form $\bb{P}^1 \times W$ for some $W \subseteq (\bb{P}^1)^r$ after rearranging coordinates. We are then done by induction again. 
\par 
Now if $f_1,\ldots,f_{r+1}$ are all isotrivial, we may conjugate them so that each $f_i$ are defined over $k$. Then $\h_{\Phi'}$ for $\Phi'=(f_1,\ldots,f_{r+1})$ is simply the standard Weil height $h$. If $h\left(\pi\left(V\right)\right) > 0$, then applying \cite[Theorem 1.1]{MY25} again gives us our $W$ and $E$. Else if $h\left(\pi\left(V\right)\right) = 0$, then $\pi\left(V\right)$ must itself must be defined over $k$. Writing $k = k'\left(B'\right)$ with $B'$ smooth projective over $k'$ and $k'$ algebraically closed in $k$, we obtain our $W$ and $E$ by induction as $\trdeg k < \trdeg K$. This finishes our proof.   
\end{proof}

{\blue We recall the definition of special curves from \cite{MS25}.

\begin{definition}\label{def:SpecialCurves}
Let $f,g$ be rational maps of degree $d\geq2$ and let $C\subseteq\left(\bb{P}^1\right)^2$ be an irreducible curve over an algebraically closed field of characteristic $0$. We say that $C$ is $\left(f,g\right)$-special if it is $\left(f,g\right)$-preperiodic or if case~\textup{(3)} of Definition~\ref{def:WeaklySpecialCurves} holds with $H$ a torsion coset, i.e. a translate of a connected algebraic subgroup by a torsion point.
\end{definition}

\medskip

\proof[Proof of Theorem~\ref{IntroUniformCurve1}]
Suppose that $C$ is not $(f,g)$-special. If at least one of $f,g$ is non-exceptional, we can then apply Theorem~\ref{IntroTheorem1} with $n=2$.   A horizontal or vertical non-special curve contains no preperiodic points. We may therefore assume that both projections are dominant and both maps are exceptional.

Choose standard group presentations $\vphi_i:G_i\to\bb{P}^1$ of degrees $k_i\leq6$: the degrees are $1$ for power maps, $2$ for Chebyshev maps, and $2,3,4$ or $6$ for Latt{\`e}s maps; see \cite[Section 2 and Theorem 3.1]{Mil06}. Let $G:=G_1\times G_2$, $\vphi=(\vphi_1,\vphi_2)$ and $Y=\vphi^{-1}(C)$. The degree of $Y$ (in the standard compactification of $G$) is uniformly bounded by a number depending only on $D$.

Applying \cite[Theorem 1.1]{HLY26} to $Y$ and the rank-zero torsion subgroup $G_{\mathrm{tors}}$, we can bound $\#(Y\cap G_{\mathrm{tors}})$ only  in terms of $D$.

Every preperiodic point in $C\cap\operatorname{Im}\vphi$ lifts to a torsion point of $G$. Each $\vphi_i$ omits at most two points of $\bb{P}^1$, so $C\setminus\operatorname{Im}\vphi$ contains at most $2D$ points. Combining these bounds proves the theorem.
\endproof

\medskip

\proof[Proof of Theorem~\ref{IntroCommonPreperiodic1}]
Let $\Delta\subseteq(\bb{P}^1)^2$ be the diagonal.  If $\Prep(f)\neq\Prep(g)$, then \cite[Theorem 1.2]{BD11} implies that $\Prep(f)\cap\Prep(g)$ is finite. Thus $\Delta$ is not $(f,g)$-special.   Applying Theorem~\ref{IntroUniformCurve1} to $\Delta$, we have
\[
\left|\Prep(f)\cap\Prep(g)\right|
=\left|\Delta\cap\Prep(f,g)\right|
\leq B(d,2),
\]
which finishes the proof.
\endproof}

\begingroup
\emergencystretch=3em
\printbibliography
\endgroup
\end{document}